\documentclass[a4paper, 12pt]{amsart}
\usepackage{amssymb,amsmath,txfonts,mathrsfs,titletoc,appendix}
\usepackage{graphicx}
\usepackage{latexsym}
\usepackage[alphabetic]{amsrefs}
\usepackage{geometry}
\usepackage{fancyhdr}
\usepackage{xcolor}
\usepackage{comment}
\usepackage{subcaption}
\usepackage{pifont} 

\usepackage{exscale}        
\usepackage{relsize}       

\usepackage{tensor}   

\usepackage{float} 

\usepackage[new]{old-arrows}
\usepackage{longtable}
\usepackage{booktabs}
\usepackage{array}
\usepackage{multirow}
\usepackage{supertabular}
\usepackage{amssymb}
\usepackage{amsmath}
\usepackage{mathrsfs}
\usepackage{titletoc}
\usepackage{latexsym}
\usepackage[alphabetic]{amsrefs}

 \newtheorem{thm}{Theorem}[section]
 \newtheorem{cor}[thm]{Corollary}
 \newtheorem{lem}[thm]{Lemma}
 \newtheorem{prop}[thm]{Proposition}
 \newtheorem{defn}[thm]{Definition}
 \newtheorem{rem}[thm]{Remark}
 \numberwithin{equation}{section}
\newtheorem{lem*}{Lemma}
\newtheorem{cor*}{Corollary}

\newenvironment{pf0}{\medskip \noindent
{\it Proof of Theorem \ref{mL}.}}{\hfill $\square$\par
} 
\newenvironment{pf}{\medskip \noindent
{\bf Proof of Theorem \ref{mainthm1}.}}{\hfill \rule{.5em}{1em}
\\}

\newenvironment{rmk}{\mbox{ }\\{\bf  Remark}\mbox{ }}{
\hfill $\Box$\mbox{}\bigskip}

\def\p#1{\partial #1}

\def\R{\Bbb{R}}
\def\C{\Bbb{C}}
\def\N{\Bbb{N}}

\begin{document}

\title{
Spiral Minimal Products}
\dedicatory{Dedicated to Prof. F. Reese Harvey and Prof. H. Blaine Lawson, Jr.  with admiration}
\author{Haizhong Li}
\address{Department of Mathematical Sciences, Tsinghua University, Beijing 100084, P. R. China}
\email{lihz@mail.tsinghua.edu.cn}
\author{Yongsheng Zhang}
\address{Academy for Multidisciplinary Studies, Capital Normal University, Beijing 100048, P. R. China}
\email{yongsheng.chang@gmail.com}
\date{\today}

\begin{abstract}
This paper exhibits  a structural strategy to produce new minimal submanifolds in spheres based on two given ones.
The method is to spin a pair of given minimal submanifolds by a curve $\gamma\subset \mathbb S^3$ in a balanced way 
and leads to resulting minimal submanifolds $-$ spiral minimal products, 
which form a two-dimensional family arising from intriguing pendulum phenomena decided by $C$ and $\tilde C$.
With $C=0$, we generalize the construction of minimal tori in $\mathbb S^3$ explained in \cite{Brendle} to higher dimensional situations.
When $C=-1$,  we recapture previous relative work in \cite{CLU} and \cite{HK1} for special Legendrian submanifolds 
in spheres,
and moreover, can gain numerous  
$\mathscr C$-totally real and
totally real embedded minimal submanifolds in spheres and in complex projective spaces respectively.
A key ingredient of the paper is to apply a beautiful extension result of minimal submanifolds by Harvey and Lawson \cite{HL0} 
for a rotational reflection principle in our situation to establish curve $\gamma$.
\end{abstract}
\maketitle
\section{Introduction}\label{P} 
By a submanifold in sphere, we mean an immersion of a manifold into a unit Euclidean sphere.
                       When the immersion has vanishing mean curvature vector field, 
                       we say that it is a spherical minimal submanifold. 
                       The study of spherical minimal submanifolds plays  a core role in differential geometry, e.g.
                       \cite{L, HL2, HL1, Mc,  CM, Haskins, CLU, HK0, Joyce, HK1, Brendle, z, XYZ2, x-y-z0, TZ} and etc.

                  In order to understand more about structures among spherical minimal submanifolds,
                   one may wonder
                       whether there
                        are some non-trivial 
                        algorithms to manufacture new examples based on given ones
                        and how useful that could be?
                       There is a relatively simple one already known for a while from different perspectives, 
                       for instance see \cite{X,TZ, CH}.
            Let us explain it a bit.          
                       Denote $S^{n}(1)\subset \R^{n+1}$ by $\mathbb S^n$.
                       Given two spherical minimal
                       submanifolds  
                       $f_1:M_1^{k_1}\longrightarrow \mathbb S^{n_1}$ and 
                       $f_2:M_2^{k_2}\longrightarrow \mathbb S^{n_2}$,
                       the (unique up to coefficient signs) {\it constant minimal product} of $f_1$ and $f_2$
                      is the minimal immersion from                     $M_1\times M_2$ into $ \mathbb S^{n_1+n_2+1}$ by
         \[
         (x,y)\longmapsto \Big(\lambda f_1(x), \mu f_2(y)\Big)\,  
         \text{  \    \       with\ \ } 
         \lambda=\sqrt{\frac{k_1}{{k_1+k_2}}\ }\ \text{\ \ and }\ \mu=\sqrt{\frac{k_2}{{k_1+k_2}}} .
         \]
        Such an algorithm stimulates several new discoveries  on spherical minimal submanifolds and area-minimizing cones \cite{TZ, JCX, WW, JXC}.


                 In this paper we 
                 provide
                 a new structural  strategy  to generate minimal objects of one dimension higher 
                 compared to the constant minimal product 
                 and the strategy can further enrich the realm of spherical minimal submanifolds.
                 The construction                 is quite universal and can take any pair of 
                    {\it $\mathscr C$-totally real} 
                    spherical minimal immersions as inputs.

\begin{defn}\label{wiso}
               An immersion $f$ of $M$ into $\mathbb S^{2n+1}\subset \C^{n+1}$ 
               is called $\mathscr C$-totally real
               if the vector field $\big\{i\cdot f(x)\, \vert\, x\in M\big\}$ along the immersion is a normal vector field.
\end{defn}
\begin{rem}
In $\mathbb S^{2n+1}$, $\mathscr C$-totally real submanifolds of dimension $n$ are called Legendrian submanifolds.
\end{rem}
                 Let $\gamma(t)=\big(\gamma_1(t),\gamma_2(t)\big)
                \in
                 \mathbb S^3\subset \mathbb C^2$
                 represent an immersed curve over domain $\R$.
                 When $n_1+1=2m_1+2$ and $n_2+1=2m_2+2$,
                 we identify $\mathbb R^{n_1+1}\cong\mathbb C^{m_1+1}$ and $\mathbb R^{n_2+1}\cong\mathbb C^{m_2+1}$.
\begin{defn}   
          The $\gamma$-product of spherical immersions $f_1$ and $f_2$ is the map
   \begin{eqnarray}
 \ \ \ \ \ \  \ \ \  G_\gamma: \mathbb R\times M_1\times M_2 \longrightarrow \mathbb S^{n_1+n_2+1}\ \ by \ \ 
                               \big(t,x,y\big)\longmapsto \Big(\gamma_1(t)f_1(x),\, \gamma_2(t)f_2(y)\Big) .
                               \label{expression}
   \end{eqnarray} 
   Here the multiplication 
   is the complex multiplication.
   When $G_\gamma$ is a minimal immersion,
   we call it a spiral minimal product.
\end{defn}

                    We shall assume $\gamma$ to be an immersion throughout the paper.
          Based on types of $\gamma$-spin, we introduce some concepts of $G_\gamma$.
          Let $a(t)=|\gamma_1(t)|$ and  $b(t)=|\gamma_2(t)|$. 
                         \begin{defn}
                  If $a(t)$ and  $b(t)$ remain unchanged, 
                $G_\gamma$ is called of steady magnitudes,
                 and otherwise of varying magnitudes.
                If
                 neither (or exact one) argument of 
                 $\gamma_1(t)$ and $\gamma_2(t)$ 
                 is constant,
        $G_\gamma$ is called a doubly (or singly) spiral product.
               %
                 \end{defn}
                  
                 The first main result of our paper is the following. 
    \begin{thm}\label{main}
    Suppose that 
          $f_1:M_1^{k_1}\longrightarrow \mathbb S^{n_1}$ and 
                       $f_2:M_2^{k_2}\longrightarrow \mathbb S^{n_2}$
                       are $\mathscr C$-totally real 
    minimal  immersions with $k_1+k_2\geq 1$.
      Then, for either type among 
      \\
      \ding{172} doubly spiral minimal product of steady magnitudes, 
      \\
                               \ding{173} 
                               doubly spiral minimal product of varying magnitudes,
                        \\
                        \ding{174} singly spiral minimal product of varying magnitudes,
                        \\
                  there are uncountably many choices of $\gamma$ 
                  to generate $ G_\gamma$
        of that type.
                  
      \end{thm}
                  It turns out that spiral minimal products of  Type \ding{172}
                  come from limits of those of Type \ding{173}
                  and that
                  the determination of $\gamma$ for \ding{173} and \ding{174}
                 is completely dominated by amusing pendulum phenomena.
                  Let us draw a blueprint for Type \ding{174} which is allowed to produce immersed minimal hypersurfaces. 
                 To seek for a singly spiral minimal product $G_\gamma$
                  with varying $\gamma_2(t)=b(t)$ (i.e., spin occurs only in the first complex slot of $\gamma$), 
                          we shall use an arc parameter $s$ for $\big(a(t), b(t)\big)$ and consider the following function (a special case of \eqref{Cfun})
                          \begin{equation}\label{P(s)}
                        P_0(s)= \dfrac{1}{ \Big(\cos s\Big)^{2k_1+2}
         \Big( \sin s\Big)^{2k_2}}
         \ \ \ \ \text{ for }s\in\left(0,\frac{\pi}{2}\right).
\end{equation}
%
%
\begin{figure}[h]
	\centering
	\begin{subfigure}[t]{0.48\textwidth}
		\centering
		\includegraphics[scale=0.38]{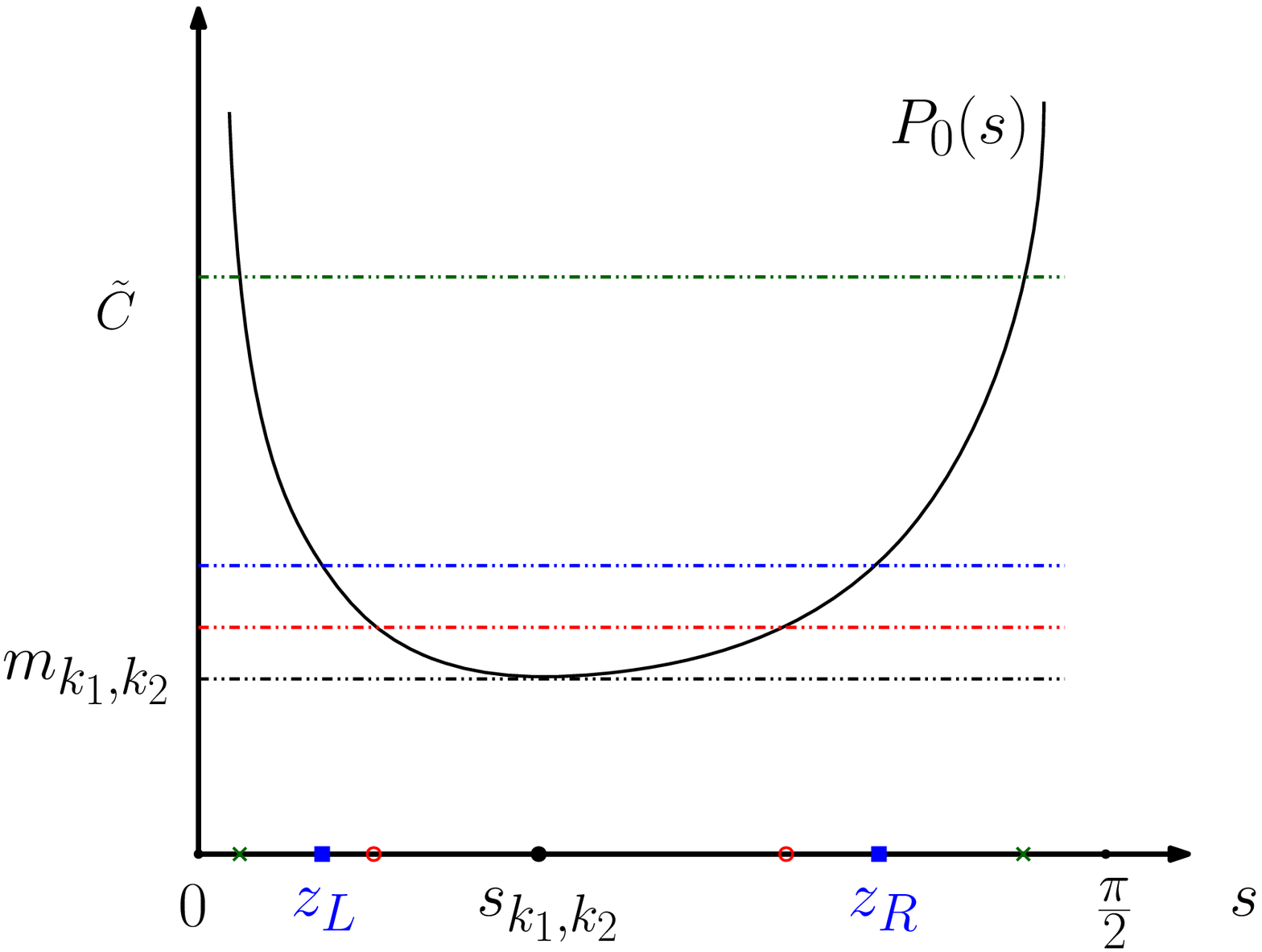}
                              \captionsetup{font={scriptsize}} 
                               \caption{Range of pendulum decided by $\tilde C$}
                               \label{fig:1a}
	\end{subfigure}
	\begin{subfigure}[t]{0.48\textwidth}
		\centering
	 \includegraphics[scale=0.49]{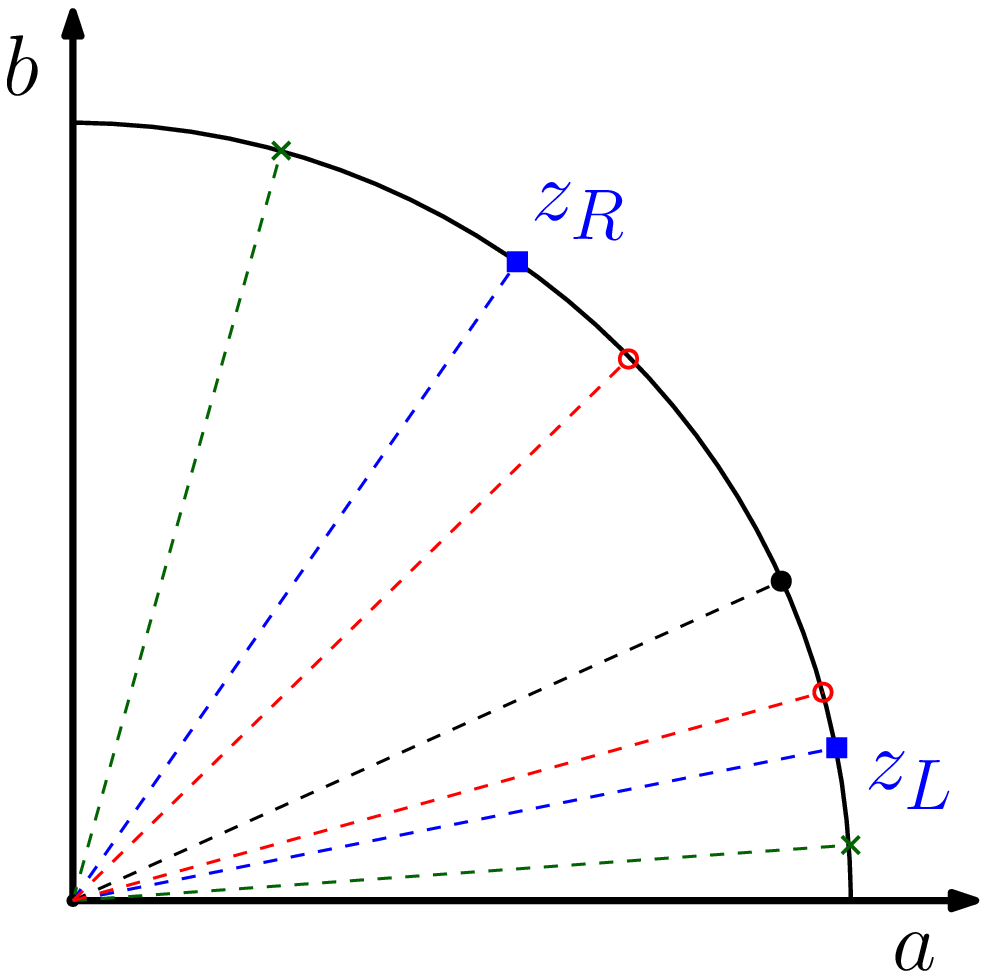}
                     \captionsetup{font={scriptsize}} 
                               \caption{Appearance of the pendulum in $\big(a(t), b(t)\big)$}\label{fig:1b}
	\end{subfigure}
	\caption{Pendulum phenomenon for singly spiral minimal products} 
\end{figure}
                   
                     It is obvious that, when $k_2>0$, function $P_0(s)$ has a unique interior critical point $s_0=s_{k_1,k_2}$
                  of value
                  $m_{0}=m_{k_1, k_2}=P_0(s_{k_1,k_2})$.
                  Then, for any $\tilde C>m_{0}$,
                  there 
                  are 
                   two solutions $z_L$ and $z_R$ for $P_0(s)=\tilde C$
                  as shown in Figure \ref{fig:1a}
                  and 
                  the pendulum
                  of
                  $\big(a(t), b(t)\big)$
                  moves
                  back and
                  forth
                  between angles $z_L$ and $z_R$ in Figure \ref{fig:1b}.
                  As $\tilde C\downarrow m_{0}$ in Figure \ref{fig:1a},
                  the amplitude of the pendulum shrinks to zero, i.e., limiting to steady magnitudes
                  with $a(t)\equiv \cos s_{0}$ and $b(t)\equiv \sin s_{0}$.
%
%
           For doubly spiral products, 
                  the same pendulum phenomena  
                  also happen but with a bit more complicated 
                  apparences
               characterized by
           $C$ and $\tilde C$   
                 (to be explained)
                  in \eqref{circledno}, \eqref{Gexp} and \eqref{dotsss}.
                  By studying \eqref{dotsss}
                  we obtain the following.
\begin{thm}\label{main0}
                   %
                   Additionally assume that $M^{k_1}_1$ and $M^{k_2}_2$  in Theorem \ref{main} are closed manifolds. 
                   Then there are $\mathbb Q \times \mathbb Q$ 
                   many                    doubly spiral minimal products
                    of varying magnitudes  
                    each of which factors through a minimal immersion
                   of some \textbf{closed} manifold.
\end{thm}

    When  constructing singly spiral minimal products, there is no need for the non-spiral part to be $\mathscr C$-totally real.
    As a useful application, a bunch of  generalizations of 
    the construction of minimal tori in $\mathbb S^3$ in \cite{Brendle} 
    (see \eqref{BK} in below) can be gained
    in Corollary \ref{corpm0}.
    The following is a special case of Corollary \ref{corpm0}.
    \begin{cor}\label{S1xSn}
                   For either odd or even $n\geq 1$,
                   there are infinitely  many singly spiral minimal products
                    of varying $\gamma_2(t)=b(t)$ 
                    for $f_1(point_1)=(1,0)\in \mathbb S^1$
                    and $f_2=id_{\mathbb S^{n}}$                     such that 
                           each can  induce a minimal immersion from 
                     $S^1\times S^n$
                    into $\mathbb S^{n+2}$.
\end{cor}

                    We remark that
                    formulation \eqref{expression} was first introduced
                               in \cite{CLU} to construct special Legendrian submanifolds 
                              based on special Legendrian $f_1$ and $f_2$,
                               and further
                               developed in  \cite{HK1}.
 Actually, $C$ in our construction means the ratio of $\gamma$'s signed angular momenta, see Remark \ref{AM}.
%
                       We find
                        that 
                        the value of $C$  in \eqref{circledno}
                        completely determines
                        whether $G_\gamma$ of two $\mathscr C$-totally real immersions
                         is also $\mathscr C$-totally real or not.
\begin{lem}\label{Lgn}
                Suppose that spherical immersions $f_1$ and $f_2$ are $\mathscr C$-totally real.
                              Then a spiral product $G_\gamma$ 
                                  of them
                              is $\mathscr C$-totally real
                                 if and only if 
                                      $C=-1$.
\end{lem}

        By fixing $C=-1$, a corollary of Theorems \ref{main} and \ref{main0} is the following.

           \begin{cor}\label{main2}
    Let
          $f_1:M_1^{k_1}\longrightarrow \mathbb S^{n_1}$ and 
                       $f_2:M_2^{k_2}\longrightarrow \mathbb S^{n_2}$
                       be $\mathscr C$-totally real
    minimal  immersions of $k_1+k_2\geq 1$.
      Then there are uncountably many          $\mathscr C$-totally real                     doubly spiral minimal products of varying magnitudes with $C=-1$.
      When  $M_1$ and $M_2$ are closed, 
      infinitely many among these products
        can descend to  minimal immersions
                    from some closed manifolds
                    into  $\mathbb S^{n_1+n_2+1}$,
                    and moreover,
in conjunction  with the Hopf projection, each  leads to  a totally real
         \footnote{Here totally real means that the tangent space at each point of the immersion
         is mapped into its orthogonal space by the complex structure action of the ambient K\"ahler manifold.
         A totally real submanifold is called Lagrangian if it has exactly half dimension of the ambient manifold.}
minimal immersion into $\mathbb C P^{\frac{n_1+n_2}{2}}$.
    %
    \end{cor}
                     \begin{rem}
    When $f_1$ and $f_2$ are special Legendrian immersions,
    the construction recaptures relative work in \cite{CLU} and \cite{HK1}.
    \end{rem}
          
         Since embedded submanifolds are extremely important, 
         we emphasize the following.
        Given embedded submanifolds
         $M_1 \subset\mathbb S^{r_1}\subset \R^{r_1+1}$ and $M_2 \subset\mathbb S^{r_2}\subset \R^{r_2+1}$,
        by complexifications
 we have
                       \begin{equation}\label{Cpx1}
                               M_1 \times \big\{0\big\}
                                           \hookrightarrow
                                                  \mathbb S^{r_1}\times \big\{0\big\}
                                                            \subset \C^{r_1+1}=\R^{r_1+1}\oplus\,  \R^{r_1+1}
                      \end{equation}
                          \begin{equation}\label{Cpx2}
                          M_2 \times \big\{0\big\}
                                           \hookrightarrow
                                                     \mathbb S^{r_2}\times \big\{0\big\}
                                                               \subset \C^{r_2+1}=\R^{r_2+1}\oplus \,\R^{r_2+1}
                                  \end{equation}
                  $\mathscr C$-totally real with respect to the $i$-multiplication: $i\cdot (v, w)=(-w, v)$.
\begin{thm}\label{mainebd}
  Let $M_1\subset \mathbb S^{r_1}$ and  $M_2\subset \mathbb S^{r_2}$ be  embedded minimal submanifolds.
          Then, by virtue of \eqref{Cpx1} and \eqref{Cpx2},
           there
           {exist} infinitely many 
                   spiral minimal products 
                   with $C=-1$,
 %
               {each of which} factors through
                    a $\mathscr C$-totally real minimal
                    embedding 
                    of some suitable quotient 
                    of $S^1\times M_1\times M_2$
                    into  $\mathbb S^{2r_1+2r_2+3}$.
                    Moreover, 
                   joint with the Hopf projection,
                  %
               infinitely many  embedded totally real minimal submanifolds
                    in $\mathbb C P^{\,r_1+r_2+1}$ can be gained.
\end{thm}
        %
%

  
%
  
                  Another situation where the spiral minimal products show the power is to
                  generate embedded minimal Lagrangians
                   in complex projective spaces.    
                   
                   \begin{thm}\label{mL}
                   Given two connected embedded minimal Lagrangians
                   $$
                   u_1': M_1'\longrightarrow \mathbb CP^{m_1} 
                   \text{\ \ \ \ \  and \ \ \ \ \ }
                    u_2':M_2'\longrightarrow \mathbb CP^{m_2} 
                   $$
                   with $m_1, m_2\geq 1$.
                   Then there exist $\mathbb Q\times S^1$
                   many  connected embedded minimal Lagrangians in $\mathbb CP^{m_1+m_2+1}$
                   based on $u_1'$, $u_2'$ and related spiral minimal product procedure.
                   \end{thm}
                   The paper is organized as follows.
                   Due to different natures,
                   we
                   demonstrate  basic ideas 
               for global construction 
                  of spiral minimal products of steady magnitudes
                   in \S \ref{SMP},
                  while 
                  in
                  \S \ref{GSMP} 
                             we   derive local      construction of  spiral minimal products of varying magnitudes.
                                %
                                       The partial differential equation system 
                                       for $G_\gamma$ to be minimal
                                      gets reduced
                                       to a pair of ordinary differential equations.
                                    By employing arc parameter $s$ for curve $(a(t), b(t))$ in $\mathbb S^1$
                                                (not that of curve $\gamma$ in $\mathbb S^3$),
                                      we have more geometric visions 
                                      and 
                                                 are able to show that the minimal equation system of $G_\gamma$
                                                             can be solved.
                                    In \S \ref{S4},
                                    we
                               assemble local pieces together for global construction of spiral minimal products of varying magnitudes
                               through a wonderful extension result by Harvey and Lawson \cite{HL0}.
                               \S \ref{C=0} is devoted to singly spiral minimal products with $C=0$
                               and 
                               \S \ref{C2=1} for $C=-1$.
                               In particular, in \S \ref{calG} 
                               we apply a fundamental decomposition result of forms due to Harvey and Lawson in \cite{HL1}
                               for useful information regarding antipodal symmetry of $\gamma$
                               and are fortunately able to  further
                                prove Theorem \ref{mainebd}.
                               Based on nice symmetric property of horizontal liftings of minimal Lagrangians in complex projective spaces,
                                Theorem \ref{mL} can be proved in a similar way in \S \ref{calG}.
                             Proofs of Theorems \ref{main} and \ref{main0}
                              will be presented in 
                              \S \ref{S4}
                                                           and \S \ref{GC}
     respectively.
     Finally, we show  in \S \ref{MCR}  that,
           when $k_1+k_2\geq 1$,
           all created spiral minimal products 
           are geometrically distinct from each other
           and, when $k_1=k_2=0$,
           all the created are geometrically the same.

                               {\ }
                               
 \section{Spiral  Minimal Product of Steady Magnitudes}\label{SMP}      
                               In this section,  
                                we focus on the situation where $|\gamma_1|$ and $|\gamma_2|$ remain constant, i.e.,
                                   \begin{equation}\label{A0}
                                           \gamma_1(t)=a e^{is_1(t)} 
                                                 \text{\ \ \ \ and \ \ \ \ }
                                                     \gamma_2(t)=b e^{is_2(t)}
                                   \end{equation} 
                                    for  some $a, b\in \mathbb R$ 
                                     and  $s_1(t), s_2(t)$ of class $C^2$ satisfying
                                                     \begin{equation}\label{A2}
                                                  a^2+b^2=1 , \, \, \, ab\neq 0\, \text{ \ \ and \ }\, \,\,
        \bigg(\frac{ds_1}{dt}\bigg)^2+\bigg(\frac{ds_2}{dt}\bigg)^2\neq 0.
                                                     \end{equation}

                   Given spherical immersions 
                    $f_1:M_1^{k_1}\longrightarrow \mathbb S^{n_1}$ and 
                       $f_2:M_2^{k_2}\longrightarrow \mathbb S^{n_2}$,
                   choose open neighborhoods $U_1$ and $U_2$ of $x\in M_1$ and $y\in M_2$
                   respectively
                   such that  $f_1\vert_{{}_{U_1}}$ and $f_2|_{{}_{U_2}}$ form embeddings.
                  Then
                   we have local         orthonormal bases
         \begin{equation}\label{ns}
                         \Big\{\sigma_1,\cdots, \sigma_{n_1-k_1}\Big\}\text{\ \ \ and \ \ \ }\Big\{\tau_1,\cdots,\tau_{n_2-k_2}\Big\}
         \end{equation}
         of $T^\perp f_1(U_1)$ and $T^\perp f_2(U_2)$ 
and
  orthonormal bases 
                         \begin{equation}\label{ts}
                                       \Big\{e_1,\, \cdots, e_{k_1}\Big\}\text{ \ \ \ and \ \ \ }\Big\{v_1,\, \cdots, v_{k_2}\Big\}
                         \end{equation}
                                       of $Tf_1(U_1)$ and $Tf_1(M_2)$
                                          around $f_1(x)$ and $f_2(y)$ respectively.

                  By Definitions \ref{wiso},
                   $f_1$ and $f_2$ being $\mathscr C$-totally real implies that
      \begin{equation}\label{perp}
                    i\cdot f_1(x)\in \text{ span}_{\mathbb R}\Big\{\sigma_1,\,\cdots, \, \sigma_{n_1-k_1}\Big\}
               \text{\ \ \ and\ \ \ }
                i\cdot f_2(y)\in \text{ span}_{\mathbb R}\Big\{\tau_1,\,\cdots, \, \tau_{n_2-k_2}\Big\}.
        \end{equation}
               So, via selecting bases accordingly,
              one can assume 
     \begin{equation}\label{select}
              i\cdot f_1(x)=\sigma_1 \text{\ \ \ and\ \ \ }i\cdot f_2(y)=\tau_1.
      \end{equation}

               Now we 
               derive preferred local orthonormal bases for the tangential and normal spaces of a $\gamma$-spiral product with \eqref{A0} and \eqref{A2}.
\begin{lem}\label{ltb}
      Assume that spherical immersions $f_1$ and $f_2$ are both $\mathscr C$-totally real.
           Then,
            locally around 
               $p=G_\gamma(t, x, y)=\left(a e^{is_1(t)} f_1(x), b e^{is_2(t)} f_2(y)\right)$
              in $\mathbb S^{n_1+n_2+1}$, 
 the tangent space of  $G_\gamma$                    has 
                      orthonormal basis
    \begin{equation}\label{tangent}
                     \Big\{\left(e^{is_1(t)} e_1,0\right),\,\cdots,\, \left(e^{is_1(t)} e_{k_1},0\right),\,
         \left(0,e^{is_2(t)} v_1\right),\,\cdots,\, \left(0,e^{is_2(t)} v_{k_2}\right), \, E
         \Big\}
      \end{equation}
         where 
                    \begin{equation}\label{E}
                                            E= \dfrac{\frac{dG_\gamma}{dt}}{\left\|\frac{dG_\gamma}{dt}\right\|}
                                            =
         \dfrac{
         \left(
                      a\frac{ds_1}{dt}e^{is_1(t)}\sigma_1, 
         \,
                       b\frac{ds_2}{dt}e^{is_2(t)}\tau_1
         \right)}
                          {\sqrt{
                                    \big(
                                          a\frac{ds_1}{dt}
                                      \big)^2+
                                      \left(
                                            b\frac{ds_2}{dt}
                                       \right)^2}
                            }
                                       .
                       \end{equation}
\end{lem}        
                     \begin{proof}
                     It is clear that all elements in \eqref{tangent}
                     but the last one are  unit tangents 
                             corresponding to 
                                  \eqref{ts}, 
                     mutually orthogonal to each other and perpendicular to the position vector $G_\gamma$.
                     Expression \eqref{E} announces  that $E$ is the unit vector of $\frac{dG_\gamma}{dt}$.
                     Since $\|G_\gamma\|\equiv 1$, it follows automatically that $E\perp G_\gamma$.
                     So we only need to show that $E$ is perpendicular to the rest in \eqref{tangent}.
                     Taking the first element in \eqref{tangent} for example,
                     we have
                     \begin{equation}\label{ptt}
                                   \left\|\frac{dG_\gamma}{dt}\right\|                
                                               \cdot
         \left<
             \left(e^{is_1(t)} e_1,0\right)                                           , 
             E
                                  \right>
 =
          \left<
             e^{is_1(t)} e_1                                         , \,
               { a\frac{ds_1}{dt}e^{is_1(t)}\sigma_1}
                                  \right>
=
   \left<
            e_1                                         , \,
               { a\frac{ds_1}{dt}\sigma_1}
                                                       \right>
 =0.
\end{equation}
                                       The same perpendicular property to others in \eqref{tangent} similarly follows.
                                          \end{proof}
                                          \begin{rem}\label{RRA}
                                             In the second last equality of \eqref{ptt},
                     we consider the $e^{is_1(t)}$-multiplication as a rigid rotational action on $\R ^{n_1+1}$.
                                          \end{rem}
   \begin{lem}\label{lnb}
    Assume that spherical immersions $f_1$ and $f_2$ are both $\mathscr C$-totally real.
           Then,
            locally around 
               $p=G_\gamma(t, x, y)=\left(a e^{is_1(t)} f_1(x), b e^{is_2(t)} f_2(y)\right)$
              in $\mathbb S^{n_1+n_2+1}$,
         the normal space of $G_\gamma$ has orthonormal basis 
         \begin{align}\label{normal}
      \Big\{%
                             \left(e^{is_1(t)} \sigma_2,0\right),\cdots, \left(e^{is_1(t)}  \sigma_{n_1-k_1},0\right),
                   \left(0, e^{is_2(t)}  \tau_2\right),\cdots, \left(0, e^{is_2(t)} \tau_{n_2-k_2}\right),  \eta_1, \eta_0
            \Big\}
         \end{align}
              where
     \begin{equation}\label{eta}
                   \eta_1= 
                                \dfrac{
                                    \left(-b\frac{ds_2}{dt}e^{is_1(t)}\sigma_1, 
         \,
                                       a\frac{ds_1}{dt}e^{is_2(t)}\tau_1\right)
                                       }
         {\sqrt{\left(a\frac{ds_1}{dt}\right)^2+\left(b\frac{ds_2}{dt}\right)^2}}
         \text{\ \ \ and \ \ \ }
         \eta_0=\Big(b e^{is_1(t)} f_1(x),
         \,
          -a e^{is_2(t)}f_2(y)\Big) .
   \end{equation}
   \end{lem}
   \begin{proof}
   According to Remark \ref{RRA} and  our preferences \eqref{ns} and \eqref{select},
   it can be seen that elements of \eqref{normal} are of unit length and mutually orthogonal.
   Also,  all of them are perpendicular to the position vector $G_\gamma$.
   Moreover, it is easy to check that
    elements in \eqref{normal} are perpendicular to those in \eqref{tangent}.
   \end{proof}

                            {\ }
                            
                            \subsection{Second fundamental form of spiral products}\label{2nd}
                            Based on the preparations,
                            we study the second fundamental form of $G_\gamma$ around $p=G_\gamma(t, x, y)$ in this subsection.
                           If need be, shrink the previously chosen $U_1$ and $U_2$ a bit
                           and
                           there exists some sufficiently small $\delta>0$ 
                                    such that $G_\gamma$ restricted to $(t-\delta, t+\delta)\times U_1\times U_2$ is an embedding
                                    with image 
                                    $N$.
                                    For $2\leq k\leq n_1-k_1$
                                                                 and $2\leq l\leq n_2-k_2$,
                                                                 let  
                                                                 $$A_{(e^{is_1(t)}\sigma_{k},0)}:T_pN\longrightarrow T_pN$$
                                                                 $$ A_{(0,e^{is_2(t)}\tau_{l})}:T_pN\longrightarrow T_pN$$
                             $$A_{\sigma_{k}}:
                                             T_{f_1(x)}f_1(U_1)
                                                    \longrightarrow 
                                                    T_{f_1(x)}f_1(U_1)$$
                           %
       $$
         A_{\tau_{l}}:
                        T_{f_2(y)}f_2(U_2)
                                  \longrightarrow 
                                  T_{f_2(y)}f_2(U_2)$$
         be the corresponding shape operators.
         (When $n_1-k_1=1$ or $n_2-k_2=1$, we  skip disappeared indices.)
       Moreover, there are 
                     $$
                     A^{e^{is_1(t)}f_1(U_1)}_{e^{is_1(t)}\sigma_{k}}: 
                            T_{e^{is_1(t)}f_1(x)}\left(e^{is_1(t)}f_1(U_1)\right)
                                \longrightarrow 
                                T_{e^{is_1(t)}f_1(x)}\left(e^{is_1(t)}f_1(U_1)\right)
                                $$
                                and
                       $$A^{e^{is_2(t)}f_2(U_2)}_{e^{is_2(t)}\tau_{l}}: 
                            T_{e^{is_2(t)}f_2(y)}\left(e^{is_2(t)}f_2(U_2)\right)
                                \longrightarrow 
                                T_{e^{is_2(t)}f_2(y)}\left(e^{is_2(t)}f_2(U_2)\right)
                                $$
         for  fixed  $t$.
         Since multiplications $e^{is_1(t)}$ and $e^{is_2(t)}$
         are isometries,  
              $ A^{e^{is_1(t)}f_1(U_1)}_{e^{is_1(t)}\sigma_{k}}$ and $A^{e^{is_2(t)}f_2(U_2)}_{e^{is_2(t)}\tau_{l}}$
         are exactly those of $A_{\sigma_{k}}$ and  $A_{\tau_{l}}$ respectively under natural corresponding bases.
         Hence we identify each corresponding pair of them 
         throughout this paper.

         The following two basic propositions regarding covariant derivatives in Euclidean space 
         will be repeatedly employed.
         The importance is this. 
         Given two tangent vector fields along a hypersurface, for example, a hypersphere,
         projections of resulting derivatives induce covariant derivatives on the hypersurface \cite{doCarmo, Top}.
              \begin{prop}\label{001}
              Given embedding $f:U\longrightarrow \R^ {n+1}$,
              tangent vector fields $V_1(x), V_2(x)$   on $f(U)$
              and positive $a\in \R$.
              Multiply base points by $a$ to get
               $V_1(ax), V_2(ax)$  as tangent vector fields on $a\cdot f(U)$.
               With Levi-Civita connection  $\overline{\mathcal D}$ of $\R^ {n+1}$,
               we use
               $\overline{\mathcal D}_{V_1}V_2\big |_{(\cdot)} $
               for $\left(\overline{\mathcal D}_{V_1 \, |_{(\cdot)}}V_2\right) {(\cdot)}$\, .
               Then by ignoring base points we have that
               $\overline{\mathcal D}_{V_1}V_2\big |_{a\cdot f(x)}
               =\frac{1}{a}\overline{\mathcal D}_{V_1}V_2\big |_{f(x)}$ .
              \end{prop}
              \begin{proof}
              Let $\mu(t)$ be an integral curve of $V_1$ in $f(U)$ with $\frac{d\mu}{dt}=V_1 |_{\mu(t)}$.
              Then $a\cdot \mu(t)$  gives an integral curve of  vector field $a\cdot V_1$ in $a\cdot f(U)$ with $\frac{d (a\cdot \mu)}{dt}=a\cdot V_1 |_{a\cdot \mu(t)}$.
              Ignoring base points, $V_2$ at  $\mu(t)$ equals the corresponding vector at $a\cdot \mu(t)$.
              Hence, 
                     $\overline{\mathcal D}_{a\cdot V_1}V_2\big |_{a\cdot \mu(t)}=\frac{d}{dt}V_2|_{a\cdot \mu(t)}=\frac{d}{dt}V_2|_{\mu(t)}=\overline{\mathcal D}_{V_1}V_2\big |_{\mu(t)}$,
                     which arrives at the conclusion.
              \end{proof}

              \begin{prop}\label{002}
              Let $f:U\longrightarrow \R^ {n+1}=\mathbb C^{m+1}$ be an embedding
              when $n+1=2m+2$,
              $V_1, V_2$   tangent vector fields  on $f(U)$
              and 
              $e^{\lambda i}$
              with $\lambda\in \R$
              regarded
             as rotation $e^{\lambda i}I_{m+1}$.
              Then
              we have isometry
                $e^{\lambda i}:  f(U)\longrightarrow  e^{\lambda i} f(U)$
                and push-forwarding vector fields $e^{\lambda i}V_1, e^{\lambda i}V_2$   on $e^{\lambda i}f(U)$.
                Ignoring base points, 
                it follows that
                 $$\overline{\mathcal D}_{e^{\lambda i}V_1}e^{\lambda i}V_2\big |_{e^{\lambda i} f(x)}
               =e^{\lambda i}\left(\overline{\mathcal D}_{V_1}V_2\big |_{f(x)}\right).$$
              \end{prop}
              \begin{proof}
              This is clear by isometry $e^{\lambda i}$ of $\mathbb C^{m+1}$.
              \end{proof}
              
              \begin{rem}
              When there is no ambiguity, 
              we ignore base points 
              as in Lemmas \ref{ltb}  and \ref{lnb}
              unless necessary for some computations.
              \end{rem}

         With identifications explained in the beginning of this subsection, 
                we have the following result on the second fundamental form of $G_\gamma$ with \eqref{A0} and \eqref{A2}.

         \begin{lem}\label{L1}
         At $p=
          \left(a e^{is_1(t)}f_1(x),\, b e^{is_2(t)}f_2(y)
         \right)$,
         it follows that, with respect to the rotated tangent basis \eqref{tangent},
         \begin{align}
A_{\left(e^{is_1(t)}\sigma_{k},0\right)}&=
                   \begin{pmatrix}
                          \frac{1}{a} A^{e^{is_1(t)}f_1(U_1)}_{e^{is_1(t)}\sigma_{k}} &  & *\\
                           &O & *\\
                           *&* & 0
                    \end{pmatrix}
            =
                      \begin{pmatrix}
                     \frac{1}{a}A_{\sigma_{k}} &  & *\\
                     &O & *\\
                     *&* & 0
                     \end{pmatrix}
            \text{\ \ for \  } 2\leq {k}\leq n_1-k_1, 
             \label{R1}
            \\
A_{\left(0, \, \, e^{is_2(t)} \tau_{l}\right)} &=
                      \begin{pmatrix}
                         O &  &* \\
                          & \frac{1}{b} A^{e^{is_2(t)}f_2(U_2)}_{e^{is_2(t)}\tau_{l}} & * \\
                          *&* & 0
                          \end{pmatrix}
 =
                      \begin{pmatrix}
                      O &  &* \\
                      & \frac{1}{b}A_{\tau_{l}} & * \\
                      *&* & \,0 {\ }
                      \end{pmatrix}
                  \text{\ \ for \  } 2\leq {l}\leq n_2-k_2,
                  \label{R2}
\end{align}
   %
          %
          %
\begin{equation}\label{R3} 
                   A_{\eta_1} =
                   \dfrac{         1}
                   {\sqrt{\Theta}}
                           \begin{pmatrix}
                           -\frac{b}{a}\frac{ds_2}{dt}A^{e^{is_1(t)}f_1(U_1)}_{e^{is_1(t)}\sigma_{1}}   &  & \\
                           & \frac{a}{b}\frac{ds_1}{dt} A^{e^{is_2(t)}f_2(U_2)}_{e^{is_2(t)}\tau_{1}} & \\
                           & & \circledast
                           \end{pmatrix}
        =
                               \dfrac{1}
                                {\sqrt{\Theta}}
                                \begin{pmatrix}
                                -\frac{b}{a}\frac{ds_2}{dt}A_{\sigma_1}  &  & \\
                                & \frac{a}{b}\frac{ds_1}{dt}A_{\tau_1} & \\
                                & & \circledast
                                \end{pmatrix}
\end{equation}
          where 
                  \begin{equation}
                  \Theta={\left(a\frac{ds_1}{dt}\right)^2+
                                        \left(b\frac{ds_2}{dt}\right)^2}
                                        \text{\ \ \ \ \ and\ \ \ \ \ }
                                        \circledast=
                          \frac{
                             ab
                          }
                                        {\sqrt \Theta}
                                               \left(
                             \frac{ds_1}{dt}
                             \frac{d^2s_2}{dt^2}
                             -
                             \frac{d^2s_1}{dt^2}
                                                           \frac{ds_2}{dt}
                             \right)
                                       .
              \end{equation}
\end{lem}
\begin{proof}
          We take the advantage of the following three Levi-Civita connections.
            \begin{eqnarray}
                    N\ \longrightarrow &\mathbb S^{n_1+n_2+1}\longhookrightarrow &\mathbb R^{n_1+n_2+2}
                       \nonumber\\
                          \specialrule{0em}{2pt}{2pt} 
                    \nabla\ \ \ \ \ \ \
                    \ \ &
                     \overline{\nabla}\ \ \ \ \ \ \ &\ \ \ \ \ \mathcal D
                       \nonumber
                       \end{eqnarray}
Here $\nabla$ is the induced Levi-Civita connection.
For $p\in N$, tangent vector fields $X$, $Y$
and normal vector field $\eta$ of $N$  around $p$,
we have that  (at $p$)
\begin{eqnarray}
                     \left<A_{\eta} X, Y\right>
                     &=&
                    \left <\overline{\nabla}_XY-{\nabla}_XY,\ \eta\right>
                    \nonumber\\
                    &=& 
                     \left <\overline{\nabla}_XY,\ \eta\right>
  \nonumber\\
              &=& 
                     \left <\mathcal D_XY-\bar h(X,Y),\ \eta\right>
  \label{n01}\\
                   &=& \left<\mathcal D_XY,\ \eta\right>
           \label{N2}\\
                    &=& -\left<\mathcal D_X\eta,\ Y\right> 
           \label{N3}
\end{eqnarray}
where    
                  $\bar h$ in \eqref{n01} means the second fundamental form of 
$\mathbb S^{n_1+n_2+1}$ in $\mathbb R^{n_1+n_2+2}$.
               \begin{rem}\label{exind}
                Here note that local   
          extensions of $X, Y, \eta$ along $N$ to $\mathbb S^{n_1+n_2+1}$ and further to $\mathbb R^{n_1+n_2+2}$
          are taken by default.
         Since results are independent of extensions, they will be omitted unless useful in some computations.
               \end{rem}

            With vector fields $X=(X_1,X_2)$ and $Y=(Y_1,Y_2)$ around $p=(p_1,p_2)$
where 
 \begin{eqnarray}
         & X_1, Y_1\in T
               \big(ae^{is_1(t)} f_1(U_1)\big),
     \label{splitXY1}           \\
                    &            X_2, Y_2\in T
                                  \big(be^{is_2(t)} f_2(U_2)\big),
                                  \label{splitXY2}
 \end{eqnarray}
                                  and splitting normal vector (at $p$)
              \begin{eqnarray}
                          \ \ \ \ \ \ \ \ \ \ \      \eta &=& 
                                                                        \ \ \ \ \ \ \ \ \ \ \ \ \ \ \ \ \ \eta_A
                                                                            \ \ \ \ \ \ \ \ \ \ \ \ \ \  \ \  \ +
                                                                       \   \ \ \ \ \ \ \ \ \ \ \ \ \  \ \  \  \eta_B
                                \\
                                       &= & \overbrace{
                                                                   \sum_{2\leq k\leq n_1-k_1} c_k\cdot  
                                                                                             {\left(e^{is_1(t)}\sigma_{k},0\right)} 
                                                               }
                                        \ +\ 
                                               \overbrace{
                                                           \sum_{2\leq l\leq n_2-k_2} c'_l\cdot
                                                                                             {\left(0, e^{is_2(t)} \tau_{l}\right)} 
                                                                }
                                        \text{ \ with \ }
                                        c_k,\, c'_l\in \R,
              \end{eqnarray}
                             by \eqref{N2} we have that
\begin{alignat}{3}
       &\, \, \, \ \ \left<A_{\eta} X, Y\right>
      \nonumber \\
                     &= \left<    
                                   \left.
                                         \mathcal D_XY
                                    \right|_p,\ 
                            \eta\right>
  \nonumber\\
            &=
         \left<
              \left(
              \left.
              \mathcal D'_{X_1}Y_1 
              \right|_{p_1},
              \, 
                    \left.
                 \mathcal D''_{X_2}Y_2
                 \right|_{p_2}
              \right)
                                           , \eta
                                  \right>
  \label{N5}\\
      &=  \left<
                                          \left(
                                           \left.
                                       \frac{1}{a}   \mathcal D'_{{X_1}}Y_1
                                          \right|_{e^{is_1(t)}f_1(x)},
                                          \, 
                                         \left.
                                         \frac{1}{b} \mathcal D''_{{X_2}}Y_2
                                           \right|_{e^{is_2(t)}f_2(y)}
                                          \right)
                                          , \eta\right>_{\mathbb R^{n_1+n_2+2}}
  \label{ff} \\
      &=   \left<
                                          \left(
                \left.
                                  \frac{e^{is_1(t)}}{a}
                                          \mathcal D'_{{e^{-is_1(t)}X_1}} e^{-is_1(t)}Y_1
                  \right|_{f_1(x)},
                                          \, 
                                    \left.
                                         \frac{e^{is_2(t)}}{b}
                                          \mathcal D''_{{e^{-is_2(t)}X_2}}  e^{-is_2(t)}Y_2
                                     \right|_{f_2(y)}
                                          \right)
                                           , \eta\right>_{\mathbb R^{n_1+n_2+2}} 
      \label{fff}\\
        &=
                                               \frac{1}{a}
                                                    \left<
                \left.
                                                                        \mathcal D'_{{e^{-is_1(t)}X_1}} e^{-is_1(t)}Y_1
                  \right|_{f_1(x)}
                                           , e^{-is_1(t)}\eta_A\right>_{\mathbb R^{n_1+1}} 
           \\
  &   \ \ \ \ \ \ \ \ \ \ \ \ \ \ \ \ \ \ \ \ \ \ \ \ \ \ \ \ \ \ \ \ \ \   \ \ \ \ \ \ \ \ \ \ \ \   +
                                            \frac{1}{b}
                                                    \left<
                \left.
                                                                        \mathcal D''_{{e^{-is_2(t)}X_2}} e^{-is_2(t)}Y_2
                  \right|_{f_2(x)}
                                           , e^{-is_2(t)}\eta_B\right>_{\mathbb R^{n_2+1}} 
                                           \nonumber
     \\
                   &=
                                    \frac{1}{a}
                          \left<
                                    A_{e^{-is_1(t)}\eta_A}     {e^{-is_1(t)}X_1}
                                    ,
                                    e^{-is_1(t)}Y_1
                          \right>
         +
           \frac{1}{b}
                          \left<
                                    A_{e^{-is_2(t)}\eta_B}     {e^{-is_2(t)}X_2}
                                    ,
                                    e^{-is_2(t)}Y_2
                          \right>.
  \label{ffff}
\end{alignat}
            Here $\mathcal D'$ and $\mathcal D''$ represent the Levi-Civita connections of $\mathbb R^{n_1+1}$ and $\mathbb R^{n_2+1}$,
            and according to Remark \ref{exind}  one can extend $X_1,Y_1,\eta_A$ and $X_2,Y_2,\eta_B$ to be local vector fields around $p_1$ and $p_2$ in $\mathbb R^{n_1+1}$ and $\mathbb R^{n_2+1}$ respectively
            for understandings of \eqref{N5} and \eqref{ffff}.
            Equalities \eqref{ff} and \eqref{fff} are due to Propositions \ref{001} and \ref{002}.
                 Expression \eqref{ffff} then verifies all elements except in the last row or last column of \eqref{R1} and \eqref{R2}.

                                              For elements in  the bottom right corners  of  \eqref{R1} and \eqref{R2},   we need to do some computations.
With $\Theta={\left(a\frac{ds_1}{dt}\right)^2+
                                        \left(b\frac{ds_2}{dt}\right)^2}$,
                                 on one hand,     
     \begin{equation}\label{ddG1}
                  \frac{d^2 G_\gamma}{dt^2}
                  =
                          \mathcal D_{\sqrt \Theta E}{\sqrt \Theta E}
                  =        \Theta\, \mathcal D_EE
                         +\dfrac{a^2\frac{ds_1}{dt}\frac{d^2s_1}{dt^2}+b^2\frac{ds_2}{dt}\frac{d^2s_2}{dt^2}}{\sqrt \Theta}\, E.
     \end{equation}
            On the other hand,
     \begin{equation}\label{ddG2}
        \frac{d^2 G_\gamma}{dt^2}
                  =
                     \bigg(
                          ae^{is_1(t)}
                           \bigg[
                           \frac{d^2s_1}{dt^2}\sigma_1-\left(\frac{ds_1}{dt}\right)^2f_1
                           \bigg]
                           ,\,
                            be^{is_2(t)}
                           \bigg[
                           \frac{d^2s_2}{dt^2}\tau_1-\left(\frac{ds_2}{dt}\right)^2f_2
                           \bigg]
                       \bigg).
         \end{equation}
                        Hence, $\mathcal D_EE$ is perpendicular to the normals in \eqref{R1} and \eqref{R2}.
                            Therefore, relations \eqref{R1} and \eqref{R2} get proved.
                            
                            As for relation \eqref{R3}, 
                            elements except in the last row or last column can be gained in the same way by \eqref{N5}$-$\eqref{ffff}.
                            Furthermore, 
                            by \eqref{ddG1} and \eqref{ddG2},
                             the bottom right corner is filled by 
          \begin{equation}
                            \circledast
                            =
                            \sqrt\Theta
                         \big  <\mathcal D_EE, \eta_1 \big>
                            =
                            \frac{ab}{\sqrt \Theta}
               \left(
                             -
                             \frac{d^2s_1}{dt^2}
                             \frac{ds_2}{dt}
                             +
                              \frac{ds_1}{dt}
                             \frac{d^2s_2}{dt^2}
               \right).
           \end{equation}
                            For the rest part of \eqref{R3}, we can employ \eqref{N2}
                            to have, for $X=(X_1,X_2)$ as above, that
                                  \begin{eqnarray}
                             \ \ \ \  \ \ \ \ \ \     \left<A_{\eta_1} X, E\right>
&=&
                                       \big<\mathcal D_XE,\, \eta_1\big>
   \nonumber\\
&=&           
                        \frac{1}{\sqrt\Theta}\,
   \left<
              \left(
                     \left.
                             \mathcal D'_{X_1} a\frac{ds_1}{dt}e^{is_1(t)}if_1(x)
                     \right|_{p_1},
              \, 
                    \left.
                               \mathcal D''_{X_2} b\frac{ds_2}{dt}e^{is_2(t)}if_2(y)
                      \right|_{p_2}
              \right)
                                           ,\, \eta_1
                    \right>
 \label{N4}\\
                    &=&   
                 \   \       \frac{1}{\Theta}\,\,\
        \Bigg\{
                           \left<
                                        i  {\frac{ds_1}{dt}}
                                                 X_1
                                                 ,\, 
                                          -b\frac{ds_2}{dt}e^{is_1(t)}if_1(x)
                             \right>_{\mathbb R^{n_1+1}} 
   \label{ee}\\
                            & &  \ \  \ \  \ \  \ \  \ \ \ \ \  \ \ \ \ \ \ \ \ \ \  \  \ \ \ \ \ \ \ \ \  \ \  \ \ \ \ \ \ \, +
                       \left<
                                   i     \frac{ds_2}{dt}
                                                X_2
                              ,\; 
                               a\frac{ds_1}{dt}e^{is_2(t)}if_2(y)
                        \right>_{\mathbb R^{n_2+1}}
    \Bigg\}
\nonumber  \\
                                    &=& \ 0  .
 \label{eeee}
      \end{eqnarray}
      Here \eqref{N4} comes from the vanishments of $\mathcal D_{X_1}'{\frac{1}{\sqrt \Theta}}$ and $\mathcal D_{X_2}''{\frac{1}{\sqrt \Theta}}$,
                the equality \eqref{ee} is due to the fact that $X_1\vert_{p_1}$ corresponds to $\frac{e^{-is_1(t)}X_1}{a}\big\vert_{f_1(x)}$ 
                and $X_2\vert_{p_2}$ corresponds to $\frac{e^{-is_2(t)}X_2}{b}\big\vert_{f_2(y)}$,
                 and
                 \eqref{eeee} is  by Proposition \ref{002} and the assumption that $f_1$ and $f_2$ are $\mathscr C$-totally real.         Now the proof gets completed.
\end{proof}


       \begin{rem}\label{rem1}
                  By replacing the normals in \eqref{N4}, we can get expressions of all $*$'s values in \eqref{R1} and \eqref{R2}.
                  For example,
                  the $(w,k_1+k_2+1)$-element of \eqref{R1} is 
                               $\frac{1}{\sqrt\Theta}
                                   \big<\frac{ds_1}{dt}ie_{w},\, \sigma_{k}
                                   \big>
                                   $
                               and the $(k_1+k_2+1,w)$-element shares the same value.
                               Situation of \eqref{R2} is similar.
        \end{rem}
              \begin{rem}\label{rem2}
              If $\frac{ds_1}{dt}\neq 0$ (or $\frac{ds_2}{dt}\neq 0$) pointwise in local,
              then $$\circledast=\frac{ab}{\sqrt \Theta}\dfrac{d}{dt}\left(\frac{ds_2}{dt}\bigg/\frac{ds_1}{dt}\right)\left(\frac{ds_1}{dt}\right)^2.$$
             So $\circledast$ vanishes if and only if 
                   $\frac{ds_2}{dt}\big/\frac{ds_1}{dt}$ (or $\frac{ds_1}{dt}\big/\frac{ds_2}{dt}$)
                   is a constant.
                   By \eqref{A2}, we know that $\frac{ds_1}{dt}$ and $\frac{ds_2}{dt}$ cannot be zero simultaneously.
                   Therefore, 
                  $\circledast$ vanishes
                  if and only if
                 there exists some constant $c\in \mathbb R$
                   such that
                   $\frac{ds_2}{dt}=c\frac{ds_1}{dt}$ (or $\frac{ds_1}{dt}=c\frac{ds_2}{dt}$) 
                   for every $t\in\mathbb R$.
              \end{rem}

              Now let us consider the second fundamental form of $G_\gamma$ with respect to the last normal vector $\eta_0$ in \eqref{normal}.
 \begin{lem}\label{L2}
                          At $p=\left(a e^{is_1(t)}f_1(x),\, b e^{is_2(t)}f_2(y)\right)$,
                           it follows that, with respect to the rotated tangent basis \eqref{tangent},
        \begin{equation}\label{R4}
                          A_{\eta_0} =
                                  \begin{pmatrix}
                                          -\dfrac{b}{a} I_{k_1} &  & \\
                                          & \dfrac{a}{b} I_{k_2} & \\
                                          &  & \boxast
                                    \end{pmatrix}
        \end{equation}
     where 
                  \begin{equation}
                  \Theta={
                  \bigg(
                       a\frac{ds_1}{dt}
                   \bigg)^2+
                                        \left(b\frac{ds_2}{dt}\right)^2}
                                        \text{\ \ \ \ \ and\ \ \ \ \ }
                                        \boxast=
                         \frac{ab}{\Theta}
               \bigg[
                        \left(
                             \frac{ds_2}{dt}
                        \right)^2
                             -
                              \left(
                             \frac{ds_1}{dt}
                        \right)^2
               \bigg].
              \end{equation}
\end{lem}
         \begin{proof}
         By \eqref{N3} 
         and
         $\eta_0=\big(b e^{is_1(t)} f_1(x),
         \,
          -a e^{is_2(t)}f_2(y)\big)$, 
          we deduce similarly 
          with $X=(X_1, X_2)$ and  $Y=(Y_1, Y_2)$ in \eqref{splitXY1} and   \eqref{splitXY2} 
          that
\begin{eqnarray}
                                    \left<A_{\eta_0} X, Y\right>
                     &=&
                   -\big<\mathcal D_X\eta_0,\ Y\big>
   \nonumber\\
                   &=& 
                   -     \left<
              \left(
              \left.
              \mathcal D'_{X_1} be^{is_1(t)}f_1(x)
              \right|_{p_1},
              \, 
                    \left.
                - \mathcal D''_{X_2} ae^{is_2(t)}f_2(y)
                 \right|_{p_2}
              \right)
                                           ,\, (Y_1, Y_2)
                                  \right>
           \label{N6}\\
                     &=& 
               \ \ \ \Big<
               \Big(                           -               {\frac{b}{a}
                                          {X_1}}
               ,
                \frac{a }{b}X_2
                \Big),\; (Y_1,Y_2)
                \Big>_{\mathbb R^{n_1+n_2+2}} .
   \nonumber
      \end{eqnarray}
                            The last equality is because $X_1\vert_{p_1}$ corresponds to $\frac{e^{-is_1(t)}X_1}{a}\big\vert_{f_1(x)}$
                            and similarly $X_2\vert_{p_2}$  to $\frac{e^{-is_2(t)}X_2}{b}\big\vert_{f_2(y)}$.
                                   By replacing $\eta_1$ by $\eta_0$  in \eqref{N4}, 
                                             we can also figure out that elements except $\boxast$ in the last column (or row by symmetry of shape operator) are all zeros.
                                             Here we use
                                             the $\mathscr C$-totally real property: $if_1(x)\perp X_1$ and $if_2(y)\perp X_2$.
                                   
                                   Since $\eta_0\perp E$,
                                   by utilizing \eqref{ddG1} and \eqref{ddG2}
                                   we have the element in the bottom right corner
     \begin{equation}
                            \boxast
                            =
                            \big<\mathcal D_EE, \eta_0\big>
                            =
                            \frac{ab}{\Theta}
               \bigg[
                             -
                        \left(
                             \frac{ds_1}{dt}
                        \right)^2
                             +
                              \left(
                             \frac{ds_2}{dt}
                        \right)^2
               \bigg].
    \end{equation}
Thus, the proof is complete.
         \end{proof}

                            {\ }


              \subsection{Spiral minimal  products of steady magnitudes}\label{poc}
               With
               $n_1+1=2m_1+2$ and $n_2+1=2m_2+2$, 
                        we show part \ding{172} of Theorem \ref{main} about doubly spiral minimal products of steady magnitudes
                        in the following plus the complementary situation in Remark \ref{Csing}.
    \begin{thm}\label{T1}    
                Assume that
                            $k_1,\, k_2\geq 1$.
                  Given a pair of $\mathscr C$-totally real minimal       
                                       immersions 
                                               $f_1:M_1^{k_1}\longrightarrow \mathbb S^{n_1}\subset \mathbb C^{m_1+1}$ and 
                                                $f_2:M_2^{k_2}\longrightarrow \mathbb S^{n_2}\subset \mathbb C^{m_2+1}$,
                there are {uncountably} 
                                       many 
                spiral minimal products $G_\gamma$ into $\mathbb S^{n_1+n_2+1}$
                by curve 
                $\gamma$ of format
                $
                      \left(
                               a e^{is_1(t)} 
                               , \,
                               b e^{is_2(t)}
                        \right).
                               $
     \end{thm}
                     \begin{proof}
                     By Lemmas \ref{L1} and \ref{L2}, 
                                we know that 
                                       $G_\gamma$
                                              with 
                                                      $
                                                              \gamma
                                                              =
                                                         \left(
                                                               a e^{is_1(t)} 
                                                               ,
                                                               \, 
                                                                b e^{is_2(t)}
                                                          \right)
                                                      $
                                               is minimal if and only if
                     $$\circledast=0 
                               \text{\ \ \ \ \ \ and\ \ \ \ \ } 
                                -\frac{b}{a}k_1+\frac{a}{b}k_2+\boxast=0$$
                      which by Remark \ref{rem2}
                                become, for some $c\in \R$,
                                      \begin{equation}\label{c}
                                      \frac{ds_2}{dt}=c\frac{ds_1}{dt} \ \ \ \ \text{ (or the other way around) }
                                      \end{equation}
                                       and
\begin{equation}\label{cplugin}
                                               -\frac{b}{a}k_1
                                                 +\frac{a}{b}k_2
                                                   + 
                            \frac{ab}{\Theta}
               \left(
               c^2-1
               \right)
                        \left(
                             \frac{ds_1}{dt}
                        \right)^2                     
                                     =0    .
      \end{equation}                                     
                      Since $\Theta=$ $\left(a^2+b^2c^2\right)$ $\left(\frac{ds_1}{dt}\right)^2$, 
                      the only requirement for $G_\gamma$ to be minimal
                              is
                              that
      \begin{equation}\label{trace}
         \text{trace}(A_{\eta_0})=-\frac{b}{a}k_1+\frac{a}{b}k_2+ \frac{ab(c^2-1)}{a^2+b^2c^2 } \text{ \ \ becomes zero}.
       \end{equation}

                               Without loss of generality for our geometric concern (as reparametrizations do not change mean curvature vector field of $G_\gamma$), 
                              one can require that
                                            $\frac{ds_1}{dt}=1$
                                            in the very beginning
                                            and then
                                             $\frac{ds_2}{dt}=c\, (\in \R)$.
                                             Assume that $c\neq 0$.
                                            Then $a^2+b^2c^2\geq \min \{c^2, 1\}>0$.
                               Set $b=\sqrt {1-a^2}>0$.
                               Observe that, by $k_1,k_2>0$,
                                 $\text{trace }(A_{\eta_0})$ trends toward $-\infty$ 
                                           as $a\downarrow 0$
                                           and toward $+\infty$
                                           while $a\uparrow 1$.
                               Thus, by continuity, 
                                          there exist  real number $a_0$ and corresponding $b_0$ 
                                          with vanishing  $\text{trace }(A_{\eta_0})$
                                           so that 
                                                 $G_\gamma$                                                  with 
                                                     $
                                                              \gamma
                                                              =
                                                         \left(
                                                               a_0 e^{i t} 
                                                               , 
                                                                b_0 e^{ic t}
                                                          \right)
                                                      $
                                            is minimal.
                                     Evidently,
                                     with different nonvanishing $c\in \R$
                                     our construction 
                                     produces uncountably many doubly spiral minimal products of $f_1$ and $f_2$.
                     \end{proof}

     \begin{rem}\label{rem3} 
                    In the case of  $c=0$ in \eqref{c} 
                    there is a singly spiral minimal product of steady magnitudes
             satisfying
                                                                      $\text{trace }(A_{\eta_0})=-\frac{b}{a}(k_1+1)+\frac{a}{b}k_2.$
                        Geometrically, it is exactly the constant minimal product 
                        (see \S \ref{P} )
                                  of $\mathcal Sf_1$ and $f_2$
                                   based on the following proposition.
      \end{rem} 
       \begin{prop}\label{prop1}
                           Let
                                               $f:M\longrightarrow \mathbb S^{2m+1}\subset \mathbb C^{m+1}$
                                   be a $\mathscr C$-totally real minimal       immersion.
                                   Then the spiral immersion
                     \begin{eqnarray} 
                                \mathcal Sf: \R\times M&\longrightarrow& \ \ \ \mathbb S^{2m+1} \text{ \ \ \ \  \ \ given by}\\
                                      (t, x)\  \  \ &\longmapsto & \ e^{it} f(x)
    \nonumber
                   \end{eqnarray}
                                          is minimal in $\mathbb S^n$.
      \end{prop}
       \begin{proof}
       By simple computation.
       \end{proof}
           \begin{rem}\label{rem33} 
                   Compared with the structure explained in Remark \ref{rem3},
                   the case of $c= 1$  in \eqref{c} 
                   is the other way around.
                   It is first to take the constant minimal product of $f_1$ and $f_2$ 
                   and then apply Proposition \ref{prop1}.
                   Spiral minimal product $G_\gamma$ with $c=-1$ 
                          is geometrically
                                 the mirror,  through the complex conjugate action on the $\C^{m_2+1}$ part, of the spiral minimal product with $c=1$ of $f_1$ and complex conjugate $\overline{f_2}$ of $f_2$.
      \end{rem} 
       \begin{rem}\label{absqrt}
               Without sign restrictions, 
                  values of $a$ and $b$ for vanishing $\text{trace }(A_{\eta_0})$ in \eqref{trace}
              satisfy
\begin{equation}\label{b/a}
          \frac{b}{a}=\pm 
          \sqrt{
          \dfrac{-\left((k_1+1)-(k_2+1)c^2\right)+
          \sqrt{\left((k_1+1)-(k_2+1)c^2\right)^2+4k_1k_2c^2}
          }
          {2k_1c^2}}.
\end{equation}
                         Letting $c^2\uparrow \infty$,
                         one gets $\left|\frac{b}{a}\right|\longrightarrow \sqrt{\frac{k_2+1}{k_1}}$
                         and
                         the limit corresponds to the constant minimal product of $f_1$ and $\mathcal Sf_2$ with $\frac{ds_1}{dt}=0$
                        (see Remark \ref{rem3} and Proposition \ref{prop1}).      
      \end{rem}
 \begin{rem}
             Although for any $c\in \R$,
             by \eqref{b/a}
             there exists a unique solution $(a,b)$ with $a, b>0$ to generate a spiral minimal product,
            however it is not true that 
                         given any positive pair $(a,b)$ 
                               there will be suitable $c$
                               to generate a spiral minimal product.
              The reason is this.
                     When $1-c^2>0$, 
             \begin{equation}
                         \label{ineq1}
                               \dfrac{ab(1-c^2)}{a^2+b^2c^2}
                                        \leq 
                                 \dfrac{ab(1-c^2)}{a^2}
                                        \leq
                                        \dfrac{b}{a}
                                         ;
              \end{equation}
                         and when $1-c^2<0$,
        \begin{equation}
                     \label{ineq2}
                               \dfrac{ab(1-c^2)}{a^2+b^2c^2}
                                        \geq 
                                 \dfrac{ab(1-c^2)}{b^2c^2}
                                        \geq
                                        -\dfrac{a}{b}
                                         .
                    \end{equation}
             Hence, plugging \eqref{ineq1} and \eqref{ineq2} into vanishing $\text{trace}(A_{\eta_0})$ in \eqref{trace}
                       gives
             $$
             \dfrac{k_1}{k_2+1}
             \leq
                   \dfrac{a^2}{b^2}
                   \leq
                  \dfrac{k_1+1}{k_2}
                  .
              $$
              The equalities correspond to extremal situations with 
              $c^2=\infty$ and $c^2=0$
              (see Remark \ref{rem3}).
              Therefore, domain of $\frac{a^2}{b^2}$ 
                          on which spiral minimal products of steady magnitudes exist 
                           is exactly the interval $\big[\frac{k_1}{k_2+1},  \frac{k_1+1}{k_2}\big]$
                           by Remarks \ref{rem3} and \ref{absqrt}.
\end{rem}
      \begin{rem}\label{Csing}
              When only one of $k_1$ and $k_2$ is zero,
             say $k_2=0$,
            we have 
              \begin{equation}\label{trace2}
                                                 \text{trace }(A_{\eta_0})=-\dfrac{b}{a}k_1-\dfrac{ab(1-c^2)}{a^2+b^2c^2}  .
                                       \end{equation}
                  When $a,b>0$,
                   the inequality $1-c^2<0$ is a necessary condition for $A_{\eta_0}$ possible to be traceless.
                  Actually, \eqref{ineq2} tells that
                              $$
                                      \bigg[  -\dfrac{a}{b}, 0\bigg)
                                      =
                                      \left\{
                              \left.
                                      \dfrac{ab(1-c^2)}{a^2+b^2c^2}
                             \right| 1-c^2<0, 
                             \, c\in \R \cup \{\pm\infty\}
                                      \right\}
                                       .
                                      $$
                                      
                                            Taking singly spiral situation $-\frac{a}{b}$ ( i.e, $c^2=\infty$) into consideration,
        \eqref{trace2} 
               becomes $-\frac{b}{a}k_1+\frac{a}{b}$
              and we get $a=\sqrt{\frac{k_1}{k_1+1}}, b=\sqrt{\frac{1}{k_1+1}}$ for $G_\gamma$ to be minimal.
               See Remark \ref{absqrt}.
               
               When $1<c^2<\infty$, for $\text{trace }(A_{\eta_0})$ to be zero, we have
\begin{equation}
          \label{spfrac}
                        \dfrac{b^2}{a^2}=\dfrac{c^2-(k_1+1)}{k_1 c^2}\, .
\end{equation}
               
               As a result, we establish:
               When $k_1\neq 0$ and $k_2=0$, 
               for every  $c\in \R$ 
                         satisfying $c^2>k_1+1$,
               there exists a unique positive pair $(a, b)$ 
                           such that $G_\gamma$ is minimal.
                           Similarly, when $k_2\neq 0$ and $k_1=0$,
                           the same existence and uniqueness hold for every $c\in \R$ with $c^2<\frac{1}{k_2+1}$.
\end{rem}

\begin{rem}
                              When $k_1=k_2=0$,
        it follows that vanishing $\text{trace}(A_{\eta_0})$ 
        under restriction \eqref{A2} forces $c=\pm 1$ (cf. Proposition \ref{prop1} and Remark \ref{rem33})
        which run  great circles.
Moreover, once sizes $a$ and $b$ are fixed, 
the resulting two circles are, in the geometric sense, the only great circles in $\mathbb S^3$ with projection to each $\C$-component  of $\C\oplus\C$ forming a round circle.
Compare with the case of varying magnitudes for $k_1=k_2=0$ in \S \ref{k120}.
\end{rem}

{\ }

 \section{Local Construction of Spiral  Minimal Products of Varying Magnitudes}\label{GSMP} 
                   For more diversities of spiral minimal products,
                  in this section we
 consider
                                     \begin{equation}\label{ggamma}
                                     \gamma(t)
                                                              =
                                                         \left(
                                                               a(t) e^{is_1(t)} 
                                                                , \, 
                                                                b(t) e^{is_2(t)}
                                                          \right)
                                                          \in
                                                          \mathbb S^3
                                                          \,  
                                                         \text{\ \ \ for  \ }
                                                         t\in \R
                                                         \, 
                                         \end{equation}
                                       of varying $(a(t), b(t))$  in $\mathbb S^1$ and
          \begin{equation}\label{2expression}
                             G_\gamma:\  \big(t,x,y\big)\longmapsto \Big( a(t) e^{is_1(t)} f_1(x),\,  b(t) e^{is_2(t)} f_2(y)\Big) .
        \end{equation}

          Let us first study  the doubly spiral type.
         Morrey's regularity theory \cite{M1, M2}
         implies that every spherical $C^2$ minimal submanifold is analytic.
         Hence, 
                     for a
                     doubly spiral minimal product,
 there                   must be  $t_0\in \R$
                     such that
                                     $
                                           \frac{ds_1}{dt}\cdot\frac{ds_2}{dt}\neq0.
                                     $
       By continuity, 
          $\frac{ds_1}{dt}\cdot\frac{ds_2}{dt}\neq0$ locally around  $t_0$.
          So, around $p=G_\gamma(t_0, x, y)=\left(a(t_0) e^{is_1(t_0)} f_1(x), b(t_0) e^{is_2(t_0)} f_2(y)\right)$,
           we can assume that $\frac{ds_1}{dt}\cdot\frac{ds_2}{dt}\neq0$ on $(t_0-\delta, t_0+\delta)$
          and that $G_\gamma\, |_{(t_0-\delta, t_0+\delta)\times U_1\times U_2}$ is an embedding
          where $U_1, U_2$ are some open neighborhoods of $x,y$ respectively. 
   

%
%
\subsection{Local format of doubly spiral minimal products}\label{snd}
%
%
                                            Abbreviate $a(t)$,\, $b(t)$,\, $s_1(t)$,\, $s_2(t)$,\, $f_1(x)$,\, $f_2(y) $ to $a, b, s_1, s_2, f_1, f_2$
                                            and use $'$ for taking derivative in $t$.
                                            Denote the image 
                                           $G_\gamma\left( (t_0-\delta, t_0+\delta)\times U_1\times U_2 \right)$
                                            by $N$.
                                          For bases of tangent and normal spaces of $N$,
                                               we have the following.

                                            \begin{lem}
      Assume that spherical immersions $f_1$ and $f_2$ are both $\mathscr C$-totally real.
           Then,
            locally around 
               $p$ 
              in $\mathbb S^{n_1+n_2+1}$
              with $s'_1s'_2\neq0$,
 the tangent space of  $N$                    has 
                      orthonormal basis
                     \begin{equation}\label{tangent2}
                     \Big\{ 
                      \left(e^{is_1} e_1,0\right),\,\cdots,\, \left(e^{is_1} e_{k_1},0\right),\,
         \left(0,e^{is_2} v_1\right),\,\cdots,\, \left(0,e^{is_2} v_{k_2}\right), \, \tilde E
         \Big\}. 
         \end{equation}
         where 
              \begin{equation}\label{tE}
                                            \tilde E
                    = 
                                            \dfrac{G'_\gamma}
                                            {
                          \left
                          \|
                          G'_\gamma
                          \right
                          \|
                            }
 \text{ \ \ \ \ and\ \ \ \ }
 G'_\gamma={
                                               \bigg(\,
                   \Big  [
                         i a s'_1 
                         +
                         a'
                    \Big  ]
                         e^{is_1}
                         f_1, 
         \,
            \Big [
                         i b s'_2 
                         +
                         b'
             \Big ]
                      e^{is_2}
                         f_2
                                            \bigg) .
                      }
  \end{equation}
\end{lem}        
  \begin{rem}
  When $a, b$ remain constant, basis \eqref{tangent2}  is exactly \eqref{tangent} in Lemma \ref{ltb}.
  \end{rem}
                  \begin{proof}
                  Similar to the proof of Lemma \ref{ltb}.
                  By Proposition \ref{002}, 
                  it can be seen that $G'_\gamma$ is perpendicular to the rest elements in \eqref{tangent2}.
                  \end{proof}

                    For simplicity, unlike the normalized $\eta_1$  in \eqref{eta} 
                    we use
                     \begin{equation}\label{eta_1}
                         \tilde \eta_1= 
                                    \left(
                                          -i b s'_2 e^{is_1} 
                                           f_1, 
         \,
                                       i a s'_1 e^{is_2} 
                                           f_2
                                       \right)
         .
                        \end{equation}
                        Set $\mathcal V=a'b -ab'$
                        and
                        recall that $\Theta=(as'_1)^2+(bs'_2)^2$.
                        Since $s_1's_2'\neq 0$, we use the following instead of previous $\eta_0$
\begin{equation}\label{te0} 
                        \tilde \eta_0=
                          \Big(be^{is_1}f_1, -ae^{is_2}f_2\Big) 
                         -
                          \frac{\mathcal V}
                        {         %
                         \Theta
                          }
                         \Big( ias'_1 e^{is_1}f_1, ibs'_2 e^{is_2}f_2\Big) .
\end{equation}

                     \begin{lem}\label{lnb2}
    Assume that spherical immersions $f_1$ and $f_2$ are $\mathscr C$-totally real.
           Then,
            locally around 
               $p$ 
              in $\mathbb S^{n_1+n_2+1}$
                            with $s'_1s'_2\neq0$,
         the normal space of $N$ has orthogonal basis 
        \begin{equation}\label{2normal}
             \Big\{  
                             \left(e^{is_1} \sigma_2,0\right),\,\cdots,\, \left(e^{is_1}  \sigma_{n_1-k_1},0\right),\,
                   \left(0, e^{is_2}  \tau_2\right),\,\cdots,\, \left(0, e^{is_2} \tau_{n_2-k_2}\right), \, \tilde \eta_1,\, \tilde\eta_0
           \Big\}
            .
        \end{equation}
   \end{lem}
        \begin{rem}
        When $a$ and $b$ remain constant, up to normalizations of $\tilde \eta_1$ and $\tilde \eta_0$,
        this lemma goes back to Lemma \ref{lnb}.
        \end{rem}
                  \begin{proof}
                                      By  \eqref{select} and Proposition \ref{002},
                                      $\tilde \eta_1$
                                      and $\tilde \eta_0$
                                      are perpendicular to rest elements in \eqref{2normal}.
                                      Moreover, it is easy to see from \eqref{eta_1} and \eqref{te0} that 
                                      $\tilde \eta_1$ is orthogonal to $\tilde \eta_0$ as well.
                                      
                                    To finish the proof, 
                                    we need to verify that elements in \eqref{2normal} are perpendicular to those in \eqref{tangent2}.
                                    Again, by  \eqref{select} and Proposition \ref{002},
                                    elements except $\tilde \eta_1$  and $\tilde \eta_0$ in \eqref{2normal}
                                    are orthogonal to all elements of \eqref{tangent2},
                                    and it is easy to see that 
                                    $\tilde \eta_1$  and $\tilde \eta_0$
                                    are perpendicular to elements except $\tilde E$ in \eqref{tangent2}.
                                   Finally, direct computations
                                  show that
                                     $\left<\tilde \eta_1,  G'_\gamma\right>=0$ 
                                     and that
                                     $\big<\tilde \eta_0,  G'_\gamma\big>
                                        =  \left<\Big(be^{is_1}f_1, -ae^{is_2}f_2\Big) 
                                            -
                          \frac{\mathcal V}
                        {         %
                         \Theta
                          }
\Big( ias'_1 e^{is_1}f_1, ibs'_2 e^{is_2}f_2\Big),  G'_\gamma \right>=\mathcal V-  \mathcal V=0
$.
                  \end{proof}

              Note that \eqref{ddG1} and  \eqref{ddG2} are replaced  by
                    \begin{equation}\label{tddG1}
                  \frac{d^2 G_\gamma}{dt^2}
               \, \,   =\,\,
                          \mathcal D_{\left\|G'_\gamma\right\| \tilde E}{\left\|G'_\gamma\right\| \tilde E}
                 \, \, =\,  \,     \left\|G'_\gamma\right\|^2 \, \mathcal D_{\tilde E}\tilde E
                         +\ \left\|G'_\gamma\right\| 
                         \left( 
                             \mathcal D_{\tilde E}\left\|G'_\gamma\right\|
                             \right)\, \tilde E,
     \end{equation}     
         and
\begin{eqnarray}
       \ \ \ \ \           \frac{d^2 G_\gamma}{dt^2}
                 & =&
                     \bigg(
                         %
                           \Big[
                           i
                           (2a's'_1
                           +as''_1)
                           +a''
                           -a\cdot(s'_1)^2
                           \Big]
                            e^{is_1}
                           f_1
                           ,\,
 \label{tddG2}      \\ & &
            \ \ \ \ \ \ \ \    \ \ \ \  \ \ \   \ \ \   \ \ \ \ \ \ \ \ \ \ \ \    %
                           \Big[
                           i
                           (2b's'_2
                           +bs''_2)
                           +b''-b\cdot(s'_2)^2
                           \Big]
                           e^{is_2}
                           f_2
                       \bigg).
                       \nonumber
\end{eqnarray}
%
             The second fundamental form of $N$ can be described as follows.

   \begin{lem}\label{tL1}
         At $p$,
         it follows that, with respect to the rotated tangent basis \eqref{tangent2} of $N$,
                 \begin{equation*}\label{tR1}
                                                                    A_{\left(e^{is_1} \sigma_{k},0\right)} 
                                                                          =
                                                                          \begin{pmatrix}
                                                                          \frac{1}{a(t)}A
                                                                          _{\sigma_{k}} &  & * \\
                                                                           &O & * \\
                                                                      *     & * & 0
                                                                           \end{pmatrix}
                                                                           \, ,
%
  \text{\ \ \ \ \ \ \ \ \ \ \ \ \ \ \ \ \ \ \ \ \ \ \ \ \ \ \ \  \ \ \ \ \  } 
                 A_{\left(0,  e^{is_2}\tau_{l}\right)} 
                    =
                      \begin{pmatrix}
                      O &  & *\\
                      & \frac{1}{b(t)}A_{\tau_{l}} & * \\
                    *  & * & 0
                      \end{pmatrix}
%
                 , 
\end{equation*}
\begin{equation*}\label{tR3} 
 \   \ \  \ \     A_{\tilde \eta_1}
                         =
                                \begin{pmatrix}
                                -\frac{bs'_2}{a}A_{\sigma_1}  &  & *\\
                                & \frac{as'_1}{b}A_{\tau_1} & *\\
                                * &  * & \tilde \circledast
                                \end{pmatrix},
\text{\ \ \, \, \ \ \ \ }                        
     A_{\tilde \eta_0}
                         =
                                \begin{pmatrix}
                                -\frac{b}{a}
                                 I_{k_1}
                                -\frac{\mathcal Vs'_1}{\Theta}
                               A_{\sigma_1}  &  & *\\
                                & 
                                \frac{a}{b}I_{k_2}
                                 -\frac{\mathcal Vs'_2}{\Theta}
                                A_{\tau_1} & *\\
                                * &  * & \tilde \boxast
                                \end{pmatrix}, \ \ \ \ \text{ \ \ \ \ }
\end{equation*}
           %
           %
           %
           %
               %
               %
               where
                           $2\leq k\leq n_1-k_1$, 
                            $2\leq l\leq n_1-k_1$,
                            $\mathcal V=a'b-ab'$,
                            $\Theta=(as'_1)^2+(bs'_2)^2$,
                            $\tilde \circledast=\big<\mathcal D_{\tilde E}\tilde E, \tilde \eta_1\big>$
               and
                           $\tilde \boxast=\big<\mathcal D_{\tilde E}\tilde E, \tilde \eta_0\big>$ 
                           with
\begin{eqnarray}
                           \left\|G'_\gamma\right\|^2 \cdot \tilde \circledast
                            &=&
                                    -2 s'_1 s'_2 b^2 \left(\frac{a}{b}\right)'
                                    -ab{s'_2}^2 \left(\frac{s'_1}{s'_2}\right)', \ \ \ \ \ \ \ \ \ \ \ \ \ \ \ \,
\label{circled}
\end{eqnarray}
                             and
\begin{eqnarray}
               \ \ \ \ \ \  \ \ \   \ \ \    \ \     \ \ \ \ \         \left\|G'_\gamma\right\|^2 \cdot \tilde \boxast
                    &=&
                                       \Bigg\{
                                          \Big[
                                          a''-a(s'_1)^2
                                          \Big]
                                          b
                                                                                 -
                                                       \Big[
                                                       b''-b(s'_2)^2
                                                       \Big]
                                                       a
                                                                                 \bigg\}
\label{box}                    \\&&         \ \ \ \ \ \     \ \ \ \ \                                          -
                                       \frac{\mathcal V}{\Theta}
                                           \bigg\{
                                              \Big(2a's'_1+as''_1\Big)as'_1+\Big(2b's'_2+bs''_2\Big)bs'_2
                                          \bigg \} .
\nonumber
\end{eqnarray}
\end{lem}
                \begin{proof}
               Note that the ideas for shape operators in \S \ref{SMP} also work with varying $(a(t), b(t))$.
                                       Elements except in the last row or last column of  $A_{\left(e^{is_1} \sigma_{k},0\right)} $
                                                    $A_{\left(0,  e^{is_2}\tau_{l}\right)} $
                                                    and 
                                                    $A_{\tilde \eta_1}$
                                                    are clear by the proof of Lemma \ref{L1},
                                                                 and so are those except in the last row or last column  of $A_{\tilde \eta_0}$ by Lemma  \ref{L2}.
                                                    To determine elements in the bottom right corner of them,
                                                    we compute 
                                                    \begin{equation}\label{tocompute} 
                                                            \left< \frac{d^2}{dt^2}G_\gamma,\, \left(e^{is_1} \sigma_{k},0\right)\right>,\,\,
                                                             \left< \frac{d^2}{dt^2}G_\gamma,\, \left(0,  e^{is_2}\tau_{l}\right)\right>,\,\,
                                                              \left< \frac{d^2}{dt^2}G_\gamma,\, \tilde \eta_1\right>,\,\,
                                                               \left< \frac{d^2}{dt^2}G_\gamma,\, \tilde \eta_0\right> .
                                                      \end{equation}

                                                            By \eqref{tddG1}, \eqref{tddG2} and Proposition \ref{002},
                                                            it follows that the first two quantities of \eqref{tocompute} vanish
                                                            and so do the elements in the bottom right corner of 
                                                             $A_{\left(e^{is_1} \sigma_{k},0\right)} $
                                                             and
                                                    $A_{\left(0,  e^{is_2}\tau_{l}\right)} $.
                                                    By \eqref{tddG1}, \eqref{tddG2} and Proposition \ref{002}, the third term of \eqref{tocompute}
reads
             $$
             \left\|G'_\gamma\right\|^2 \cdot \tilde \circledast=
                             -2 a'b s'_1 s'_2
        -abs''_1s'_2 
                             + 2 a b' s'_1 s'_2
        +abs'_1s''_2
                             =
                                    -2 s'_1 s'_2 b^2 \left(\dfrac{a}{b}\right)'
                                    -ab{s'_2}^2 \left(\dfrac{s'_1}{s'_2}\right)',
              $$
                           while the last gives
                                                $$
                                              \left\|G'_\gamma\right\|^2 \cdot\, \tilde \boxast=
                                       \bigg\{
                                          \Big[
                                          a''-a(s'_1)^2
                                          \Big]
                                          b
                                                                                 -
                                                       \Big[
                                                       b''-b(s'_2)^2
                                                       \Big]
                                                       a
                                                                                 \bigg\}
                                                                                 -
                                       \frac{\mathcal V}{\Theta}
                                           \bigg\{
                                              \Big(2a's'_1+as''_1\Big)as'_1+\Big(2b's'_2+bs''_2\Big)bs'_2
                                          \bigg \}\, .
                                                $$            
             Hence, we finish the proof.
              \end{proof}

                Note that trace$(A_{\tilde \eta_1})=\tilde \circledast$ vanishes if and only if 
                        \begin{equation}\label{Coccurs}
                       2\bigg(\dfrac{a}{b}\bigg)^{-1}\bigg(\frac{a}{b}\bigg)'
                       +
                       \left(\dfrac{s'_1}{s'_2}\right)^{-1}\left(\frac{s'_1}{s'_2}\right)'=0,
                        \end{equation}
                        which, by integration, leads to
                                    \begin{equation}\label{circledno} 
                                   \frac{b^2}{a^2}
                                    =C \cdot\dfrac{s'_1}{ s'_2}
                                   \text{\ \ \ \  for some nonzero constant \ }C\in \R.
\end{equation}

                          Since \eqref{box} looks quite complicated and our concern is purely geometric
                          (whether or not an immersion is minimal is independent of choice of parametrization), 
                                 without loss of generality, 
                                           we  employ arc parameter $s$ of 
                                              curve $(a(t),\, b(t))$ in $\mathbb S^1\subset \C^1$ instead of $t$
                                              at the very beginning
                                                                                                    and use $\dot{\ }$ for derivative in $s$.
                                                    We can  assume that $a(s)=\cos s$ and $b(s)=\sin s$.
                                                    \footnote
                                                    {
                                                    Although it seems that we take a peculiar choice of parametrization in analysis, 
                                                    actually we do not lose any information (of solution $\gamma$ in local) from the geometric viewpoint.
                                                    In the current situation, 
                                                    curve $(a(t), b(t))$ has nonvanishing velocity in local.
                                                     Any immersed (connected piece of) curve has arc parameters
                                                     and,
                                                     compared with the canonical anti-clockwise arc parameter $s$ of $\mathbb S^1$
                                                     measuring oriented 
                                                     length from 
                                                     starting point $(1,0)$,
                                                     all arc parameters have to be 
                                                     $\pm s$ plus some constant.
                                                     }
                                                    This preference of parameter will dramatically reduce the difficulty of solving 
                                                    trace$(A_{\tilde \eta_0})=0$.
By $\dot a =-b$, $\dot b= a$ and \eqref{box},
                            $\left\|G'_\gamma\right\|^2\cdot$ trace$(A_{\tilde \eta_0})=0$ means
\begin{eqnarray}
       & & \left(
         1+ (a\dot s_1 )^2
         + (b\dot s_2 )^2
         \right)
         \left(-k_1\frac{b}{a}+k_2 \frac{a}{b}\right)
 \nonumber        \\
& &\ \ \ \ \ \ \ \ \ \ \ \ \ \ \ \ \ \ \ =
         ab
         \left[
         (\dot s_1)^2-(\dot s_2)^2
         \right]
 -
\frac{
             -2ab(\dot s_1)^2
             +a^2\dot s_1 \ddot s_1
             +2ab(\dot s_2)^2
             +b^2 \dot s_2 \ddot s_2
             %
}
         {(a\dot s_1)^2+(b\dot s_2)^2} .
          \label{boxno}
\end{eqnarray}
              Further simplification transforms \eqref{boxno} to
\begin{equation}\label{dotG}
        %
         -k_1\frac{b}{a}+k_2 \frac{a}{b}
                  -
          \frac{ab\left[
         (\dot s_1)^2-(\dot s_2)^2
         \right]}{(a\dot s_1)^2+(b\dot s_2)^2}
=-
                \frac{\dot \Theta}{2\Theta  \left(
         1+\Theta
         \right)}
             %
\end{equation}
 where $\Theta=(a\dot s_1)^2+(b\dot s_2)^2$.
           Relation \eqref{circledno} now gives us
\begin{equation}
              \left(\dot s_2\right)^2=\frac{C^2 a^4}{b^4}\left(\dot s_1\right)^2,
\end{equation}
           and hence, 
   \begin{equation}
                     \frac{
                     ab  
         \left[
         (\dot s_1)^2-(\dot s_2)^2
         \right]
         }
         {(a\dot s_1)^2+(b\dot s_2)^2}
=\frac{b^4-C^2 a^4}{ab(b^2+C^2a^2)}\, .
\end{equation}         
                                     So, by integrating double of \eqref{dotG}, we get
\begin{equation}\label{Gamma}
                                     \exp  {  \left\{2\displaystyle \int
                                             k_1\frac{b}{a}-k_2 \frac{a}{b}+ \frac{b^4-C^2 a^4}{ab(b^2+C^2a^2)}
                                             %
                                               \right\}
                                               }
                                                    =\dfrac{ \tilde C\,  \Theta}{1+\Theta}  ,
                                                    \text{\ \  for some  \ } \tilde C\in \R_{>0} .
\end{equation}
   As an antiderivative of 
            $
            2\frac{b^4-C^2 a^4}{ab(b^2+C^2a^2)}
            $
            is
            $
            \ln
                     \left|
                     1+(C^2-1)\cos^2 s
                     \right|
            -\ln
                          \big|
                          \cos^2 s \sin^2 s
                           \big|
                          \, 
            $,
after absorbing a positive constant from the left hand side of \eqref{Gamma} into $\tilde C$, 
          we obtain
 explicitly
\begin{equation}\label{Gexp}
            \displaystyle 
            \Theta
            =\dfrac
            { 
                   {
                      1+\big(C^2-1\big)\cos^2 s
                   }
              }
            {\tilde C  
                 \Big(\cos s\Big)^{2k_1+2}
                  \Big( \sin s\Big)^{2k_2+2}
                  \,
                   -
                   {
                      1-\big(C^2-1\big)\cos^2 s
                   }
            }
            \, .
\end{equation}
                                By \eqref{circledno},
                             up to $\pm$ sign    the expression \eqref{Gexp} produces
                                 
                \begin{equation}\label{dotsss}
                 \begin{pmatrix}
       \dot s_1\\
     \dot s_2
       \end{pmatrix}
       = 
                             \sqrt{\frac{1}{ 
                                   {\tilde C  
                 \Big(\cos s\Big)^{2k_1+2}
                  \Big( \sin s\Big)^{2k_2+2}
                  \,
                   -
                   {
                      1-\big(C^2-1\big)\cos^2 s
                   }
            }
            }
 }
           \,     
       \begin{pmatrix}
       \tan s\\
     C\cot s
       \end{pmatrix}.
                \end{equation}

     %
%
%
%
                 As a result, 
                        we  gain a local version of
                        part \ding{173} of Theorem \ref{main} about doubly spiral minimal products of varying magnitudes
                  in
                         the following.
\begin{thm}\label{P2}
                   There is a two-dimensional family of doubly spiral minimal local patches 
                   of varying magnitudes
                   depending on $C$ and $\tilde C$ 
                according to     
                   \eqref{Gexp} and \eqref{dotsss} 
into the target sphere. 
                  %
\end{thm}

            \begin{proof}
            Given any $C\neq 0\in \R$, we can choose  sufficiently large $\tilde C$ so that 
            $\left\{\tilde C> P_C(s)\right\}\neq \emptyset$
            where
            \begin{equation}\label{0Cfun}
      P_C(s): =\dfrac{ 1+\big(C^2-1\big)\cos^2 s}{
                  \Big(\cos s\Big)^{2k_1+2}
                  \Big( \sin s\Big)^{2k_2+2}
                  }\, .
\end{equation} 
                   Then, 
                   whenever the denominator of $\Theta=(a\dot{s}_1)^2+(b\dot{s}_2)^2$ in \eqref{Gexp} is positive,
                   \eqref{dotsss} has meanings
                   and  accordingly
                   induces a spiral minimal  local patch by a solution curve segment.
            \end{proof}

     {\ }
     
In the coming subsection we focus on local patch of the singly spiral case.
                    %
%
\subsection{Local format of singly spiral minimal products}\label{ind1} %
%
%

    Note that \eqref{circledno} works when $s'_1s'_2\neq 0$
    and that $C\neq 0$ in
                   $
                      b^2 {s_2'}
                                    =C {a^2}{s_1'}.
                   $
                              However, the value of $P_C(s)$ in \eqref{0Cfun} decreases in $C^2$.
                   Namely, if $\tilde C>P_{C_0}(s)$ over some open interval $(\alpha, \beta)$,
                   then $\tilde C>P_{C}(s)$ also holds 
                   over $(\alpha, \beta)$
                   for $C={\tt t}C_0$ where ${\tt t}\in [0,1)$.
                    So,
                    with a fixed $\tilde C$ and a common starting point $(\cos \alpha, \sin \alpha)$ with $s_1(\alpha), s_2(\alpha)\in \R$,
                    letting ${\tt t}\downarrow 0$
                    we can see that
                    the solution curves over $(\alpha, \beta)$ converge
                    and correspondingly
                    those doubly spiral minimal patches limit to 
                    a 
                    singly spiral minimal patch with  $s_2\equiv s_2(\alpha)$.
                      Now, 
                         \eqref{dotG} 
                       becomes
                        \begin{equation}\label{dotG2}
        %
         -(k_1+1)\frac{b}{a}+k_2 \frac{a}{b}
=-
                \frac{\dot \Theta}{2\Theta  \left(
         1+\Theta
         \right)}\, 
             %
\end{equation}
 in the arc parameter $s$ of $(a(t), b(t))$ with $(a(t), b(t))=(\cos s, \sin s)$ and 
 $\Theta=(a\dot s_1)^2$.
 Further, integrating double of \eqref{dotG2} produces
\begin{equation}\label{dotG3}
        %
         \Big(\cos s\Big)^{-2k_1-2}
         \Big( \sin s\Big)^{-2k_2}
         = \dfrac{\tilde C\, \Theta}{1+\Theta}
         \end{equation}
and thus
\begin{equation}\label{0dotG3}
\big(\dot s_1\big)^2
       =
         \dfrac{ 1}{\tilde C\, \Big(\cos s\Big)^{2k_1+4}
         \Big( \sin s\Big)^{2k_2}-\cos^2 s}
  \text{ \ \  for some } \tilde C >0\, .       
\end{equation}
%

Therefore, we obtain  a local version of
                        part \ding{174} in Theorem \ref{main} regarding singly spiral minimal products of varying magnitudes
                  in
                         the following.

\begin{thm}\label{P3}
With  vanishing $s'_2$, 
            there is a one-dimensional family of  singly spiral minimal local patches 
                of varying magnitudes
                   depending on $\tilde C$ 
                according to     
                   \eqref{0dotG3}
                   into the target sphere.
\end{thm}
                        \begin{proof}
                        Similar to that of Theorem \ref{P2}.
                        When $k_2>0$, we take $\tilde C$ larger than  the minimal value $m_{0}=P_0(s_{0})$ of $P_0(s)$ 
                         described in \eqref{P(s)} of \S \ref{P}.
                        When $k_1\geq k_2=0$,
                        we take $\tilde C>1$.
                        Plug this $\tilde C$ into \eqref{0dotG3} for a singly spiral minimal local patch 
                        by a solution curve segment of varying magnitudes.
                        \end{proof}
\begin{rem}\label{P3R}
                     With vanishing $ s_2'$,
                     by checking our construction
                     there is no need to require $f_2$ to be $\mathscr C$-totally real
                      and then $k_2$ is allowed to take $0, 1,\cdots, n_2$.
                      \end{rem}

\begin{rem}\label{P3R0}
If one prefers to follow arguments in \S \ref{snd},
         then there are two cases.
          When $k_2<n_2$ with $s_2\equiv 0$,
        the normalization of  $\tilde \eta_1$ is $(0, e^{is_2}  \tau_1)$
        and
        $\tilde \circledast$
        vanishes automatically
        which results in \eqref{dotG2}.
         When $k_2=n_2$,
         the normal basis \eqref{2normal} needs to be replaced by
         $$
             \Big\{  
                             \left(e^{is_1} \sigma_2,0\right),\,\cdots,\, \left(e^{is_1}  \sigma_{n_1-k_1},0\right),\,
                   \tilde\eta_0
           \Big\}
            .
         $$
        Thus, the system of minimal surfaces now turns out to be the single equation \eqref{dotG2}.
           \end{rem}

                                        {\ }
                                        
                              \section{Global Construction of Spiral  Minimal Products of Varying Magnitudes}\label{S4}
                              In this section, we assemble suitable local patches gained in \S \ref{GSMP} together.
                              A key ingredient is to apply the beautiful extension result of minimal submanifolds by Harvey and Lawson \cite{HL0} 
                             which leads to a rotational reflection principle for our situation, see Lemma \ref{analytic} and Remark \ref{rrp}.
                              %
                              %
                              %
                              
                              %
                              %
                              %
                              In \S \ref{GSMP},
                              we construct spiral minimal local patches 
                              by 
                              \eqref{dotsss} 
                              for $C\neq 0$ and 
                              \eqref{0dotG3} when $C=0$.
                              More precisely, once $C$ is fixed, 
                              in order for $\Theta>0$
                              with an allowed $\tilde C$
                              one can figure out a maximally defined connected interval $\Omega_{C, \tilde C}$ of $s$ in $(0, \frac{\pi}{2})$
                              for
\begin{equation}\label{4.1}
                                {\tilde C  
                 \Big(\cos s\Big)^{2k_1+2}
                  \Big( \sin s\Big)^{2k_2+2}
                  \,
                   -
                   {
                      1-\big(C^2-1\big)\cos^2 s
                   }
            }>0.
\end{equation}
                                    When $C\neq 0$ or $C=0$ but $k_2\neq 0$,
                                    the $\Omega_{C, \tilde C}=\big(z^{C, \tilde C}_L, \,    z^{C, \tilde C}_R \big)\Subset \left(0,\frac{\pi}{2}\right)$;
                                     when $C=k_2=0$,
                                 the  $\Omega_{\, 0, \tilde C}=(0, z^{0, \tilde C}_R)\subsetneq (0,\frac{\pi}{2})$.
 \begin{figure}[h]
	\centering
	\begin{subfigure}[t]{0.48\textwidth}
		\centering
		\includegraphics[scale=0.5]{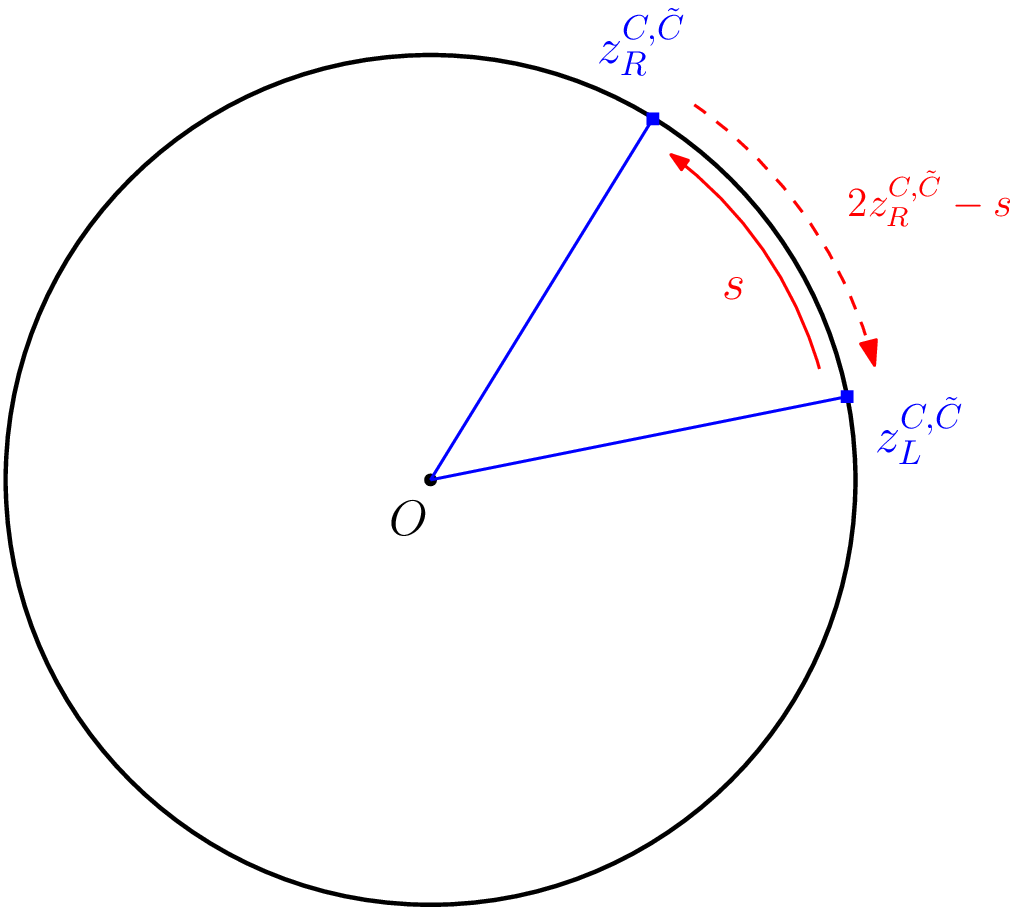}
                              \captionsetup{font={scriptsize}} 
                               \caption{$C\neq 0$ or $C=0$ but $k_2\neq 0$ \ \ \ \  \ \ \ \ }
                               \label{fig:2a}
	\end{subfigure}
	\begin{subfigure}[t]{0.48\textwidth}
		\centering
	 \includegraphics[scale=0.5]{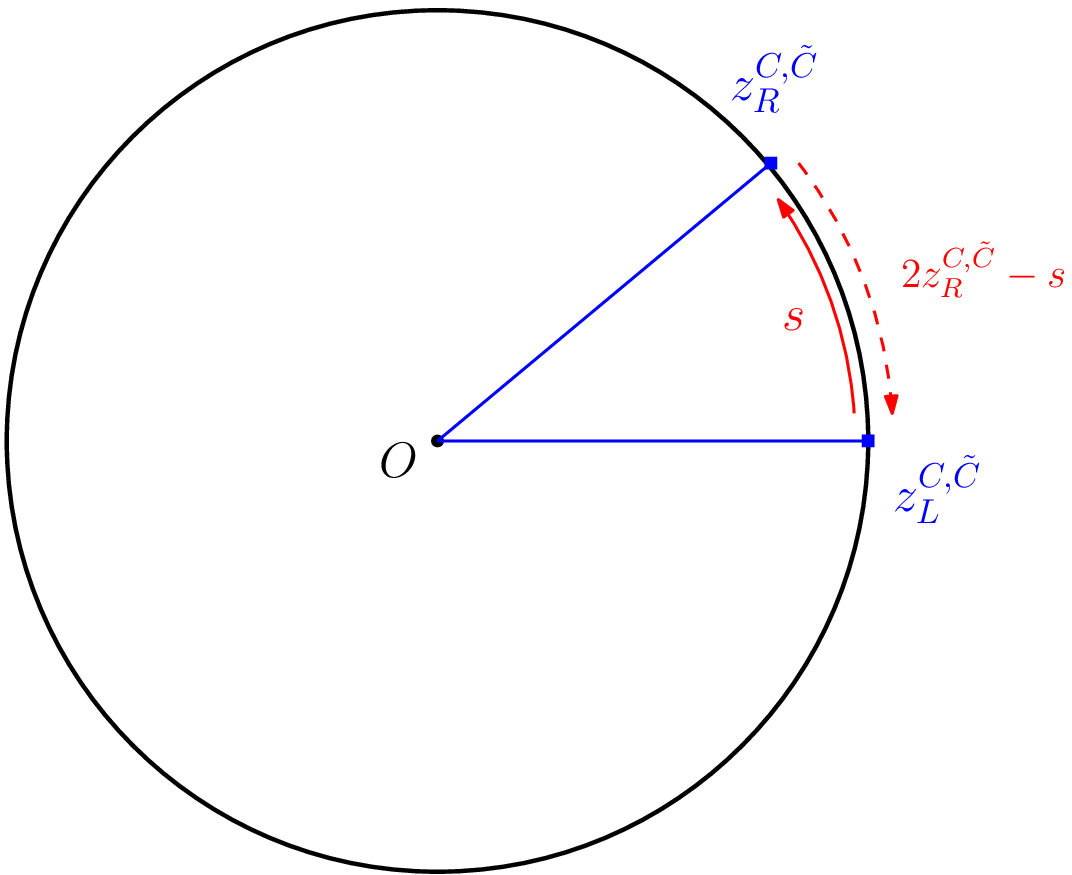}
                           \captionsetup{font={scriptsize}} 
                               \caption{$C=k_2=0$ 
                                \ \ \ \ \ \ \ \ \ \ }\label{fig:2c}
	\end{subfigure}
	 \caption{Basic domains of different pendulum situations}\label{pend}
\end{figure}
                  Now we can flip the basic domain $\Omega_{C, \tilde C}^0=\Omega_{C, \tilde C}$ repeatedly to pave the entire real line as follows.
                  Denote by $\Omega_{C, \tilde C}^{1}$ 
                  the interval $\big(z^{C, \tilde C}_{L, 1},\, z^{C, \tilde C}_{R, 1}\big)=\big(z^{C, \tilde C}_R,\, 2 z^{C, \tilde C}_R-z^{C, \tilde C}_L\big)$
                             which is the image of $\Omega^0_{C, \tilde C}$ under the reflection $\iota^R_0: s\mapsto 2 z^{C, \tilde C}_R-s$.
                                   We define $\iota^R_k$ inductively for $\Omega_{C, \tilde C}^{k}$
                                    and  $\iota^L_{k+1}$ for $\Omega_{C, \tilde C}^{k+1}$ to be $\big(\iota^R_k\big)^{-1}$  for every $k\in \mathbb N$.
                              
                                      Before proceeding, 
                                      we need to mention a bit more.
                                       At ending point(s) of $\Omega_{C, \tilde C}^0$ in the interior of $(0,\frac{\pi}{2})$, 
               the slope of \eqref{4.1}  is nonzero.
                              So, around each of them, up to finite constants, $\dot s_{1}$ and $\dot s_{2}$ dominated by \eqref{dotsss}  
                              behave like $\frac{1}{\sqrt x}$ around $x=0$.
         Therefore, 
         $$ 
         \lim_{s\downarrow z^{C, \tilde C}_L}s_{l}(s)
         \text{ and}
         \lim_{s\uparrow z^{C, \tilde C}_R} s_{l}(s)
         \text{ exist of finite values} \
         L_{l}^{0,-}
         \text{ and }
         L_{l}^{0,+}
         \text{ respectively for } l=1,2.
         $$
         When $k_2=0$ and $z_L^{0, \tilde C}=0$, the same finiteness of the limit at $z_L^{0, \tilde C}=0$ also holds. 
             As a result, 
             we eliminate
             the possibility that a spiral minimal patch may have infinitely many windings as $s$ approaches $\p\Omega_{C, \tilde C}^0$. 
                 Denote such a solution curve segment for $G_\gamma$ to be minimal over the basic domain $\Omega_{C, \tilde C}^0$ by $\gamma^0$.
                 
                 We discover 
                that the $\gamma^0$-spin can survive through the ending points of $\Omega_{C, \tilde C}^0$.
                     To see this, 
                            define $(a(s), b(s))=\left(a\big(\iota^L_1(s)\big), b\big(\iota^L_1(s)\big)\right)$
                            for $s\in \Omega_{C, \tilde C}^{1}$
                            and 
                            define
\begin{equation}\label{Dext}
                                                       s_l(s)
                                                               =L_{l}^{0,+}+\mathlarger{\int}_{\iota^L_1(s)}^{z^{C, \tilde C}_{R}} \dot{s_l}\left(\zeta\right)\, d\zeta\, , \ \ \text{ for } l=1,2,
\end{equation}
                                                       where $\dot{s_l}(\cdot)$ is defined by \eqref{dotsss} over $\Omega_{C, \tilde C}^{0}$.
                            Then, on $\Omega_{C, \tilde C}^{0}$,
                            $$
                            \frac{d s_l(s)}{ds}=\dot{s_l}\left(\iota^L_1(s)\right)
                            $$
                            and
                                  $\gamma^1$ 
                                  satisfying 
                                  \eqref{dotsss}
                                    over $\Omega_{C, \tilde C}^{1}$
                            with initial    data                    $
                            a\big(z^{C, \tilde C}_R\big), b\big(z^{C, \tilde C}_R\big)
                            $
                            and 
                            $s_{l}\big(z^{C, \tilde C}_R\big)=   L_{l}^{0,+}$
                                       is defined.
Up to an ambient isometry (i.e., a composition of suitable separate rotations in two $\C$-components of $\C\oplus\C$), 
one can assume that both $    s_1\big(z^{C, \tilde C}_{R}\big)=L_{1}^{0,+}$ and 
                   $    s_2\big(z^{C, \tilde C}_{R}\big)=L_{2}^{0,+}$ vanish.
                                                               Then, 
                                                               for any $0<\delta< z^{C, \tilde C}_{R}-z^{C, \tilde C}_{L}$,
                                                               the extension \eqref{Dext}
                                                               implies that
                                                               $
                                                              s_l\big(z^{C, \tilde C}_R-\delta\big)
                                                                 +
                                                        s_l\big(z^{C, \tilde C}_R+\delta\big)=0.
$
Now, $\gamma^1$ is exactly the complex conjugate mirror of $\gamma^0$. 
                                     In this way, we extend the spiral minimal open patch on $\Omega_{C, \tilde C}^{0}$ to its adjacent domain 
                                     $\Omega_{C, \tilde C}^{1}$ in a $C^1$ manner along their common
                                     border ${\tt A}=  \left(a(s)\, e^{i s_1(s)}f_1(M_1),
                                 \, b(s)\, e^{i s_2(s)} f_2(M_2)\right)|_
                                       {s=z^{C, \tilde C}_{R}}$.             
                                     
                                     Geometrically in $\mathbb S^3$, what we have done is in fact taking a $\pi$-rotation
                                     of the solution segment $\gamma^0$ 
                                     about the direction by 
                                     $
                                     q=\left(
                                                     a(s)\, e^{i s_1(s)}, 
                                                     \, 
                                                     b(s)\, e^{i s_2(s)}
                                       \right)\, \big|_
                                       {s=z^{C, \tilde C}_{R}}
                                     $
                                     and hence
                                     the resulting open curve segment $\gamma^1$ corresponding to $\Omega_{C, \tilde C}^{1}$  
                               is sewed with $\gamma^0$ at $q$ in a $C^1$ manner.
                               Denote the $C^1$ joint curve segment by $\tensor[^0]{\gamma}{^1}$.

                               \begin{lem}\label{analytic}
                               $\gamma^1$ is sewed with $\gamma^0$ analytically through $q$.
                                       \end{lem}
                                       \begin{proof}
                                     Recall that \eqref{dotsss} is derived from the system of minimal surface equations.
                                     So both $G_{\gamma^0}$ and $G_{\gamma^1}$ are minimal.
                                    As $G_{\tensor[^0]{\gamma}{^1}}$
                                     is a   $C^1$ submanifold of dimension $k_1+k_2+1$ in the target sphere,
                               and
                              $                    
                                 {\tt A}
                                                                        $
                              has local finite Hausdorff $(k_1+k_2)$-measure,
                               we can apply 
                                    the powerful extension result $-$ Theorem 1.4 in \cite{HL0}
                                     to deduce that $G_{\tensor[^0]{\gamma}{^1}}$ is minimal 
                                     over the entire interior part of $\overline {\Omega_{C, \tilde C}^{0}\bigcup \Omega_{C, \tilde C}^{1}}\times M_1\times M_2$.
                                Since the ambient space is analytic, 
                                by Allard \cite{Allard} or Morrey \cite{M1, M2},
                                the $C^1$ patch $G_{\tensor[^0]{\gamma}{^1}}$ turns out to be analytic, and in particular it is analytic through ${\tt A}$.
                                     Namely,
                                     $\gamma^0$ and $\gamma^1$ are  joint together through $q$ analytically. 
                                       \end{proof}
                               
                               \begin{rem}\label{rrp}
                               A punch line for $\tensor[^0]{\gamma}{^1}$ to be $C^1$ differentiable
                               is that $a$ and $b$ attain their extremal values simultaneously at $z^{C, \tilde C}_{R}$.
                               This rotational reflection principle in some sense generalizes  Corollary 1.1 of \cite{HL0}.
                          Here we  replace ambient space reflection in \cite{HL0} by the rotational reflection on the generating curve segment $\gamma^0$.
                               \end{rem}

                                     \begin{thm}\label{ltog}
                                     Every spiral minimal local patch in Theorems \ref{P2} and \ref{P3}
                                     can develop to a globally defined spiral minimal product.
                                     Hence, Theorem \ref{main} is true.
                                                                          \end{thm}
                                                        \begin{proof}
                                                        Via repeated rotational reflections, $\gamma^0$ over the basic domain $\Omega^0_{C, \tilde C}$ extends itself geometrically to a solution curve (see Figure \ref{fig:000})
                                                        $$\gamma=\bigcup_{\cdots\,  -1, 0, 1, \cdots \, \,}
                                                          \overline{\,\gamma^l\,} \, \, \subset \mathbb S^3$$
                                                        with infinite length in both directions
                                                        and Lemma \ref{analytic} guarantees corresponding pieces assembled analytically.
                                                        
                                                        Part \ding{172} of Theorem \ref{main} has been shown in \S \ref{poc}.
                                                        Now, the global constructions for \ding{173} and \ding{174} have been confirmed. 
                                                        Hence, Theorem \ref{main} holds.
                                                        \end{proof}
                                                         \begin{rem}
                              Here we consider $\Omega_{C, \tilde C}^{l}$ as the domain by $l$-th reflection (toward positive infinity) of the basic domain 
                              $\Omega_{C, \tilde C}^{0}=\Omega_{C, \tilde C}$ and
                              $\gamma^l$ to be the                   corresponding $l$-th rotational reflection image 
                              (toward right hand side) of $\gamma^0$ in  $\mathbb S ^3$.
                              \end{rem}
                                         \begin{figure}[h]               
                                                        \includegraphics[scale=0.8]{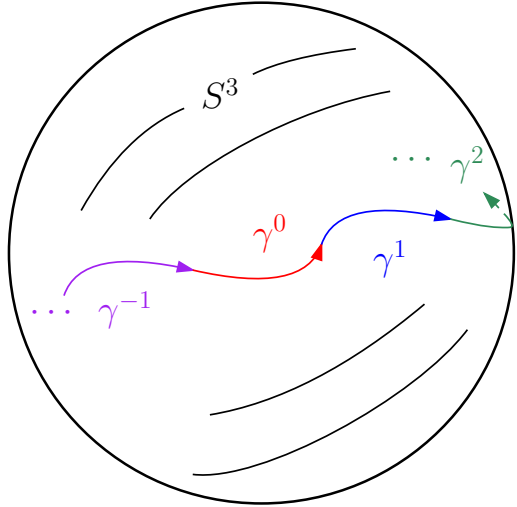}
                            \caption{Illustration of a solution curve $\gamma$ in $\mathbb S^3$} 
                          \label{fig:000}
                             \end{figure}

                              \begin{rem}\label{AM}
                              In view of $(a(t), b(t))$, the generating curve $\gamma$ performs a pendulum back and forth
                              (with a full round of the pendulum over $\overline{\Omega_{C, \tilde C}^{k}\bigcup \Omega_{C, \tilde C}^{k+1}}$,
                              cf. Figure \ref{pend})
                              with an amplitude determined by $C$ and $\tilde C$.
                              Moreover, 
                              $C$ in \eqref{circledno} actually serves as a fixed ratio of projected signed areas 
                              (to the two $\C$-components of $\C\oplus \C$ respectively)
                caused by instant movement  of segment $\overrightarrow{0p}$
               where 
               $p\in \gamma \subset \mathbb S^3$.
               This also means that globally $C$ stands for the ratio of signed angular momenta in the two complex components 
               and the quantity is independent of the choice of variables of $\gamma$ as its first appearance in \eqref{circledno}.
                              \end{rem}%
          \begin{rem}\label{Rork}
                          We  remark that combinations of separate rotations in two $\C$-components of $\C\oplus\C$
                          for group actions on solution curves of \eqref{dotsss}.
                          So, in the sequel,
                          we only need to study solution curves $\gamma$ with starting point of $\gamma^0$ given by $(\cos z^{C, \tilde C}_L, \sin z^{C, \tilde C}_L)\in \mathbb S^1\subset \mathbb S^3\subset \C\oplus \C$.
                      Once these curves are known, solution curves with other starting points
                      can be generated by the group actions.
           \end{rem}%

                                        {\ }

                              \section{Situation for $C=0$}\label{C=0}

In this section, we consider singly spiral minimal products, which seem not having beed discussed elsewhere before,
except the construction of two-dimensional minimal tori (in $\mathbb S^3$) 
occurred in Theorem 1.4 of \cite{Brendle} (informed  to Brendle by Kusner)
                                    with $M_1=$ $\{point_1\}$  and $M_2=\mathbb S^1$.
                                    The minimal tori were given  by
\begin{eqnarray}
            F(\tilde s, t)
            &=&
                       \left(
                       \, r(t)\cos t,\, r(t)\sin t, 
                       \sqrt{1-r(t)^2} \cos \tilde s,\, \sqrt{1-r(t)^2} \sin \tilde s
                         \right)
      \nonumber         \\    &      =&
                             \Big(
                       a(t)\, e^{it},\, 
                      b(t)\,
                       (\cos \tilde s,\, \sin \tilde s)
                         \Big)
                         \,
\label{BK} 
\end{eqnarray}
which falls in our construction of singly spiral minimal products.
A profound advantage of singly spiral minimal products is that this situation can generate immersed minimal submanifolds of dimensions larger than half  that of the ambient sphere. Example 1 in below
will highlight this feature.

                                    Let us fix $C=0$ (so $\dot{s_2}\equiv 0$), i.e., $M_2$  not being rotated but just dilated.
                                        In order to get closed immersed spherical minimal submanifolds, 
                                        we consider when $\gamma$ can be closed.
By the rotational reflection extension of $\gamma^0$,
                                      we have   
 \begin{equation}\label{sQ0}
                                \mathlarger{\int}_{
                                \Omega_{0, \tilde C}^{0} \bigcup \cdots \bigcup \Omega_{0, \tilde C}^{l-1}}
                               \dot{s_1}
\, ds 
=  l\mathlarger{\int}_{\Omega_{0, \tilde C}^0}
                               \dot{s_1} \, ds\, .
\end{equation}
                                         As a result, to obtain a closed solution curve $\gamma$, one only needs to adjust $\tilde C$ such that 
\begin{equation}\label{sQ}
                              I({\tilde C})= \mathlarger{\int}_{\Omega^0_{0, \tilde C}}\sqrt{   \dfrac{ 1}{\tilde C\, \Big(\cos s\Big)^{2k_1+4}
         \Big( \sin s\Big)^{2k_2}-\cos^2 s}
}\, ds \, \in \pi\mathbb Q.
\end{equation}

                    \begin{cor}\label{corpm0}
                              When $C=0$, 
                              there are infinitely many solution curves (for singly spiral minimal products) which factor through closed curves.
                              In particular, Corollary \ref{S1xSn} holds with $k_1=0$. 
                              \end{cor}
                              \begin{proof}
                                   Suppose that, 
                                   for some $\tilde C$,
                                     the integral in \eqref{sQ} is $\pi\frac{j_1}{j_2}$ where $j_1$ and $j_2$ are positive  integers  prime to each other.    
                 Then, 
                 at the right hand side ending point $z^{0, \tilde C}_{R, 2j_2-1}$ of $\Omega_{0, \tilde C}^{2j_2-1}$,
                 $\gamma$ has the same local data as 
                 $z^{C, \tilde C}_{L, 0}$
                 and it then forms a closed curve over 
                  the closure 
                  $\big[z^{0, \tilde C}_{R,0},\, z^{0, \tilde C}_{R, 2j_2-1}\big]$
                  of ${\Omega_{0, \tilde C}^{0} \bigcup\Omega_{0, \tilde C}^{1} \bigcup \cdots \bigcup \Omega_{0, \tilde C}^{2j_2-1} }$
                  which runs $j_2$ rounds of the pendulum.

                              By Lemma \ref{B1} and Corollary \ref{BC1}, 
                              $  I({\tilde C})$ is not a constant function 
                              when $k_1+k_2\geq 1$
                              and
                                   can take infinitely many numbers
                                    in $\pi\mathbb Q\bigcap\big(\frac{\pi}{2(k_1+1)}, \frac{\pi}{\sqrt {2(k_1+1)}}\big)$
                                    as values.
                                    When $k_1=k_2=0$,
                                    either the range of $I(\cdot)$
                                    is $\{\frac{\pi}{2}\}$
                                    or it contains an open interval
                                    and either case can lead to infinitely many solution curves.
         \footnote{The second possibility will be eliminated in \S \ref{k120}.}
                              Hence we finish the proof and Corollary \ref{S1xSn} holds as a special case.
                              \end{proof}

{\ }

The reason why our current arguments cannot 
                        (perhaps, after analyzing the exact range of $   I(\cdot)$, still no way to) 
                        strengthen the conclusion of Corollary \ref{S1xSn}
                        with $M_2=\mathbb S^n$
                         to be embedding
is that, when $k_1=0$, we have $2\pi \in (2r,4r)$ for any rational number $r$ in the interval $\big(\frac{\pi}{2}, \frac{\pi}{\sqrt 2}\big)$ derived by \eqref{B101} and \eqref{B102}.
%
%
                       \begin{figure}[h]
	\centering
	\begin{subfigure}[t]{0.48\textwidth}
		\centering
		\includegraphics[scale=0.5]{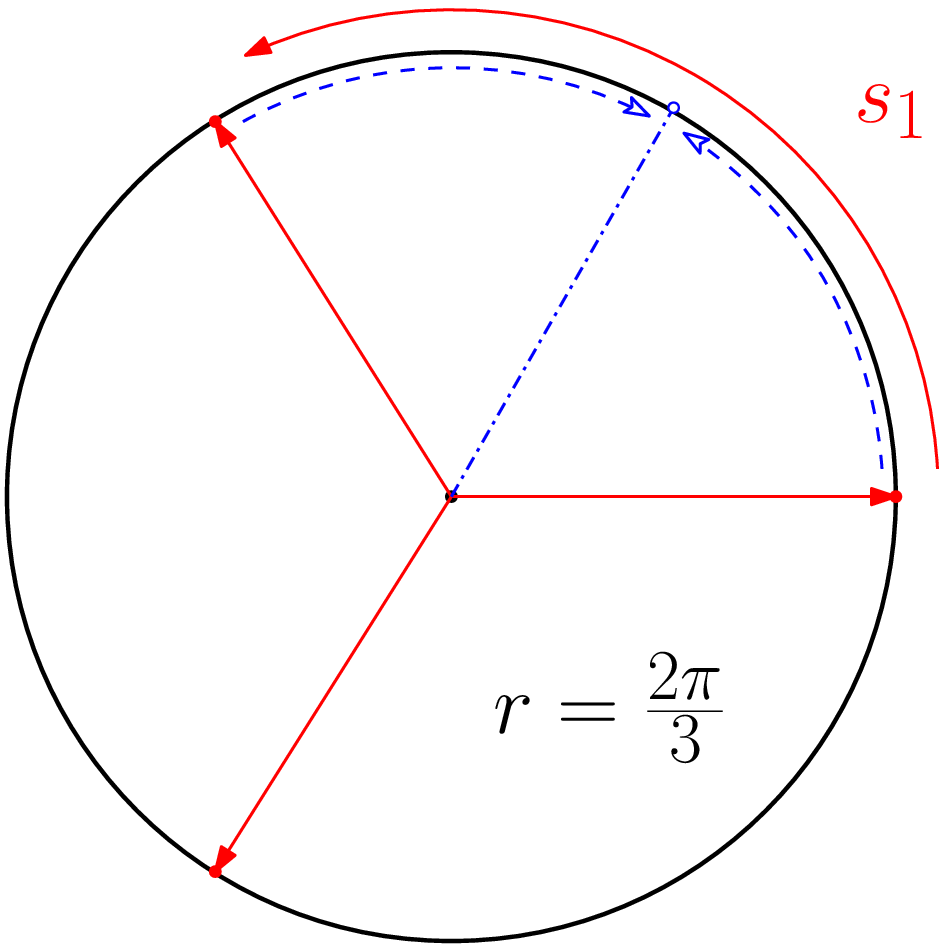}
                              \captionsetup{font={scriptsize}} 
                               \caption{ Variation of $s_1$ when $  I({\tilde C})=\frac{2\pi}{3} $\ \ \ \  \ \ \ \ }
                               \label{fig:3a}
	\end{subfigure}
	\begin{subfigure}[t]{0.48\textwidth}
		\centering
	 \includegraphics[scale=0.5]{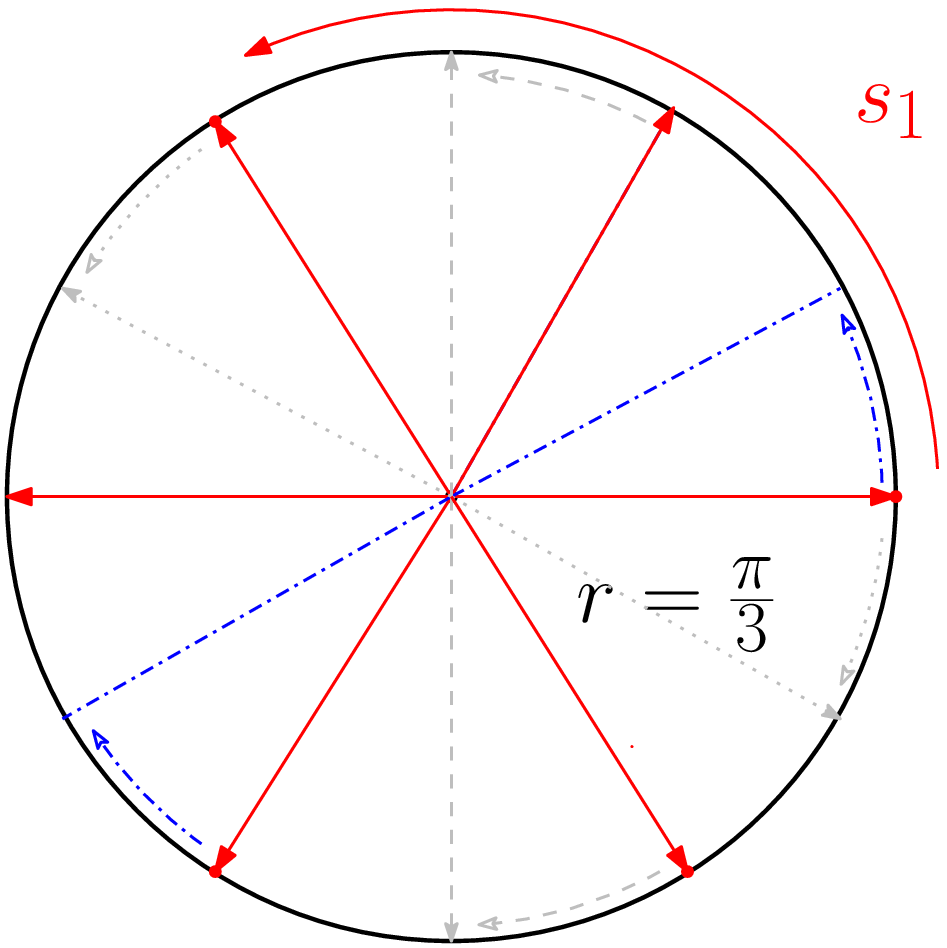}
                           \captionsetup{font={scriptsize}} 
                               \caption{Variation of $s_1$ when {$  I({\tilde C})=\frac{\pi}{3}$ }} \label{fig:3c}
	\end{subfigure}
	 \caption{Pictures for changes of $s_1$ by $2\pi$ and $\pi$}\label{ell}
\end{figure}
           %
           %
           %
           With $I(\tilde C)=r$, by our rotational reflection principle (cf. Figure \ref{fig:2a}) 
           we know that the pair $(a, b)$ shares same values at $s$ and $4 z^{0, \tilde C}_R-2z^{0, \tilde C}_L-s$
           but the argument $s_1$ differs by the integral $\Upsilon_4(s)$ of $\dot{s_1}$ over the interval 
           $$
           \left(s,\, 4 z^{0, \tilde C}_R-2z^{0, \tilde C}_L-s\right).
           $$
           Since $\Upsilon_4 \big(z^{0, \tilde C}_L\big)=4r>2\pi$ and $\Upsilon_4 \big(z^{0, \tilde C}_R\big)=2r<2\pi$,
           there exists $z\in \big(z^{0, \tilde C}_L, z^{0, \tilde C}_R\big)$ 
                   such that 
                         $\Upsilon_4(z)=2\pi$.
\footnote{
                         As shown in Figure \ref{fig:3a} for $r=\frac{2\pi}{3}$, a pair of  angles of $s_1$ for intersections are $\frac{\pi}{3}$ and $\frac{7\pi}{3}$ respectively.
          So $\gamma$ cannot avoid  self-intersection 
          in
          the interior of the closure of
            $\Omega_{0, \tilde C}^{0} \bigcup \cdots\bigcup \Omega_{0, \tilde C}^{5}$.
            As for exact locations $s$ and $s'$ for $s_1(s)=\frac{\pi}{3}$ and $s_1(s')=\frac{7\pi}{3}$,
            one needs to solve
            
 \text{ \ \ \ \ \ \ \ \ \ \ \ \ \ \ \ \ \ \ \ \ \ \ \  \ \ \ \ \ \ \ \ \ \ \ \ \ \ \ \ \ \ \ \ \ \ }$\int_{z_{L}^{0,\tilde C}}^{s}\dot{s}_1\, ds=\frac{r}{2}
 \text{\ \ \ and meanwhile \ }      
 s'=\frac{8\pi}{3}-s.
  $
  }
            Hence, $G_\gamma$ non-tangentially
\footnote{If two pieces of immersions contact each other at a point with the same tangent space there, 
then such a point is called a tangential intersection point and two immersions tangentially intersect  in the point.
If otherwise, the intersection point is called a non-tangential intersection point and two immersions non-tangentially intersect there.\label{footN}}
              self-intersects along the set (cf. Remark \ref{Rork})
            $$\Big(\cos z \, e^{is_1(z)}, \sin z \cdot \mathbb S^n\Big).$$
%
                          %
                       \begin{figure}[h]
	\centering
	\begin{subfigure}[t]{0.48\textwidth}
		\centering
		 \includegraphics[scale=0.8]{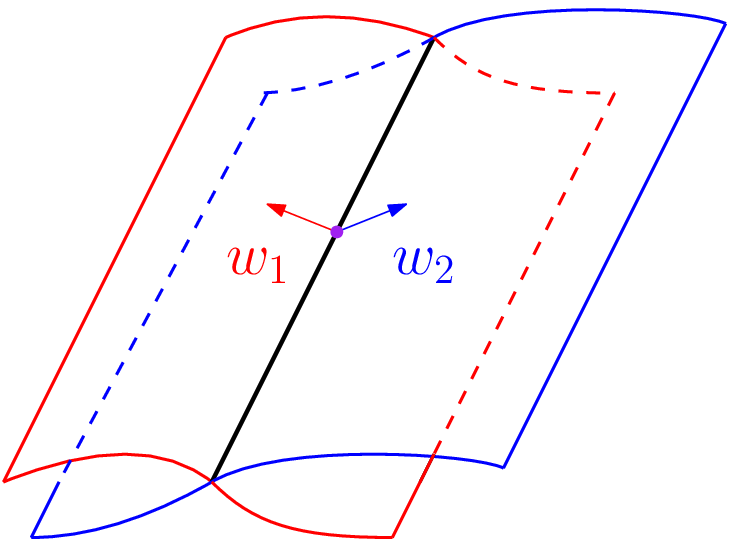}
                              \captionsetup{font={scriptsize}} 
                                 \caption{Non-tangential intersection along a set}\label{nTi}
                              	\end{subfigure}
	\begin{subfigure}[t]{0.48\textwidth}
		\centering
	 \includegraphics[scale=0.76]{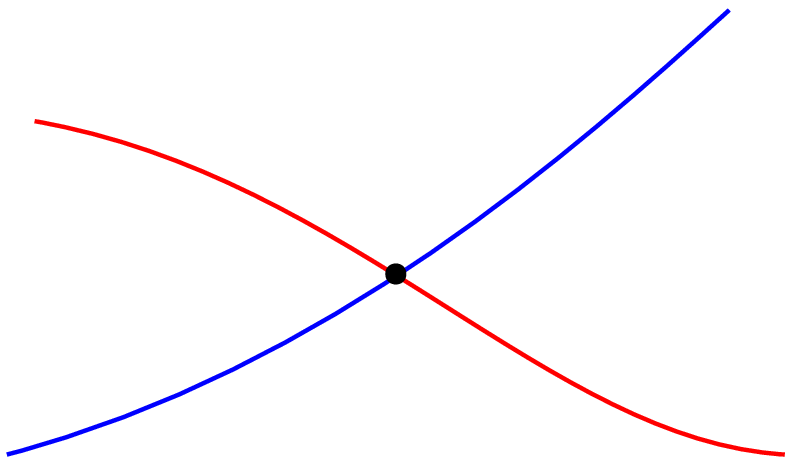}
                           \captionsetup{font={scriptsize}} 
                               \caption{Non-tangential intersection of curves} \label{nti}
	\end{subfigure}
	 \caption{Illustration of a non-tangential intersection}\label{n-i}
\end{figure}
%
%

           In fact, regarding embeddedness related to singly spiral minimal products through complexification \eqref{Cpx1} for the first slot,
           we can  establish a  characterization
           of when
           such a singly spiral minimal product of embedded spherical minimal submanifolds 
            $M_1^{k_1}\subset \mathbb S^{n_1}$ and $M_2^{k_2}\subset \mathbb S^{n_2}$ 
            factors through an embedding.
           The following observation is useful to us.
           
            {\ }\\
            {\bf Observation $(\star)$.} 
            By virtue of \eqref{Cpx1}, it can be seen, in the first slot of a spiral minimal product of embedded minimal submanifolds $M_1^{k_1}\subset \mathbb S^{n_1}$ and $M_2^{k_2}\subset \mathbb S^{n_2}$  with $C=0$, 
                        that 
                        $$a(s)e^{is_1(s)}\cdot\big(p,0\big)=a(s')e^{is_1(s')}\cdot\big(q,0\big)\in \mathbb C^{n_1+1}\, \, \,  \text{\ where }p, q\in M_1$$
                       demands that
            $|a(s)|=|a(s')|>0$ and $p=\pm q$ with $s_1(s)-s_1(s')\in \pi \mathbb Z$.
            {\ }\\

            Let $\N$ be the set of natural numbers, i.e., positive integers.

                       \begin{thm}\label{52}
           Assume that $k_1+k_2\geq 1$
           and that $M^{k_2}_2\subset \mathbb S^{n_2}$ is an arbitrary embedded minimal submanifold.
           If $M_1^{k_1}\subset \mathbb S^{n_1}$ is an embedded minimal submanifold of antipodal symmetry, 
           then
           by virtue of \eqref{Cpx1}
            a singly spiral minimal product  
            factors through an embedding of some 
            manifold
            if and only if
           $r=  I({\tilde C})\in \frac{\pi}{2\N}$.
           If $M_1^{k_1}$ contains no antipodal pairs at all, 
           then
           such a  singly spiral minimal product  
            factors through an embedding of some 
            manifold
            if and only if
           $r=  I({\tilde C})\in \frac{\pi}{\mathbb N}$.
           \end{thm}
   
   \begin{rem}\label{exasym}
   There are abundant examples of embedded spherical minimal submanifolds with antipodal symmetry,
   such as Euclidean spheres,
   Lawson surfaces $\xi _{k,l}$  in $\mathbb S^3$  \cite{L} with $k,l$ simultaneously even or odd,
   and an isoparametric minimal hypersurface or a focal submanifold (automatically being minimal) 
   of isoparametric foliations of spheres with $g=2,4,6$ different principal curvatures.
   \end{rem}

   \begin{rem}\label{plentiful}
            Among  embedded spherical minimal submanifolds, those without any antipodal pairs
            are plentiful.
            Take $M_A$ to be any of an isomparametric minimal hypersurface or a focal submanifold of an isoparametric foliation 
            with
            $g=3$ different principal curvatures of multiplicities $m_1=m_2=m_3=1, 2, 4 \text{ or } 8$
            in $\mathbb S^4$, $\mathbb S^7$, $\mathbb S^{13}$ or  $\mathbb S^{25}$ respectively,
            and $M_B$ be any embedded spherical minimal submanifold.
            Then the constant minimal product of $M_A$ and $M_B$ 
            provides such an example.
            \end{rem}

            \begin{proof}
            Case 1.
            Let $M_1$ be antipodal symmetric.
            Set $r=I(\tilde C)$.

            If $r\notin \frac{\pi}{\N}$,
            then
                      there exists a unique nonnegative integer $\ell$ such that 
         $\frac{\pi}{r} \in (\ell, \ell+1)$.
       If $\ell$ is odd, then  
               one can similarly define $\Upsilon_{\ell+1}(s)$
       to be the integral of $\dot {s_1}$ over
              $\big(s, (\ell+1) z^{0, \tilde C}_R-(\ell-1) z^{0, \tilde C}_L-s\big)$
              for $s\in \big(z^{0, \tilde C}_L, z^{0, \tilde C}_R\big)$.
              Then 
                   $\Upsilon_{\ell+1} \big(z^{0, \tilde C}_L\big)=(\ell+1)r>\pi$
                   and 
                   $\Upsilon_{\ell+1}\big(z^{0, \tilde C}_R\big)=(\ell-1)r<\pi$.
                   Hence, there exists  $z\in  \big(z^{0, \tilde C}_L, z^{0, \tilde C}_R\big)$
       such that $\Upsilon_{\ell+1}(z)=\pi$.
      Namely, the spiral minimal product non-tangentially self-intersects 
      in the set (cf. Remark \ref{Rork})
                 $$\Big(\cos z \, e^{is_1(z)}\cdot 
                 {\big(M_1,\, 0\big)}, \sin z \cdot M_2\Big) .$$
                  Note that $\Upsilon$ can be  well defined only for even subscript (for running integer rounds of the pendulum).
                   If $\ell$ is even, 
            then use $\Upsilon_{\ell+2}$.
            Now 
                           $\Upsilon_{\ell+2} \big(z^{0, \tilde C}_L\big)=(\ell+2)r>\pi$
                   and 
                   $\Upsilon_{\ell+2} \big(z^{0, \tilde C}_R\big)=\ell\, r<\pi$.
                   Similarly, non-tangential self-intersection occurs.
                   
                     If $r=\frac{\pi}{2\ell+1}\in \frac{\pi}{2\N-1}$, 
                          then integrating $\dot{s}_1$ over the closure of $\Omega_{0, \tilde C}^{0}\bigcup \cdots \bigcup\Omega_{0, \tilde C}^{2\ell}$ 
                          gives exactly $\pi$.  
                          However, 
                          value pairs $(a, b)$ have farthest distance at two ending points  of the closure.
                          Now, for $s\in \Omega_{0, \tilde C}^{0}$, let 
                        $\hat{\pi}_1(s)$ be the smallest positive value for 
                        $\int_{s}^{\hat{\pi}_1(s)}\dot{s_1}(\zeta)d\zeta=\pi$.
                          Define 
                          $$\Psi(s)=a(s)-a (\hat{\pi}_1(s)).$$
                          It is clear that 
                               $$\Psi\big(z^{0, \tilde C}_L\big)=\cos z^{0, \tilde C}_L -\cos z^{0, \tilde C}_R>0,$$
                          while 
                             $$\Psi\big(z^{0, \tilde C}_R\big)=\cos z^{0, \tilde C}_R -\cos z^{0, \tilde C}_L<0.$$
                             Thus, there exists some $z\in \Omega_{0, \tilde C}^{0}$ 
                             (as shown in Figure \ref{fig:3c} of corresponding $s_1(z)$)
                                   with
                                          $$\int_{z}^{\hat{\pi}_1(z)}\dot{s_1}(\zeta)d\zeta=\pi, \ 
                                              \text { and moreover, }\,  
                                              \big(a(z), b(z)\big)=\big(a(\hat\pi_1(z)),\, b(\hat\pi_1(z)) \big). $$
                           Namely, the spiral minimal product non-tangentially self-intersects 
      in (cf. Remark \ref{Rork}) 
                 $$\Big(\cos z \, e^{is_1(z)}\cdot 
                 {\big(M_1,\, 0\big)},\,\, \sin z \cdot M_2\Big) .$$    
                 
                  If $r=\frac{\pi}{2\ell}\in \frac{\pi}{2\N}$,
                  then  
                  %
                  after $\ell$ rounds of the pendulum
                  the generated 
                  spiral minimal patch smoothly closes up
                      by the antipodal symmetry of $M_1$.
On the other hand, by virtue of \eqref{Cpx1} and the {Observation $(\star)$},
        the first slot of $G_\gamma$ will not run the same vector unless 
        $s_1$ differs by an integer multiple of $\pi$ and the positive pair $(a, b)$  remains the same at $s$ and $s'$.
                      Since $r=\frac{\pi}{2\ell}$, 
                      the former requirement already forces 
                      $s'-s$
                      to be 
                      an integer multiple of 
                      $z^{C, \tilde C}_{R, 2\ell-1}-z^{C, \tilde C}_{L, 0}$.
                      Consequently, 
whenever $G_\gamma$ has a self-intersection point, 
        that means, locally around the point, two pieces of $G_\gamma$ must completely coincide due to the antipodal symmetry of $M_1$.
        Now $G_\gamma$ factors through an embedding of 
                  $Q_\gamma=\left(\big[-1,1\big]\times  M_1\big/\big\{\big(-1, -x\big)\sim \big(1,x\big)\big\}\right)\times M_2$
                  into $\mathbb S^{2r_1+r_2+2}$.

                  Case 2. $M_1$ has no antipodal pairs at all.
                   A similar argument
                  works with $\pi$ replaced by $2\pi$,
                  and 
                  $G_\gamma$ factors through an embedding of 
                  $Q_\gamma=\mathbb S_1\times M_1\times M_2$
                  into $\mathbb S^{2r_1+r_2+2}$.
 \end{proof}

             \begin{cor}
             Assume that $k_1+k_2\geq 1$ and that  $M^{k_2}_2\subset \mathbb S^{r_2}$ is an arbitrary embedded minimal submanifold.
             If $M^{k_1}_1\subset \mathbb S^{r_1}$ is a connected embedded minimal submanifold which does not have entire antipodal symmetry but does contain some antipodal pairs,
             then no singly spiral minimal product of them with $C=0$ by virtue of  \eqref{Cpx1}   can factor through an embedding of some manifold into $\mathbb S^{2r_1+r_2+2}$.
             \end{cor}
             \begin{rem}
             $M_1$ can take a Lawson minimal surface  $\xi_{k,l}$ with $k,l>0$ and $k\not\equiv l$ mod 2.
             \end{rem}
             \begin{proof}
             Stare at an antipodal pair $p$ and  $-p$ on $M_1$.
             Then, following the argument in the proof of Theorem  \ref{52}, 
             in order to get a spiral minimal product by virtue of \eqref{Cpx1} factoring through an embedding of some manifold,
             a necessary condition is $ I({\tilde C})=\frac{\pi}{2\ell}\in \frac{\pi}{2\N}$.
             This implies that, in view of \eqref{Cpx1}, for any $s\in \R$ 
\begin{equation}\label{00}
\gamma_1(s)\cdot (p,0)=\gamma_1\Big(s+2\ell \big(z^{0, \tilde C}_R- z^{0, \tilde C}_L\big)\Big)\cdot (-p,0).
\end{equation}
                       However, there are points of $M_1$ whose antipodal companions are not in $M_1$.
                       So, two pieces of 
                       analytic $G_\gamma$ arising from expressions in two sides of \eqref{00} cannot coincide in local.
                       Otherwise, by the connectedness assumption, 
                       entire $M_1$ should enjoy antipodal symmetry
                       which contradicts with the assumption.
                       As a result, such an intersection in \eqref{00}
                       asserts that $G_\gamma$ cannot factor through an embedding of some manifold.
             \end{proof}
             
            It shows in Lemma \ref{B1eq}
            that the range of $I(\cdot)$ contains
            $\big(\frac{\pi}{2(k_1+1)}, \frac{\pi}{\sqrt{2(k_1+1)}}\big)$
           if $k_2>0$,
           and $\big(\frac{\pi}{2(k_1+1)}, \frac{\pi}{2\sqrt{(k_1+1)}}\big)$
           if $k_1>k_2=0$.
           The interval expressions rely on the value of $k_1$ only.
           Hence, whenever $k_1\geq 1$, through simple calculations,
          $\frac{\pi}{2k_1+1}$ belongs to both intervals.
           If $k_1\geq 2$ or $k_2\geq k_1= 1$,
           so does $\frac{\pi}{2k_1}$.
           Hence, we get the following corollary of Theorem \ref{52}.

           \begin{cor}\label{singlyemb}
           Assume that $M_1^{k_1}\subset \mathbb S^{r_1}$ and $M_2^{k_2}\subset \mathbb S^{r_2}$ are embedded minimal submanifolds.
           If $M_1$ has no antipodal pairs at all 
           with $k_1\geq 1$ or $M_1$ has  antipodal symmetry with $k_1\geq 2$ or  $k_2\geq k_1= 1$,
           then  by virtue of \eqref{Cpx1} there is at least one 
           singly spiral minimal product of them with $C=0$
           which factors through an embedding 
           into $\mathbb S^{2r_1+r_2+2}$.
           \end{cor}
           
           {\ }\\
            {\bf Example 1.} With $r_1=k_1=49$ and $r_2=k_2\geq 1$,
            we take $M_1=\mathbb S^{49}$ and $M_2=\mathbb S^{r_2}$.
            Then elements of $\mathcal R=\big\{\frac{\pi}{98}, \frac{\pi}{96}, \cdots, \frac{\pi}{12}\big\}$ belong to the interval $\left(\frac{\pi}{100}, \frac{\pi}{10}\right)$
            which is contained in the range of $I(\cdot)$ according to \eqref{B101} and \eqref{B102}.
            Then, with $C=0$ by virtue of \eqref{Cpx1}, 
            every element in $\mathcal R$ can induce an embedded minimal submanifold 
            of topological type $Q_\gamma=\left(\big[-1,1\big]\times \mathbb S^{49} \big/\big\{\big(-1, -x\big)\sim \big(1,x\big)\big\}\right)\times S^{r_2}$ 
            in $\mathbb S^{100+r_2}$.
            This $Q_\gamma$ is not orientable.
            
          {\ }  
          
          More generally,
             assume that $M_2$ is an arbitrary oriented spherical embedded submanifold
             and that $M_1^{k_1}\subset S^{r_1}$ is an oriented embedded submanifold of antipodal symmetry.
             Then we have that,
             when $k_1$ is odd, 
              the embedded minimal $Q_\gamma$ induced by the singly spiral minimal is always non-orientable (as in Example 1),
              whereas, when $k_2$ is even, 
              the corresponding $Q_\gamma$ is orientable.

            {\ }
            
              \section{Situation for $C^2=1$} \label{C2=1}

                 Let us emphasize on the case of $C=\pm 1$.
                 As pointed out in Remark \ref{rem33},
                 cases of $(C, \tilde C)$ and $(-C, \tilde C)$ are
                 essentially the same
                through actions of complex conjugates.
                So we shall only talk about  $C=-1$ in this section.
                          
\subsection{Basis relation and result}
                          The first complexity beyond the situation of $C=0$ 
                          is that both complex slots spin when $C\neq 0$.                   
                              Let us derive a basic relation between variations of $s_1$ and $s_2$ over 
                              $\Omega_{-1, \tilde C}$. 
                              Set $\Delta=(\cos s)^{k_1+1} (\sin s)^{k_2+1}$ and $\Lambda=\tilde C\Delta^2-1$ (presumed to be positive in below).
Consider 
\begin{eqnarray}
                       &  &      \mathlarger{\int} \dfrac{1}{\Delta\sqrt{\tilde C\Delta^2-1}} \, d\Delta
       \nonumber                  \\
                        &= & \dfrac{1}{2} \mathlarger{\int}\dfrac{1}{\Delta^2\sqrt{\tilde C\Delta^2-1}} \, d\Delta^2
             =\dfrac{1}{2} \mathlarger{\int} \dfrac{1}{(\Lambda+1)\sqrt{\Lambda}} \, d \Lambda
  \label{Delta}                      \\
                       & = & \mathlarger{\int} \dfrac{1}{(\Lambda^{\frac{1}{2}})^2+1} \, d \Lambda^{\frac{1}{2}}
                 = \arctan \sqrt{\tilde C\Delta^2-1} + \text{const}.
                              \nonumber        
\end{eqnarray}
                                             On the other hand,
             \begin{equation}\label{22Delta}
                                  \mathlarger{\int} \dfrac{1}{\Delta\sqrt{\tilde C\Delta^2-1}} \, d\Delta
                                  =\mathlarger{\int}\,\,
                                  \dfrac
                                  {-(k_1+1)\tan s+(k_2+1)\cot s}
                                      {\sqrt{\tilde C\Delta^2-1}} 
                                         \, ds .
                \end{equation}
                              Denote  
           \begin{equation}\label{J}
                              J_{1}(\tilde C)=                                  \mathlarger{\int}_{\Omega_{-1,\tilde C}}  \dfrac{\tan s }{\sqrt{\tilde C\Delta^2-1}} 
                                         \, ds 
              \text{     \ \        \ \ \                   and\ \ \ \  \ \ }
                               J_{2}(\tilde C)=                                   \mathlarger{\int}_{\Omega_{-1,\tilde C}}  \dfrac{\cot s}{\sqrt{\tilde C\Delta^2-1}} \, ds .
                                             \end{equation}
                                             
         Then, we have the following.

     \begin{prop}\label{Cpm0}
     For any allowed $\tilde C$ with $\Omega_{-1, \tilde C}\neq \emptyset$, it follows that
      \begin{equation}\label{2Delta}
                        -(k_1+1) J_{1}(\tilde C) +(k_2+1) J_{2} (\tilde C)=0.
                                            \end{equation}
                                            In particular, $J_1(\tilde C)\in \pi\mathbb Q$ if and only if $J_2(\tilde C)\in \pi \mathbb Q$.
     \end{prop}    
     \begin{proof}
     Combining \eqref{Delta} and \eqref{22Delta} 
         and noting that $\sqrt{\tilde C\Delta^2-1} $ vanishes in
         $\p\Omega_{-1,\tilde C}$
         can result in
         \eqref{2Delta}
         for any allowed $\tilde C$.
     \end{proof}

                            An immediate  consequence is the following.
                              \begin{prop}\label{Cpm1}
                              For $C=- 1$,
                              if $J_{1}(\tilde C)\in \pi \mathbb Q$ for some allowed $\tilde C$, 
                              then corresponding solution curve $\gamma$ factors through a closed curve.
                              \end{prop}
   
                              By adjusting $\tilde C$, we arrive at the following.
                              \begin{cor}\label{corpm1}
                              When $C=-1$,
                              there are infinitely many solution curves in $\mathbb S^3$ which factor through closed curves.
                              \end{cor}
                              \begin{proof}
                              By Lemma \ref{C1} and Remark \ref{RkC},  
                              the dimension assumption  $k_1+ k_2\geq 1$ implies that $J_{1}$ is not constant,
                              and,
                              when $k_1=k_2=0$,
                             either  the range $J_1=J_2$ is  $\{\frac{\pi}{2}\}$
                              or it contains an open interval.
                              \footnote{The second possibility will be eliminated in \S \ref{k120}.}
                              Consequently, the statement of Corollary \ref{corpm1} always holds.
                              \end{proof}

\subsection{Proof of Lemma \ref{Lgn}} 
               Recall that, when                       $n_1+1=2m_1+2$ and $n_2+1=2m_2+2$,
               we  identify $\mathbb R^{n_1+1}\cong\mathbb C^{m_1+1}$ and $\mathbb R^{n_2+1}\cong\mathbb C^{m_2+1}$.
               Let  $f_1:M_1\longrightarrow \mathbb S^{n_1}$ and  $f_2:M_2\longrightarrow \mathbb S^{n_2}$
be $\mathscr C$-totally real immersions  for the $i$-multiplications $i_{m_1+1}$ and $i_{m_2+1}$ in $\mathbb C^{m_1+1}$ and $\mathbb C^{m_2+1}$ respectively.
 Now let us figure out when a spiral product $G_\gamma$ of  them
               is totally real with respect to the $i$-multiplication $i_{m_1+m_2+2}$ of $\mathbb C^{m_1+m_2+2}$.

               Note that,
               with $p=G_\gamma(t, x, y)$, 
               $$
                          i\cdot p=  i_{m_1+m_2+2} \cdot \Big(\gamma_1(t)f_1(x),\, \gamma_2(t)f_2(y)\Big)
                              =
                                        \Big(i_{m_1+1}\cdot a e^{is_1} f_1(x),\, i_{m_2+1}\cdot be^{is_2} f_2(y)\Big)
                                           $$
               and 
               $
               G'_\gamma={
                                               \big(\,
                   \big  [
                         i a s'_1 
                         +
                         a'
                    \big  ]
                         e^{is_1}
                         f_1(x), 
         \,
            \big [
                         i b s'_2 
                         +
                         b'
             \big ]
                      e^{is_2}
                         f_2(y)
                                            \big)\, 
                      }
           $
           in general parameter $t$.
By the assumption that both $f_1$ and $f_2$ are $\mathscr C$-totally real,
it follows that $i\cdot p$ is perpendicular to all tangent vectors arising from variations of $x$ and $y$.
The only condition for $G_\gamma$ to be $\mathscr C$-totally real
is
               $\big<i\cdot p,\, G'_\gamma\big>=0$
               which leads to
               $$
               a^2 s'_1  + b^2 s'_2=0.
               $$
               This precisely corresponds to $C=-1$ and hence we prove Lemma \ref{Lgn}.

            
\subsection{Coincidence with \cite{CLU} in the situation of special Legendrians}\label{cCLU}



       Note that a spherical immersion is 
                 special Legendrian 
                    if and only if it is minimal Legendiran
          (cf. \cite{HL2, Haskins}).
          Suppose that $f_1$ and $f_2$ are minimal Legendrian submanifolds.
          To get a minimal Legendrian $G_\gamma$, 
          by Lemma \ref{Lgn} we must use $C=-1$.

          With the arc parameter $s$ that we prefer,
                              $$
                              \frac{d}{ds}(\gamma_1,\gamma_2)=
                                     \left(
                                           (-b+ia\dot{s_1})e^{is_1}
                                           ,
                                           (a+ib\dot{s_2})e^{i s_2}
                                           \right).
                              $$
                             Since 
\begin{equation*}
                              \dfrac{\gamma_1 \dot{\gamma_2}-\dot{\gamma_1} {\gamma_2} }
                              {\|\dot{\gamma}\|}
                              =\dfrac{e^{i(s_1+s_2)}}
                              {\|\dot{\gamma}\|}
                              \big
                              ( 1+ i ab(\dot{s_2}-\dot{s_1})
                              \big),
\end{equation*}
by Page 49 of \cite{CLU} the Legendrian angle $\beta_\gamma$ of $\gamma$ mod $2\pi$
                              is 
\begin{equation}\label{gammaangle}
                              s_1+s_2+\arctan\big[ab(\dot{s_2}-\dot{s_1})\big].
\end{equation}
As shown in       \cite{CLU},                        the Legendrian angle of spiral minimal product $\beta_{G_\gamma}$ of two special Legendrians
                               mod $2\pi$
                              is 
\begin{equation}\label{betaangle}
                             k_1 \pi
                              +
                              k_1 s_1
                              +
                              k_2 s_2
                              +\beta_\gamma
                              + \text{const.}
\end{equation}
                            As  a special Legendrian immersion is exactly a Legendrian immersion of  constant  Legendrian angle,
                            we have the following.
                              Here we prefer to
                          check the equivalence 
                              of  $\beta_{G_\gamma}$
                               being constant
                              and
                              our \eqref{dotG} with $C=-1$ by  direct computations.
                              
                     \begin{prop}[Coincidence with \cite{CLU}]\label{coincidence}
                     For special Legendrian $f_1$ and $f_2$ into spheres,
                     a spiral product $G_\gamma$ of them is again special Legendrian
                     if and only if $C=-1$ and \eqref{dotG} holds.
                     Namely, \eqref{dotG}  with $C=-1$
 is
                     equivalent to $\beta_{G_\gamma}$ being  constant.
                     \end{prop}
                     \begin{proof}
                     The statement is clear by Lemma \ref{Lgn}.
                     For the equivalence through direct computations, 
                     it can be seen from \eqref{gammaangle} and \eqref{betaangle} that
                      $\dot{\beta}_{G_\gamma}=0$ means
\begin{equation}\label{dotbeta}
                              (k_1+1)\dot{s_1}+
                              (k_2+1)\dot{s_2}
                              +
                              \frac{d}{ds} 
                              \arctan\big[ab(\dot{s_2}-\dot{s_1})\big]
                              =0.
\end{equation}

                              Note that
                              $
                              \dot{s_1}=-\frac{b^2}{a^2} \dot{s_2} 
                              \text { when  }      C=-1.
                              $
                              The left hand side of \eqref{dotG}
                              is exactly 
\begin{equation}\label{316}
                              \dfrac{
                              (k_1+1)\dot{s_1}+
                              (k_2+1)\dot{s_2}
                              }
                              {\frac{b}{a}\dot{s_2}}
                              .
\end{equation}
     Comparing \eqref{dotbeta} and \eqref{dotG} (cf. \eqref{316}),
                              we only need to verify that
\begin{equation}\label{veri}
                              \frac{a}{b\dot{s_2}}
                              \cdot
                              \dfrac
                              {\big[ab(\dot{s_2}-\dot{s_1})\big]^{\textbf{.}}}
                              {1+\big[ab(\dot{s_2}-\dot{s_1})\big]^2}
                              =
                              \dfrac
                              {\dot{\Theta}}
                              {2\Theta(1+\Theta)}\, .
\end{equation}
                  By
                              $
                              ab(\dot{s_2}-\dot{s_1})
                              =\frac{b}{a}
                              \dot{s_2}
                              $
                              and 
                              $
                              \sqrt{\Theta}=
                              \text{sign}
                              \left(
                              \frac{b}{a}
                              \dot{s_2}
                              \right)
                              \frac{b}{a}
                              \dot{s_2}
                              $,
                              the left hand side of \eqref{veri}
                              becomes
$$                              \frac{\frac{d}{ds}\sqrt{\Theta}}
                              {\sqrt{\Theta}(1+\Theta)}
                                  =
                              \dfrac
                              {\dot{\Theta}}
                              {2\Theta(1+\Theta)}\, .
$$
                          Thus,
                              for the Legendrian (with $C=-1$)  spiral product $G_\gamma$ of special Legendrians,
                               we get the equivalence of equation \eqref{dotG}
                              and $\beta_{G_\gamma}$'s being constant.
                     \end{proof}

               
                \subsection{$\mathscr C$-totally reals in $\mathbb S^{2m+1}$ and totally reals in  $\mathbb CP^m$}\label{Rctt}
We include some well-known relation between these objects, e.g. see \cite{Rec, CDVV}.
Denote the Hopf projection from $\mathbb S^{2m+1}$ to $\mathbb CP^m$ by $\pi$.
               Then a $\mathscr C$-totally real $f: M\longrightarrow \mathbb S^{2m+1}$ induces a totally real
                immersion $u=\pi\circ f$ into $\mathbb CP^m$.
               Conversely, if $u': M'\longrightarrow \mathbb CP^m$ is a  totally real isometric immersion,
               then there exists some Riemanniann covering 
               $\varpi: M\longrightarrow M'$
               and a $\mathscr C$-totally real isometric immersion $f: M\longrightarrow \mathbb S^{2m+1}$
               such that $u'\circ \varpi=\pi\circ  f$.
                       Moreover, 
               by the basic fact 
               \footnote
               {
               The $i$-multiplication maps all tangent vectors into its normal space for  every $\mathscr C$-totally real immersion $f$,
               e.g. see \S 5.1 of \cite{Blair}.
             So
               $
               \big<\nabla_X\,  \left(i\cdot f\right), \, Y\big>
               =
                  \big<i\cdot X,\, Y\big>=0
               $
         for tangent vectors $X, Y$ at a point along $f$.
               }
                that the second fundamental form vanishes in direction $i\cdot f$ 
                for every 
                $\mathscr C$-totally real immersion $f$, 
              it follows that $f$ is minimal if and only if so is $\pi\circ f$.
               
               The horizontal lifting mentioned here is assembled by local liftings. 
               Assume $\tilde M$ to be the universal cover of $M'$,  $p'\in M'$ and $p\in \mathbb S^{2m+1}$ with $\pi(p)=p'$ to be base points.
               Identify $\pi^{-1}(p)$ as an oriented circle $\mathbb S^1$.
               Let $\tt L$ be the homomorphism from 
               the fundamental group of $M'$ to the isometry group of $\mathbb S^1$ following horizontal lifting along loops.
               Easy to see that by homotopy $\tt L$ is well-defined.
               Define $\tt K$ to be the kernel of $\tt L$, i.e., ${\tt K}=\left\{\alpha_{p'}\in \pi_{1}(M', p')\, |\, {\tt L}(\alpha_{p'})(p)=p\right\}$.
               Then, $M$ can be the (geometric) quotient $\tilde M/\tt K$ and $f$ the corresponding $\mathscr C$-totally real isometric immersion.

               When $u': M'\longrightarrow\mathbb CP^{m}$ in the above is an embedding,
so is the horizontal lifting $f:M\longrightarrow \mathbb S^{2m_1+1}$ in the preceding paragraph if 
            the image of $\tt L$ has only finitely many elements.
Moreover, when $M'$ is connected, $M$ can be chosen to be connected and 
$f$ has the property that 
                   $$
                  (\blacklozenge)\,
                   \text{ either } f(M) \text{ has antipodal symmetry or } f(M) \text{ contains no antipodal pairs at all.}
                   $$
               
               Recall that, if $u': M'\longrightarrow\mathbb CP^{m}$ is a totally real  minimal immersion
               for a real $m$-dimensional manifold $M'$,
               then such a map is called minimal Lagrangian.
               Suppose that $u': M'\longrightarrow\mathbb CP^{m}$ is a minimal Lagrangian embedding.
               Then the image of $\tt L$ can have only finitely many elements.
               That means there exists an embedded horizontal lifting $f$ of $u'$ into $\mathbb S^{2m+1}$.
               The reason of the finiteness is the following.
               If not the case,
                then there exist $\{p_j\}_{j=1}^\infty\subset {\tt L}(\pi_1(M'))(p)$ accumulating to a point $p_a$ in $\mathbb S^1$
                with $p_a=e^{i\theta_j}p_j$ and $\theta_j\rightarrow 0$.
                It is clear that local lifting through $p_j$ is exactly the image of that through $p_a$ under left-multiplication action by $e^{i\theta_j}$.
               However, it is well known that every connected part of a lifting of a minimal Lagrangian is special Legendrian in $\mathbb S^{2m+1}$
               and the Legendrian angle is a constant in every connected component.
               As $M$ was chosen to be connected, $\{p_j\}$ are in the same connected component,
               but the Legendrian angles differ from that at $p_a$ by $-(m+1)\theta_j$ (see \cite{HL2} or page 46 of \cite{CLU}).
               This generates a contradiction and proves the finiteness of the image of $\tt L$.
               
               In fact the discussion here implies the following.
               \begin{prop}\label{emdL}
               Given an $m$-dimensional connected embedded minimal Lagrangian submanifold $M'\subset \mathbb CP^m$.
               Then it has a connected embedded horizontal lifting $M^m\subset \mathbb S^{2m+1}\subset \mathbb C^{m+1}$
               such
               that the Hopf projection $\pi: M\longrightarrow M'$ gives a covering map of degree $l$
               where $l$ is an integer factor of $2(m+1)$.
               Moreover, as a set, $e^{\frac{2\pi i}{l}}\cdot M=M$.
               \end{prop}
               
               \begin{proof}
               Note that $l=\#\,\text{image of }\tt L$ and the image of $\tt L$ distributes equally in the fiber $\mathbb S^1$.
               Also note that 
              the chosen horizontal lifting $M$  in the sphere
               is a connected special Legendrian submanifold $M$
               and hence always orientable.
               Since the cone over $M$ is special Lagrangian, 
               $m+1$ has to be divisible by $l$ 
               for $(m+1)\frac{2\pi}{l}\in 2\pi \mathbb Z$
               if $e^{\frac{2\pi i}{l}}M$ induces the same orientation of $M$;
               whereas  $l$ divides $2(m+1)$
              instead 
              for $(m+1)\frac{2\pi}{l}\in (2\mathbb Z+1)\pi$
               in case that $e^{\frac{2\pi i}{l}}M$ reverses the orientation of $M$.
                              \end{proof}
           
               
               The following with the covering fold number $l>2$  was studied in \cite{DM} 
               (see \cite{DKM} for the range of $l$ for the case of surfaces).
               
                {\ }\\
            {\bf Example 2.}                Let $m=2$. 
                    Consider $F:\R^2\longrightarrow \mathbb S^5$ by
            $$F(x,y)=\frac{1}{\sqrt 3}\left(e^{i\left(\sqrt 3 x-y\right)}, \, e^{-i\left(\sqrt 3 x+y\right)},\, e^{2yi} \right).$$
            One can check that $F$ is $\mathscr C$-totally real and induces an embedded minimal torus in $\mathbb S^5$
            that  covers its image (also embedded torus) under the Hopf projection three times.
            
            It turns out that such example is exactly the spiral minimal product with $C=-1$ and steady magtitudes 
            in Corollary 3.3 of \cite{CLU} 
                  for $\mathscr C$-totally real $f_1: \mathbb S^1\longrightarrow \mathbb S^3$ 
                  given by $x\mapsto \frac{1}{\sqrt 2}\big(e^{i\sqrt 3 x},\, e^{-i\sqrt 3 x}\big)$
                  and $f_2:\{point_2\}\mapsto (1,0)\in \mathbb S^1$.
              
              {\ }
                  
                  More generally, one can get the following based on Corollary 3.3 of \cite{CLU}.
            
               {\ }\\
            {\bf Example 3.}
            Let $m\geq 2$, $f_1:\mathbb S^{m-1}\longrightarrow \mathbb S^{2m-1}$ 
            be the trivial embedding in \eqref{Cpx1} by $x\in \mathbb S^{m-1}\subset \R^{m}\mapsto (x,0)\in \R^{2m}$
            and $f_2$ by $\{point_2\}\mapsto (1,0)\in \mathbb S^1$.
            Then 
             $$F(x,y)=\left(c_\delta e^{-s_\delta ^{m} c_\delta^{-1} iy}f_1(x),\, s_\delta e^{s_\delta ^{m-2} c_\delta iy} \right),$$
             where $s_\delta=\sin\delta$ and $c_\delta=\cos\delta$ with  $\delta=\arctan\sqrt{1/m}$,
gives an embedded $\mathscr C$-totally real minimal submanifold in $\mathbb S^{2m+1}$
which covers its image (an embedded minimal Lagrangian submanifold) in $\mathbb CP^m$ under the Hopf  projection 
            $(m+1)$ times when $m$ is even and $2(m+1)$ times when $m$ is odd.
            The reason is that $F$'s image can be covered by running $x\in \mathbb S^{m-1}$
             with $y\in (0,\pi)$ when $m$ is even but necessarily with $y\in (0,2\pi)$ when $m$ is odd.
             Since $F$'s image has exactly the $\frac{i\pi}{m+1}$-action symmetry among all $e^{i\lambda}$-actions,
              the mentioned covering fold numbers follow accordingly.
              Moreover, the image of $\pi\circ F$ presents an embedded minimal Lagrangian submanifold
              and it is orientable if and only if $m$ is even.
            This example has occurred as  Example 4 in \cite{HMU}.
            
            {\ }
            
            When $f_2$ is not restricted to be a mapping of a point,  we have the following interesting symmetry of spiral minimal products.
           \begin{prop}\label{Symm}
           Let $f_1: M_1^{k_1}\longrightarrow \mathbb S^{2m_1+1}$
           and
           $f_2: M_2^{k_2}\longrightarrow \mathbb S^{2m_2+1}$
           be two $\mathscr C$-totally real minimal embeddings with $k_1,k_2\geq 1$.
           Suppose that $\frac{k_1+1}{k_2+1}=\frac{q_1}{q_2}$ where $q_1$ and $q_2$ are relatively prime.
           Then the image $M$ of their (possibly immersed without further assumptions on $f_1$ and $f_2$) spiral minimal product 
           with steady magnitudes and $C=-1$
          has the symmetry
           $$
           e^{i\frac{\pi q_1 }{q_1+q_2}}\cdot M=M=e^{i\frac{\pi q_2 }{q_1+q_2}}\cdot M.
           $$
           \end{prop}

            \begin{proof}
            Essentially            the spiral minimal product now
            is given by 
            $$
           F(t, x,y)=\left(c_\delta e^{- iq_1t}f_1(x),\, s_\delta e^{iq_2t}f_2(y) \right).
            $$
            where $\delta=\arctan\sqrt{\frac{k_2+1}{k_1+1}}$.
            So, it can be checked that 
            $$
            F(t+\Delta t, x,y)=e^{iq_2\Delta t }\cdot F(t, x,y)
            $$
            where $\Delta t=\frac{2\pi}{q_1+q_2}$.
            By the antipodal symmetries of $M_1$ and $M_2$,
            it is clear that as sets
            $ e^{i\frac{\pi q_1 }{q_1+q_2}}\cdot M=M=e^{i\frac{\pi q_2 }{q_1+q_2}}\cdot M.$
            \end{proof}

              {\ }\\
            {\bf Example 4.}
           Let $f_1$ and $f_2$ be
           the trivial embeddings of spheres of dimension $2$ and $4$ in the manner of $f_1$ in Example 3  respectively.
            Then we have $\frac{\pi q_2 }{q_1+q_1}=\frac{3}{8}\pi$.
            Together with the antipodal symmetries of images of $f_1$ and $f_2$,
            $M$ in the proposition has a finer symmetry that
            $e^{i\frac{\pi}{8}}\cdot M=M.$
        
       
       \begin{rem} 
        By recalling \eqref{2Delta},
        certain rotational symmetry result can also be gained similarly 
        for spiral minimal products of varying magnitudes with $C=-1$ whenever  $J_1(\tilde C),\, J_2(\tilde C)\in \pi\mathbb Q$.
        \end{rem}
                \subsection{Embedded $\mathscr C$-totally reals in spheres and embedded totally reals in  complex projective spaces}\label{ECTR}
Given a general pair of $\mathscr C$-totally real isometric minimal embeddings $f_1$ and $f_2$ of connected $M_1$ and $M_2$ into $\mathbb S^{2m_1+1}$ and $\mathbb S^{2m_2+1}$.
By applying Corollary \ref{corpm1} with $C=-1$,  there are infinitely many $\mathscr C$-totally real spiral minimal products of $f_1$ and $f_2$ into $\mathbb S^{2m_1+2m_2+3}$,
and consequently,  through the Hopf projection $\pi$,  infinitely  many totally real minimal immersions into $\mathbb CP^{m_1+m_2+1}$. 
Taking into account  $S^1$-actions on $f_1$ and $f_2$, actually uncountably many totally real minimal immersions into $\mathbb CP^{m_1+m_2+1}$ can be gained.

              An attractive question is how to construct a $\mathscr C$-totally real spiral minimal product which factors through  an embedding. 
At this moment, it seems hard to get a complete answer in general and instead we shall focus on those employing \eqref{Cpx1} and \eqref{Cpx2}. 

   \subsubsection{Results by analysis}      
       
       
         The following provides  a  positive result for some simple cases.
                                \begin{thm}\label{66}
           Assume that $M_1^{k_1}\subset \mathbb S^{r_1}$ and $M_2^{k_2} \subset \mathbb S^{r_2}$ 
           are embedded minimal  submanifolds of antipodal symmetry.
           Then
           by virtue of \eqref{Cpx1} and \eqref{Cpx2}
            the spiral minimal product of them   with $C=-1$
            factors through an embedding of some 
            manifold $Q_\gamma$
            if 
           $J_1({\tilde C})$ or $J_2({\tilde C})$ belongs to $\frac{\pi}{2\N}$.
           \end{thm}

                          \begin{proof}
                          Fix $C=-1$.
                          Suppose that $J_1({\tilde C})=\frac{\pi}{2\ell}\in \frac{\pi}{2\N}$.
                          Then, according to \eqref{00} for the first slot of an intersection point in $G_\gamma$ (decided by $\tilde C$)
                          the spin must run an integer multiple of $\ell$ rounds of the pendulum.
                          Since $J_2(\tilde C)=\frac{k_1+1}{k_2+1}\frac{\pi}{2\ell}$ in \eqref{2Delta},
                          there exists a smallest $\tau\in \N$ such that $2\ell\tau J_2(\tilde C)\in \pi \mathbb Z$.
                          This means by Observation $(\star)$ for each slot 
                          ($\mathbb C^{r_1+1}$ and $\mathbb C^{r_2+1}$) 
                          that, after running $\tau \ell$ rounds of the pendulum,
                          $G_\gamma$  repeats covering its image $Q_\gamma$
                          where $Q_\gamma$ is diffeomorphic to, 
                          when $\tau\in 2\mathbb Z+1$ and $2\ell\tau J_2(\tilde C)\in 2\pi \mathbb Z+\pi$,
                          $$\big[-1,1\big]\times M_1\times M_2 \big/\big\{\big(-1, -x, -y\big)\sim \big(1,x, y\big)\big\};$$
                          when $\tau\in 2\mathbb Z$ and $2\ell\tau J_2(\tilde C)\in 2\pi \mathbb Z+\pi$,
                            $$\big[-1,1\big]\times M_1\times M_2 \big/\big\{\big(-1, x, -y\big)\sim \big(1,x, y\big)\big\};$$
             when $\tau\in 2\mathbb Z+1$ and $2\ell\tau J_2(\tilde C)\in 2\pi \mathbb Z$,
                            $$\big[-1,1\big]\times M_1\times M_2 \big/\big\{\big(-1, -x, y\big)\sim \big(1,x, y\big)\big\};$$
                          and, $\tau\in 2\mathbb Z$ and $2\ell\tau J_2(\tilde C)\in 2\pi \mathbb Z$,
                          $$S^1\times M_2\times M_2.$$
                          Evidently, $Q_\gamma$ is an embedded submanifold sitting in the  target sphere $\mathbb S^{2r_1+2r_2+3}$.
      \end{proof}

               \begin{rem}\label{others}
               When $M_1$ and $M_2$ have no antipodal points at all, or one has no antipodal points and the other has the antipodal symmetry,
               our argument still works but with both the requirement on $J_1({\tilde C})$ or $J_2({\tilde C})$
               and the resulting type of $Q_\gamma$ modified accordingly.
               \end{rem}

                     \begin{cor}\label{68}
                                                     Assume that $k_1\geq 2, k_2\geq 1$ 
                                                     and 
                                                      $M^{k_1}_1, M^{k_2}_2$ 
                                                     as  in Theorem \ref{66}.
                                                     Then there exists 
                                                     $\tilde C$
                                                     such that 
                                                     $J_1({\tilde C})$ 
                                                     belongs to $\frac{\pi}{2\N}$.
                                                     Therefore,  by virtue of \eqref{Cpx1} and \eqref{Cpx2},
                                                     we get 
                                                     an embedded
                                                     $\mathscr C$-totally real
                                                     minimal submanifold $Q_\gamma$ in $\mathbb S^{2r_1+2r_2+3}$.
                                                     Further, via the Hopf projection $\pi$,
                                                     we get  an embedded totally real minimal submanifold $\pi(Q_\gamma)$ in $\mathbb CP^{r_1+r_2+1}$.
                     \end{cor}
                     \begin{proof}
                     Fix $C=-1$.
                     To have some $\tilde C$ for $J_1({\tilde C})=\frac{\pi}{2\ell}\in \frac{\pi}{2\N}$,
                     we wish to use $2\ell=2k_1$.
                     According to \eqref{C101} and \eqref{C102}, consider whether
                                    $\frac{\pi}{2(k_1+1)}<\frac{\pi}{2k_1}<\frac{\sqrt{k_2+1}\,\pi}{\sqrt{2(k_1+1)(k_1+k_2+2)}}$.
                                    The first inequality is trivial.
                                    The second demands for
\begin{equation}\label{ineqchi}
                                    1+\frac{\chi_1}{\chi_2}<\frac{2(\chi_1-1)^2}{\chi_1}
\end{equation}
                                    where $\chi_1=k_1+1$ and $\chi_2=k_2+1$.
                                    The right hand side of \eqref{ineqchi} is
                                    $2\chi_1-4+\frac{2}{\chi_1}> 2\chi_1-4$.
                                  When $\chi_1\geq 4$ and $\chi_2\geq 2$,
                                    we have
                                     $(\chi_1-4)(\chi_2-2)\geq 0$
                                     and consequently
                                    $\chi_1+\chi_2\leq 4\chi_1+4\chi_2-16\leq 2\chi_1\chi_2-4\chi_2$
                                    which guarantees \eqref{ineqchi}.
                                    When $\chi_1=3$,
                                    \eqref{ineqchi}
                                    can be solved by
                                    $
                                    5\chi_2>9
                                    $.
                                    Hence the first half of the statement 
                                    holds
                                    by applying Theorem \ref{66}.
                                    
                                    When both $M_1$ and $M_2$ have antipodal symmetry,
                                     $Q_\gamma$ has antipodal symmetry. 
                                    By {Observation $(\star)$}  for each slot 
                          ($\mathbb C^{r_1+1}$ and $\mathbb C^{r_2+1}$),
                                    $\pi$ is a double covering from $Q_\gamma$  to $\pi(Q_\gamma)$.
                                    So, $\pi(Q_\gamma)$ is an embedded totally real minimal submanifold in $\mathbb CP^{r_1+r_2+1}$.
                                    Now the  statement gets proved.
                     \end{proof}

          \begin{rem}
 When $k_1$ is large, one can require $J_1({\tilde C})=\frac{\pi}{2k_1-2}$, etc., in the construction.
\end{rem}

\begin{rem}
 If $k_2=0$, then we shall need $k_1\geq 4$ for \eqref{ineqchi}.
 Let $M_2=\{point_2\}\in\mathbb S^1\subset \mathbb C$ and $M'_1$ a connected embedded 
minimal Lagrangian in $\mathbb CP^{r_1}$.
  By Proposition \ref{emdL}, $M'_1$ has a connected embedded horizontal lifting in $\mathbb S^{2r_1+1}$ which is special Legendrian.
 Hence, for $r_1\geq 4$, by $(\blacklozenge)$
 an embedded minimal Lagrangian in $\mathbb CP^{r_1+1}$ can be produced.
\end{rem}

                 \begin{rem}
              Regarding Remark \ref{others},
 when $M=M_1\times M_2\subset \R^{r_1+r_2+2}$ contains no antipodal points at all,
               it is easy to see that 
               $Q_\gamma$,
              gained by the same kind of construction via  \eqref{Cpx1} and \eqref{Cpx2},
             is diffeomorphic to   $\pi(Q_\gamma)$ 
               according to Observation $(\star)$.
               \end{rem}

    {\ }    
        
        Roughly speaking,
            the above is almost all that one can do through basic analysis without bothering entire distributions of $s_1$ and $s_2$.
            In the following we make use of geometry instead to establish more.

\subsubsection{Results by geometry} \label{calG}
          To be precise, we shall first borrow some machinery  from calibrated geometry
          to prove the following.
          
          \begin{thm}\label{spl}
          Let $M_1=\mathbb S^{r_1}\subset \R^{r_1+1}$ and $M_2=\mathbb S^{r_2}\subset \R^{r_2+1}$.
          Then, by virtue of \eqref{Cpx1} and \eqref{Cpx2},
         every spiral minimal product of them with $C=-1$ and $J_1(\tilde C)\in \pi\mathbb Q$
          factors through some embedding. 
          \end{thm}
 
 \begin{proof}         
           Since $\mathbb S^{r_1}$ and $\mathbb S^{r_2}$ are special Legendrian in the sense of \eqref{Cpx1} and \eqref{Cpx2} respectively,
                          by \S \ref{cCLU} we know that, by virtue of \eqref{Cpx1} and \eqref{Cpx2},  every spiral minimal product $G_\gamma$ of them with $C=-1$  is special Legendrian. 
                          Hence, the cone $C(G_\gamma)=\left\{t\cdot p\,\vert\, t\in (0,\infty), \, p\in G_\gamma \right\}$
                          is a special Lagrangian cone which is calibrated by a constant (i.e., parallel in $\R^{2r_1+2r_2+4}$)   special Lagrangian calibration form $\phi$.
                          We refer to the seminal reference \cite{HL2} for knowledge about calibrations. 

                             According to  Observation $(\star)$ and the antipodal symmetry of $\mathbb S^{r_1}$ and $\mathbb S^{r_2}$,
                            if $G_\gamma$ has an intersection point at  $ G_\gamma(s', p_1, p_2)$ where $p_1\in \mathbb S^{r_1}$ and  $p_2\in \mathbb S^{r_2}$,
then $ G_\gamma$ self-intersects along the set
\begin{equation}\label{mathscrS}
\mathscr S= G_\gamma\big(s', \mathbb S^{r_1}, \mathbb S^{r_2}\big)=\Big(a(s')\, e^{is_1(s')}\cdot \big(\mathbb S^{r_1},\, 0\big),\,\, b(s')\, e^{is_2(s')}\cdot \big(\mathbb S^{r_2},\, 0\big)\Big).
\end{equation}

Note that tangential intersections (see Footnote \ref{footN}) of $ G_\gamma$ cause no problems.
                       Due to the structure of $G_\gamma$,
                       if there is a tangential intersection point of $G_\gamma$ for $(s', p_1, p_2)$ and $(s'', p_1, p_2)$
                       or
                      for $(s', p_1, p_2)$ and $(s'',-p_1, -p_2)$,
                       then around $s'$ and $s''$
                    two local pieces $G^1$ and $G^2$ of $G_\gamma$ tangentially intersects along $G_\gamma (s', \mathbb S^{r_1}, \mathbb S^{r_2})$ vs.  Figure \ref{nTi}.
      With small $\epsilon>0$ and a suitable choice of interval pairs 
                $$\big\{(s', s'+\epsilon),\, (s'',s''+\epsilon)\big\} \text{\ \ \ \  or\ \ \ \  } \big\{(s', s'+\epsilon),\, (s''-\epsilon, s'')\big\},$$
 one can assemble corresponding parts of $G^1$ and $G^2$ for an oriented $C^1$ submanifold which satisfies Harvey-Lawson's extension result $-$ Theorem 1.4 in \cite{HL0}.
As a result, around every tangential intersection point of $ G_\gamma$,   
the corresponding pieces $G^1$ and $G^2$ of $G_\gamma$ must geometrically coincide  in local, up to an orientation.

                         It is obvious that, when $J_1(\tilde C)\in \pi\mathbb Q$, 
                         $G_\gamma$ with $C=-1$ factors through an immersion of some quotient of $S^1\times M_1\times M_2$.
                          So, if $G_\gamma$ does not factor through an embedding,
                          then, by the discussions in the preceding paragraphs (or alternatively by Cauchy–Kovalevskaya Theorem),
                          there must be a non-tangential intersection set $\mathscr S$ (as in \eqref{mathscrS}) of dimension $r_1+r_2$  in Figure \ref{nTi}.
                          As a consequence, 
                          in local
                          $C(G_\gamma)$ inherits such type of non-tangential intersections along  the cone $C(\mathscr S)$.
                          Nevertheless, this is impossible due to the following fundamental decomposition result of forms.
                                     
                                      \begin{lem}[Lemma 2.12 in \cite{HL1}] \label{hl1}
             Let $\xi\in \Lambda^k \mathbb{R}^n$ be a simple $k$-vector \footnote{Namely, an exterior product of $k$ vectors.} with
             $V = span\ \xi$. Suppose that $\phi\in 
             {(\Lambda^k\mathbb{R}^n)^*}$ satisfies  $\phi(\xi)= 1$. 
              Then there exists a unique oriented complementary subspace $W$ to $ V$  with the following property.
                      For any basis $\{v_1, \cdots, v_n\}$ of $\mathbb{R}^n$ such that $\xi=v_1\wedge... \wedge v_k$
                      and 
                      $
                      \{v_{k+1}, \cdots, v_n\}$ is basis for $ W$, 
                      one has that
                                 \begin{equation}\label{decomp}
                                 \phi=v_1^*\wedge \cdots \wedge v_k^*+ \sum a_Jv_J^*,
                                  \end{equation}
                                   where $a_J=0$ whenever $j_{k-1}\leq k$. 
                                   Here $J=\{j_1,\cdots, j_k\}$ with $j_1<\cdots<j_k$.
              \end{lem}
                               Take a non-tangential intersection point $p\in \mathscr S$
                               and $\xi_0$ to be the unit size oriented  simple $(r_1+r_2)$-vector of $\mathscr S$.
                               Assume that unit vectors $w_1$ and $w_2$ at $p$ in Figure \ref{nTi}
                                are chosen for $\xi_0\wedge w_1$ and $\xi_0\wedge w_2$ to be the unit oriented $(r_1+r_2+1)$-vectors
                                 of oriented $T_pG^1$ and $T_pG^2$.
                                Correspondingly, 
                             unit oriented $(r_1+r_2+2)$-vectors
                              $\xi_1=p\wedge \xi_0\wedge w_1$ and $\xi_2=p\wedge \xi_0\wedge w_2$
                               span the tangent spaces of $C(G^1)$ and $C(G^2)$
                               at $p$ respectively.
                                                  
                                                  As the oriented $C(G_\gamma)$ is calibrated by a  special Lagrangian calibration form $\phi$,
                                                  it means that $\phi(\xi_1)=\phi(\xi_2)=1$
                                                  and
                                                  by linearity
                                                  we have that, for any $\theta\in \R$,
                                                  $$\phi\Big(p\wedge \xi_0\wedge \big(\theta \cdot w_1+(1-\theta) \cdot w_2 \big)\Big)=1.$$
                                                 By applying Lemma \ref{hl1} to the special Lagrangian form $\phi$ and the simple $(r_1+r_2+2)$-vector $\xi_1$,
                                                  we can get a unique oriented complementary subspace $W$ to $V=\text{span\ }\xi_1$.
                                                  For the dimension reason, 
                                                 $W\bigcap \text{ span}\{w_1, \, w_2\}\neq \emptyset$.
                                                 So, there exists some $w_0=\theta_0 \cdot w_1+(1-\theta_0) \cdot w_2\in W$ with $\theta_0\neq 1$
                                                 and $\phi(p\wedge\xi_0\wedge w_0)=1$.
                                                 However, 
                                                 as $w_0\in W$ and  $a_J=0$ when $j_{k-1}\leq k$,
                                                 the pairing 
                                                 $\phi(p\wedge\xi_0\wedge w_0)$ should vanish
                                                 by \eqref{decomp}.
                                                 This results in a contradiction and Theorem \ref{spl} gets proven  to be true.
\end{proof}

          By antipodal symmetry of $\mathbb S^{r_1}$ and $\mathbb S^{r_2}$,
          we get the following.
          
          \begin{cor}\label{corgamma}
          Let $\gamma$ be a solution curve with $C=-1$.
          The, as sets,  either $\gamma=-\gamma$ or $\gamma\bigcap -\gamma=\emptyset$,
          i.e., non-tangential intersection of $\gamma$ and $-\gamma$ in Figure \ref{nti} cannot occur.
          \end{cor}
          \begin{proof}
          If $\gamma\bigcap -\gamma\neq \emptyset$, then take a common element $p$.
          By antipodal symmetry of $\mathbb S^{r_1}$ and $\mathbb S^{r_2}$,
          there is a corresponding self-intersection of $G_\gamma$ along set $\mathscr S$ of codimension one as in \eqref{mathscrS}.
          If, as sets, $\gamma$ around $p$ and around $-p$ are not antipodal symmetry,
          then the intersection along $\mathscr S$ has to be non-tangential
              (as reasoning in the proof of Theorem \ref{spl}).
          This is a contradiction with Theorem \ref{spl}.
          So,  $\gamma$ around $p$ and around $-p$ are antipodal symmetric.
          Hence, by the analyticity of $\gamma$ in \S \ref{S4},
         as sets $\gamma=-\gamma$.
          \end{proof}
          
          Another useful corollary of Theorem \ref{spl} is the following.
          By the structure of the construction,
           one particular choice in Theorem \ref{spl}
          determines thousand others.
          
\begin{cor}\label{spl->G}
                    Let $M^{k_1}_1\subset \mathbb S^{r_1}$ and  $M^{k_2}_2\subset \mathbb S^{r_2}$ be embedded minimal submanifolds
          such that as a set $M=M_1\times M_2 \subset \R^{r_1+r_2+2}$ 
          satisfies either \\
          \text{\tt (1)} 
          $M=-M$ 
          or \\
          \text{\tt (2)}
          $M\bigcap -M=\emptyset$.
          \\
          Then, by virtue of \eqref{Cpx1} and \eqref{Cpx2},
         every spiral minimal product of them with $C=-1$ and $J_1(\tilde C)\in \pi\mathbb Q$
          factors through an embedding of some manifold into $\mathbb S^{2r_1+2r_2+3}$.
              Moreover, 
                   joint the with Hopf projection,
                  %
               infinitely many  embedded totally real minimal submanifolds
                    in $\mathbb C P^{\,r_1+r_2+1}$ can be gained.
\end{cor}
                       
 

        %
 \begin{rem}
   If both $M_1$ and $M_2$ belong to those in Remark \ref{exasym}, then $M$ satisfies  \text{\tt (1)}.
   \end{rem}
   \begin{rem}\label{plentiful2}
            If either of $M_1$ and 
                  {$M_2$} belongs to those in Remark \ref{plentiful} and the other is an arbitrary embedded spherical minimal submanifold,
            then $M$ satisfies  \text{\tt (2)}.
            \end{rem}

\begin{proof}
Fix $C=-1$.
It is the same that,
                                   according to  Observation $(\star)$ and assumption \text{\tt (1)} or \text{\tt (2)},
                            if spiral minimal product $\tilde G_\gamma$ of $M_1$ and $M_2$
                            has an intersection point at  $\tilde G_\gamma(s', p_1, p_2)$,
then $\tilde G_\gamma$ self-intersects along the set
                $$\tilde G_\gamma(s', M_1, M_2)=\Big(a(s')\, e^{is_1(s')}\cdot \big(M_1,\, 0\big),\,\, b(s')\, e^{is_2(s')}\cdot \big(M_2,\, 0\big)\Big).$$

Similarly, due to assumption \text{\tt (1)} or \text{\tt (2)} and by Harvey-Lawson's extension result,
 tangential intersections of $\tilde G_\gamma$ cause no problems for $\tilde G_\gamma$ to factor through an embedding.
              So,     
                         suppose that $\tilde G_\gamma$ has
          non-tangential intersections  as in Figure \ref{nTi}.
       Since $\mathbb S^{k_1}, \mathbb S^{k_2}$ have  antipodal symmetry (no less than those of $M_1$ and $M_2$),
       if we use the same $\gamma$ to build $G_\gamma$, 
       with $C=-1$,
        by virtue of \eqref{Cpx1} and \eqref{Cpx2}, 
       for $M'_1=\mathbb S^{k_1}\subset \R^{k_1+1}, M'_2=\mathbb S^{k_2}\subset \R^{k_2+1}$,
       then $G_\gamma$ of $M'_1$ and $M'_2$ into  $\mathbb S^{{2k_1+2k_2+3}}$ must have non-tangential intersections as well.
        This would give a contradiction with Theorem \ref{spl}.
        Therefore, we finish the proof.
 \end{proof}

 Apparently, Corollary \ref{spl->G} confirms part of Theorem \ref{mainebd}.
 In order to prove entire Theorem \ref{mainebd}, we shall focus on Corollary \ref{corgamma}.
 We wish to use $\gamma$ with the property $\gamma\bigcap -\gamma=\emptyset$.
 Because, once such $\gamma$ with $C=-1$ exists, 
 by virtue of \eqref{Cpx1} and \eqref{Cpx2} and according to Observation $(\star)$,
  spiral minimal product $G_\gamma$ of arbitrary embedded minimal submanifolds
 $M^{k_1}_1\subset \mathbb S^{r_1}$ and  $M^{k_2}_2\subset \mathbb S^{r_2}$
 will have no non-tangential self-intersection point.
     As a result, we achieve the following.
     
     \begin{thm}
     When $k_1+k_2\geq 1$, there are infinitely many closed solution curves $\gamma$ with $C=-1$ 
     such that as sets $\gamma\bigcap -\gamma=\emptyset$.
     In particular, Theorem \ref{mainebd} is true.
     \end{thm}     

              \begin{proof}
              With $C=-1$ and allowed $\tilde C$,
             we have $(k_2+1)J_2(\tilde C)=(k_1+1)J_1(\tilde C)$.
             Assume that $k_1+1=2^{\ell_1} q_1$
             and $k_2+1=2^{\ell_2} q_2$
             where $\ell_1,\, \ell_2$ are nonnegative integers
              and $q_1,\, q_2$ positive odd integers.
              
             When $k_1+k_2\geq 1$, 
             by 
             {Lemma \ref{C1}}
the range of $J_1$ contains nonempty open interval              $$\mathscr I= \left(\frac{\pi}{2(k_1+1)},
                                                    \frac{\sqrt{k_2+1}\,\pi}{\sqrt{2(k_1+1)(k_1+k_2+2)}}\right).$$

           Take $\tilde C$ with $J(\tilde C)=\frac{d_1}{d_0}\pi   \in \mathscr I\bigcap \pi\mathbb Q$, where $d_1$ and $d_0$ 
           are positive integers relatively prime to each other.
Assume that $d_0=2^{\ell_0}q_0$ where $\ell_0$ is a nonnegative integer and  $q_0$ an odd integer.
Then consider
                           $h_j=\frac{2^{\ell_0-\ell_1+\ell_2+j}}{2^{\ell_0-\ell_1+\ell_2+j}+1}$
                           with $j>\min\{\ell_1-\ell_2,0\}$.
                           Since $h_j\uparrow 1$ as $j\uparrow \infty$,
                           it is clear that $\frac{h_jd_1}{d_0}\pi\in \mathscr I $ for sufficiently large $j$.
                        Take $\tilde C_j$ for
                          $J_1(\tilde C_j)=\frac{h_jd_1}{d_0}\pi$.
                          Now,  with $C=-1$, we have 
\begin{equation}\label{J12j}
                        J_1(\tilde C_j)=\frac{2^{-\ell_1+\ell_2+j}d_1\pi}{q_0(2^{\ell_0+\ell_1-\ell_2+j}+1)}
                        \ \text{\ \ \  and\ \ \  }\
                          J_2(\tilde C_j)=\frac{2^jq_1 
                          {d_1}\pi}
                          {
                           q_2
                          {q_0}
                          (2^{\ell_0+\ell_1-\ell_2+j}+1)}.
\end{equation}
Denote the corresponding  solution curve by $\gamma_{\tilde C_j}=\gamma_{-1,\tilde C_j}$.
                          
                          By Corollary \ref{corgamma},
                          if we track points
                          realizing minima of $a(s)$
                          (corresponding to $\big\{z^{-1, \tilde C}_{L, \ell}\big\}_{\ell\in \mathbb Z}$)
                          and can show that the set $\mathscr E$ of  those points is not antipodal symmetric,
                          then neither is the solution curve $\gamma_{\tilde C_j}$.
                          Further, that means $\gamma_{\tilde C_j}\bigcap -\gamma_{\tilde C_j}=\emptyset$
                         and this property enables us to take arbitrary embedded spherical minimal submanifolds as inputs
                         for the spiral minimal product program through \eqref{Cpx1} and \eqref{Cpx1} without non-tangential intersections.
                         Hence, each $G_{\gamma_{\tilde C_j}}$ with the property factors through an embedding.
                         
                         To finish the proof,
                        we use \eqref{J12j}.
                        Note that 
                        one can only run integer rounds of the pendulum to return the minima of $a(s)$,
                        i.e., 
                        $\big|z^{-1, \tilde C}_{L, \ell}-z^{-1, \tilde C}_{L, \ell'}\big|\in 2\big(z^{-1, \tilde C}_{R}-z^{-1, \tilde C}_{L}\big)\mathbb Z$ for $\ell\neq \ell'$.
 By  \eqref{J12j}, we know that
                         $J_1(\tilde C_j)$ times an even integer never produces an odd multiple of $\pi$.
                         Neither does $J_2(\tilde C_j)$.
                         This implies $\mathscr E\bigcap -\mathscr E=\emptyset$.
                         Consequently, we show that $\gamma_{\tilde C_j}\bigcap -\gamma_{\tilde C_j}=\emptyset$ holds.
                         And we complete the proof by these $\big\{\gamma_{\tilde C_j}\big\}$ and $\big\{G_{\gamma_{\tilde C_j}}\big\}$.
              \end{proof}

              Now let us provide a proof of Theorem \ref{mL} for the construction of embedded minimal Lagrangians in complex projective spaces.
       
       \begin{pf0}
       Based on \S \ref{Rctt}, 
       it can be observed that $(\blacklozenge)$
       can be strengthened to be the following:
       whenever the connected embedded horizontal lifting $M_1$ for $u_1'$
       has $e^{i\lambda}M_1\bigcap M_1\neq \emptyset$ for some $\lambda\in \R$,
       then as sets $e^{i\lambda}M_1=M_1$.
       The same holds for the connected embedded horizontal lifting $M_2$ of $u_2'$.
       
       Now take spiral minimal products with $C=-1$ and $\tilde C$ for $J_1(\tilde C)\in \pi\mathbb Q$.
       From the symmetric property of horizontal liftings that we just emphasize on and the argument of applying Lemma \ref{hl1}, 
       resulting specal Legendrian immersions cannot have non-tangential self-intersections.
       Hence, the spiral minimal products factor through certain embeddings.
       Again, by the symmetric property,
       through the Hopf projection each such embedded submanifold covers its image finitely many times. 
       
       With fixing $C=-1$, the cardinality $\mathbb Q$ of value $\tilde C$ for $J_1(\tilde C)\in \pi\mathbb Q$ is due to 
       Lemma \ref{C1} and Remark \ref{RkC}.
       Since one can choose horizontal liftings $e^{i\lambda_1}M_1$ and $e^{i\lambda_2}M_2$ 
       which have a $S^1$-rotational phase-difference $e^{i(\lambda_1-\lambda_2)}$,
       taking this into account leads us to the cardinality in Theorem \ref{mL} is $\mathbb Q\times S^1$.
              \end{pf0}
           
           \begin{rem}
           Based on $\mathbb RP^{m}\subset \mathbb CP^m$ and a huge diversity of minimal Lagrangian tori in $\mathbb CP^2$ 
           by Carberry–McIntosh \cite{CM},
           one can repeatedly apply Theorem \ref{mL} to create examples of minimal Lagrangians 
           in corresponding complex projective spaces.
           \end{rem}

{\ }

\section{Situation with General $C$}\label{GC}

Define $J_1^C(\tilde C) 
     ={
                              \displaystyle         \mathlarger{\int}_{\Omega_{C, \tilde C}} 
                                     \dfrac{\tan s}{  \sqrt{  {\tilde C  
                 \Big(\cos s\Big)^{2k_1+2}
                  \Big( \sin s\Big)^{2k_2+2}
                  \,
                   -
                   {
                      1-(C^2-1)\cos^2 s
                   }
            }
            }
 }\, ds
}
$
$$\ J_2^C(\tilde C) 
     ={
                                          \displaystyle
                                          \mathlarger{\int}_{\Omega_{C, \tilde C}} 
                                    \dfrac{\cot s}{    \sqrt{ {\tilde C  
                 \Big(\cos s\Big)^{2k_1+2}
                  \Big( \sin s\Big)^{2k_2+2}
                  \,
                   -
                   {
                      1-(C^2-1)\cos^2 s
                   }
            }
            }
 }\, ds
} 
$$
for general $C$.
So $J_1, J_2$ in the previous section correspond to $J_1^C(\tilde C) $ and $J_2^C(\tilde C)$ with $C=-1$.
However, generally for $C\neq \pm1$ or $0$, 
            the relation between $\int_{\Omega_{C,\tilde C}} \dot s_1 ds=J_1^C(\tilde C)$ and $\int_{\Omega_{C,\tilde C}} \dot s_2 ds=CJ_2^C(\tilde C)$ is not that clear.
                      In particular, a priori one does not know 
                              whether there exists $\tilde C$ with  $C\neq \pm 1$ or $0$ 
                              such that both $J_1^C(\tilde C)$ and $CJ_2^C(\tilde C)$ lie in $\pi\mathbb Q $.
           The importance is that, for closed minimal submanifolds $f_1$ and $f_2$, 
           only when both quantities $J_1^C(\tilde C)$ and $CJ_2^C(\tilde C)$ belong to $\pi\mathbb Q $,
           the corresponding spiral minimal product can descend to an immersion from a closed manifold into the target sphere (otherwise the source manifold can never be compact).

By studying  $R(C,\tilde C)=CJ_2^C(\tilde C)\big/ J_1^C(\tilde C)$,
                                the following can be confirmed. 
                              
                              \begin{thm}\label{CC}
                              When $k_1+k_2>0$
                                    \footnote{The case $k_1=k_2=0$ will be discussed in \S \ref{k120}.},
                              the cardinality of $(C, \tilde C)$ for both $J_1^C(\tilde C)$ and $CJ_2^C(\tilde C)$ in  $\pi\mathbb Q $
                              is at least that of $\mathbb Q\times \mathbb Q$.
                              Hence, Theorem \ref{main0} holds.
                              \end{thm}
                              \begin{proof}
                              With                      
                                $$
           \Omega=
          \big\{
          (C, \tilde C)
        \,\big\vert\,
           \tilde C>P_C(s_C)
               \big\}
          $$
                              we only need to consider function $R$ on region
                              $$
           \Omega_+=
          \big\{
          (C, \tilde C)\in \Omega
        \,\big\vert\,
        C> 0
               \big\}
          $$
                  the boundary of which is
                                          $$\p \Omega_+=\Big\{(0,\tilde C)\,\big\vert\, \tilde C> P_0(s_0)\Big\}\bigsqcup \big(\p \Omega\big)_+\, .$$
                                            \begin{figure}[h]
		\includegraphics[scale=0.77]{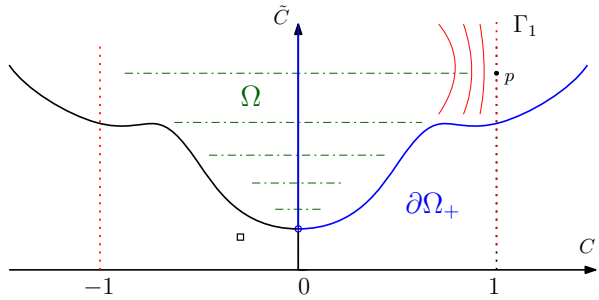}
                             \caption{Illustration Figure (not caring about precise behaviors)} 
                          \label{fig:5}
                             \end{figure}
                             {\ }\\
        {\bf Claim A. } $R$ is an analytic function in $\Omega_+$. 
                              \\
                               {\it Proof of the Claim A. }
                               Let us use the arc parameter $t_{arc}$ of $\overline{\,\gamma^0\,}$. 
                               Then, by $\frac{d t_{arc}}{ds}=\sqrt{1+\Theta}$,
                               we have 
        \begin{equation}\label{tarc}
                               \frac{d}{d t_{arc}} s_1=\frac{\dot s_1}{\sqrt{1+\Theta}}
                               =\frac{\tan s}{\sqrt{\tilde C  
                 \big(\cos s\big)^{2k_1+2}
                  \big( \sin s\big)^{2k_2+2}}}
                  =\frac{1}{\sqrt{\tilde C}\,\, a^{k_1+2}b^{k_2}}\, ,
       \end{equation}
         \begin{equation}\label{tarc1}
                               \frac{d}{d t_{arc}} s_2
                               =\frac{\dot s_2}{\sqrt{1+\Theta}}
                               =\frac{C\cot s}{{\sqrt{\tilde C  
                 \big(\cos s\big)^{2k_1+2}
                  \big( \sin s\big)^{2k_2+2}}}}
                  =\frac{C}{\sqrt{\tilde C}\, \, a^{k_1} b^{k_2+2}} ,
       \end{equation}
        \begin{equation}\label{tarc2}
                               \frac{d}{d t_{arc}} a
                               =-\frac{b}{\sqrt{1+\Theta}}
                               =
                               -\frac
                               {
                                \sqrt{\tilde C  
                 a^{2k_1+2}
               b^{2k_2+2}
                  \,
                   -
                   {
                      1-(C^2-1) a^2
                   }
            }
                               }
         {\sqrt{\tilde C}\,
                 a^{k_1+1}
                  b^{k_2}} ,
       \end{equation}
          \begin{equation}\label{tarc3}
                         \,      \frac{d}{d t_{arc}} b=\frac{a}{\sqrt{1+\Theta}}
                                                        =
                               \frac
                               {
                                \sqrt{\tilde C  
                 a^{2k_1+2}
               b^{2k_2+2}
                  \,
                   -
                   {
                      1-(C^2-1) a^2
                   }
            }
                               }
         {\sqrt{\tilde C}\,
                 a^{k_1}
                  b^{k_2+1}}
                         .
       \end{equation}
       
           By geometric meanings of $\overline{\,\gamma^0\,}$ in \S \ref{S4}, up to a rotational isometry (i.e., a  composition of suitable separate rotations in each $\C$-copy of $\C\oplus\C$),
                              one can always assume that $s_1=s_2=0$ at the starting point  of the system with vanishing \eqref{tarc2} and \eqref{tarc3}.
                               Since $\overline{\,\gamma^0\,}$ is 
                               solved by the minimal surface system for spiral minimal products,
                               it is clear that $\overline{\,\gamma^0\,}$ and its ending point (with extremal $a$-value)
                               are analytic 
                               in its initial position and unit tangent, 
                               i.e., 
                               $$\left(a_0, b_0\right)
                               \text{\ \ \ \  and\ \ \ \  } 
                               \left(
                               \frac{i}{\sqrt{\tilde C} a_0^{k_1+1}b_0^{k_2}},
                               \frac{iC}{\sqrt{\tilde C} a_0^{k_1}b_0^{k_2+1}}
                               \right).
                               $$
                               As $(C,\tilde C)$ decides $a_0$ and $b_0$ analytically
                              whenever $a_0b_0\neq 0$, i.e., $C\neq 0$,
                               the complex angles 
                               $J_1^C(\tilde C)$ and $J_2^C(\tilde C)$ 
                               of two complex components of the ending point of $\overline{\,\gamma^0\,}$
                                are analytic in $C$ and $\tilde C$.
                               Hence, $R(C,\tilde C)$ is analytic in $\Omega_+$.
                        %
                              Thus Claim A gets verified.

                               {\ }

                  To finish the proof of Theorem \ref{CC},
                  let $p\in \{C=1\}\bigcap\Omega_+$.
                  
                  If $\frac{\p R}{\p C}|_p\neq 0$, 
 then 
          there exists some small $\epsilon_0>0$
    such that
                $R^{-1}([q_0-\epsilon_0,q_0+\epsilon_0])\bigcap \Xi$  foliate
                 an open square $\Xi$ of $p$
                by connected regular curves $\{R^{-1}(q)\}_{q\in(q_0-\epsilon_0,q_0+\epsilon_0)}$
                and in fact it turns out that
                each curve is analytic by Theorem 2.35 (Real Analytic Implicit Function Theorem) in \cite{KP}. 
               %
               %
               There are two possibilities.
            The first is that
            there are infinitely many $q\in (q_0-\epsilon_0, q_0+\epsilon_0)\bigcap \mathbb Q$
            with $J_1^C(\tilde C)$ non-constant along   ${R}^{-1}(q)\bigcap \Xi$.
           Since $R(C, \tilde C)=CJ_2^C(\tilde C)/J_1^C(\tilde C)\equiv q$ in ${\tt R}^{-1}(q)$,
                            one can see that
                             $J_1^C(\tilde C)\in \pi\mathbb Q$ automatically implies $CJ_2^C(\tilde C)\in \pi\mathbb Q$ as well.
                             As a result,
                             by taking the role of  $q$ into account,
                             there are at least $\mathbb Q\times \mathbb Q$ many pairs of $(C, \tilde C)$
                              such that each induces  a doubly spiral minimal product
                             which factors through an immersion of certain quotient manifold of $S^1\times M_1\times M_2$.
                             The other is the opposite, 
                             namely only finitely many $q\in (q_0-\epsilon_0, q_0+\epsilon_0)\bigcap \mathbb Q$
                             for which $J_1^C(\tilde C)$ is non-constant along   ${R}^{-1}(q)\bigcap \Xi$.
                             However, this cannot happen.
                             Otherwise, by the density of rational numbers in $\R$,
                             one can see from the foliation by the curve segments 
                             that the function $J_1^C(\tilde C)$ must be constant along $\Gamma_1$.
This contradicts with \eqref{B101} and \eqref{B102} in Lemma \ref{B1}.
              
              Since $R$ is analytic to the boundary $\p\Omega\sim\{C=0\}$,
              let us focus on the point $p_0=\p\Omega\bigcap\{C=1\}$.
              Note that the $R$-value along the boundary is given by $C\cot^2 s_C$ and $\cot^2 s_1=\frac{k_1+1}{k_2+1}$.
              By differentiating \eqref{sCd} 
              we have 
              $$s'_1=\frac{\p s_C}{\p C}\Big|_{C=1}=\frac{\sqrt{(k_1+1)(k_2+1)}}{(k_1+k_2+2)^2}.$$
              The differentiation of $C\cot^2 s_C$ at $C=1$ is the following
              $$
              \cot^2 s_1-\frac{2\cot s_1}{\sin^2 s_1} s'_1
              =
              \frac{\cot s_1}{\sin^2 s_1} 
              \left(\cos s_1\sin s_1-2 s'_1\right)
              =
               \frac{\cot s_1}{\sin^2 s_1} 
               \frac{\sqrt{(k_1+1)(k_2+1)}}{(k_1+k_2+2)^2}
               (k_1+k_2).
              $$
              Hence, when $k_1+k_2>0$,
              the partial derivative  (horizontal portion) $\frac{\p R}{\p C}\neq 0$ in some neighborhood of $p_0$ in $\Omega_+$.
                    \footnote{In fact, away from $\{C=0\}$,  the nonvanishing property $\frac{\p R}{\p C}\neq 0$ holds for an open dense set along $\p \Omega_+$
                    when $k_1+k_2>0$.}
             Now by the argument in the preceding paragraph
               Theorem \ref{main0} stands true.
            \end{proof}

                              {\ }
                              
\section{More Comments and Remarks}\label{MCR}
 
          In this concluding section, we would like to mention something more about spiral minimal products.
          The first is the geometric distinctions among $\{G_\gamma\}$.
  \subsection{Geometric distinctions}
          For spherical $\mathscr C$-totally real (minimal) immersion $f_1$  of $M_1^{k_1}$ and  $f_2$ of $M_2^{k_2}$,
          there are induced metrics $g_1$ on $M_1$ and $g_2$ on $M_2$ respectively.
          Then a spiral minimal product $G_\gamma$ of them induces a metric $g_\gamma$ on $\R\times M_1\times M_2$ by
\begin{equation}\label{indmetric}
               g_\gamma
               =
             dt_{arc}^2 
               \oplus 
               a^2(t_{arc})\, g_1\oplus b^2(t_{arc})\, g_2 
\end{equation}
          where $t_{arc}$ is an arc parameter of $\gamma\subset \mathbb S^3$. 
         Note that  $dt_{arc}^2=(1+\Theta)\, {ds^2}$ of $\gamma^0$ over the basic domain $\Omega_{C, \tilde C}$
         and  $t_{arc}$ is convenient for the global description \eqref{indmetric}.
         
         \begin{defn}
         Two spherical immersions of manifold $M$ into $\mathbb S^n$ are called geometrically the same,
         if one can be written as  a composition of the other with an isometry of $\mathbb S^n$; 
         and are called geometrically distinct if otherwise.
         \end{defn}
         
         Note that, if immersions $G_{\gamma}$ and $G_{\gamma'}$ of $\R\times M_1\times M_2$ into the target sphere are geometrically the same, 
                      then their induced metrics are the same, i.e., $g_{\gamma}=g_{\gamma'}$
                             (independent of choice of parametrization). 
       When $k_1+k_2\geq 1$, 
         it is clear that spiral minimal products of steady magnitudes are geometrically distinct  from those of varying magnitudes
         and 
         that spiral minimal products of steady magnitudes constructed in \S \ref{SMP} are mutually  distinct.
         
         Now we emphasize the geometric distinctions among spiral minimal products $\{G_\gamma\}_{C,\tilde C}$ of varying magnitudes in our construction.
          
          \begin{thm}\label{Gdistinct}
          When $k_1+k_2\geq 1$, every $G_\gamma$ produced by $(C,\tilde C)\in \Omega_+$ is geometrically distinct from 
              $G_{\gamma'}$ by $(C', \tilde C') \in \Omega_+$ if $(C', \tilde C')\neq (C,\tilde C)$.
          \end{thm}
          
          \begin{proof}
          By Lemma \ref{Endpts} in Appendix,
          we know that $\Omega_{C,\tilde C}\neq \Omega_{C',\tilde C'}$ unless $(C, \tilde C)=(C', \tilde C')$.
          Hence, when $(C, \tilde C)\neq (C', \tilde C')$,
          the pairs of minimal and maximal values of $a^2$ for $G_\gamma$ and $G_{\gamma'}$ are different.
          The same holds for pairs of minimal and maximal values of  $b^2(t_{arc})$ for them as well.
          Due to the assumption $k_1+k_2\geq 1$, we know either $k_1\geq 1$ or $k_2\geq 1$.
          As a result, it can be seen by \eqref{indmetric} that $g_\gamma\neq g_{\gamma'}$.
          Therefore, the proof gets accomplished.
          \end{proof}
          
 \subsection{The case with \boldmath$k_1=k_2=0$.}     \label{k120}    
          One may wonder what happens when $k_1=k_2=0$.
          In this situation, the method of the proof of Theorem \ref{Gdistinct} fails.
          In fact, 
          all spiral minimal products $G_\gamma$ with $k_1=k_2=0$ and $M_1=\{point_1\}, M_2=\{point_2\}$  must run great circles in the target sphere
         and all $\{G_\gamma\}$ are geometrically the same.
         In particular, every solution curve $\gamma$  runs a great circle in $\mathbb S^3$.
          This can be proved for our constructions as follows.

          According to \eqref{tarc}$-$\eqref{tarc3} with respect to arc parameter $t_{arc}$ of $\gamma$, 
          one can check directly  that $\|G'_\gamma\|=1$
          and moreover, by some elementary calculations and relation \eqref{dotG}, that
          $$
                 G''_\gamma
                 =
                    \Big(
                         %
                           \Big[
                           i
                           (2a's'_1
                           +as''_1)
                           +a''
                           -a\cdot(s'_1)^2
                          \Big]
  {e^{is_1} f_1}
                           ,\,
                           \Big[
                           i
                           (2b's'_2
                           +bs''_2)
                           +b''-b\cdot(s'_2)^2
                           \Big]
 {e^{is_2} f_2}
                       \Big)
                 =
                 -G_\gamma\, .$$
          This means that $G_\gamma$ is a geodesic and hence runs part of a great circle of $\mathbb S^3$.

          When $C\neq 0$ and $k_1=k_2=0$, 
          preferred antiderivatives of  \eqref{dotsss} can be given  by
\begin{equation}\label{s1anti}
    s_1= -\frac{1}{2}\arctan\left[\frac{-2+(1+\tilde C- C^2)\cos^2 s}
    {2\sqrt{\tilde C\, \big(\cos s\big)^{2}
         \big( \sin s\big)^{2}-1-(C^2-1)\cos^2 s}}\right]
\end{equation}
     and
\begin{equation}\label{s2anti}
         \,
    s_2= \frac{1}{2}\arctan
    \left[
    \frac{-1+\tilde C- C^2+(1-\tilde C- C^2)\cos^2 s}
    {2C\sqrt{\tilde C\, \big(\cos s\big)^{2}
         \big( \sin s\big)^{2}-1-(C^2-1)\cos^2 s}}
    \right].
\end{equation}
By Lemma \ref{s12sign} 
we know that each numerator of the fractions inside arctan of \eqref{s1anti} and \eqref{s2anti} has opposite signs in $\p\Omega_{C,\tilde C}$.
      Note that $1+\tilde C- C^2>0$ and $1-\tilde C- C^2<0$ in the proof of Lemma \ref{s12sign}.
      The numerator of the fraction in \eqref{s1anti} is decreasing
      while that in \eqref{s2anti} is increasing.
         Hence, 
         it follows that
\begin{equation}\label{00s12}
\int_{\Omega_{C,\tilde C}} \dot s_1 ds=\frac{\pi}{2} \text{\ \ \ \ \ and\ \ \ \ } \int_{\Omega_{C,\tilde C}} \dot s_2 ds=\text{sign}(C)\cdot\frac{\pi}{2}\, .
\end{equation}
         Thus
         the total changes of $s_1$ {and} $s_2$ over 
         $\Omega^0_{C,\tilde C}\bigcup \Omega^1_{C,\tilde C}\bigcup \Omega^2_{C,\tilde C}\bigcup \Omega^3_{C,\tilde C}$ 
         are exactly $\pm 2\pi$ for every nonzero $C$.
         Namely, after running two rounds of the pendulum, $\gamma$ closes up as an entire great circle.
         Consequently, it is clear that projection of $\gamma$ on each $\C$-component of $\C\oplus\C$ draws an ellipse
           and
           the projection image of $\gamma$ over each $\Omega^\ell_{C,\tilde C}$
          where $\ell \in \{0,1,2,3\}$ performs a quarter of the ellipse.
           
           When $C=k_1=k_2=0$, 
           an antiderivative of $\dot{s}_1$
           can be chosen to be 
           $$
    s_1= -\frac{1}{2}\arctan\left[\frac{-2+(1+\tilde C)\cos^2 s}
    {2\sqrt{(\tilde C \cos^2 s-1)\sin^2s}}\right]
    \, \text{ with allowed } \tilde C>1.
     $$
                   Apparently, $\int_{\Omega_{0,\tilde C}} \dot s_1 ds=\frac{\pi}{2}$.
                   So, similarly after running two rounds of the pendulum, $\gamma$ closes up as an entire great circle in $\mathbb S^3$.
           However, only the projection of $\gamma$ on the first $\C$-copy gives an ellipse while projection to the other runs a closed interval  
           back and forth two rounds.

             {\ }
             
           We would like to remark that 
           $\int_{\Omega_{0,\tilde C}} \dot s_1 ds=\frac{\pi}{2}$
            coincides with 
            the asymptotic behaviors \eqref{B103} and \eqref{B104} for $C=0$, 
            and that
             \eqref{00s12}
             coincides with
             the asymptotic behaviors
            in
           Remark \eqref{FRkC}.

{\ }

 \subsection{Some comments about the construction}
 In this subsection, we mention some feature and some connection related to spiral minimal products, as well as a conjecture.
  
 {\ }
 
Both constructions in Theorem \ref{mainebd} and Corollary \ref{singlyemb} can be performed inductively and jointly,
a huge number of examples of embedded spherical minimal submanifolds of highly nontrivial topology and complicated magnitude distributions can be gained. 
 Due to  Corollary \ref{singlyemb}, there are numerous such examples of dimensions larger than half that of the ambient sphere.
 
 {\ }
 
Let {\boldmath$\Delta_{g}$} be the standard Laplace-Beltrami operator of Riemannian $(M, g)$.
       Then the following can be derived by the Takahashi Theorem \cite{T}.
       
       \begin{prop}
       An isometric immersion $f$ of $(M^k, g)$ into unit sphere $\mathbb S^n\subset \R^{n+1}$ is minimal 
       if and only if
        {\boldmath$\Delta_{g}$}$f=-kf$.
       \end{prop}
 
      In view of the proposition,
      isometric spherical minimal immersions $f_1$ and $f_2$ of $(M^{k_1}_1, g_1)$ and $(M^{k_2}_2, g_2)$
      mean that 
                {\boldmath$\Delta_{g_1}$}$f_1=-k_1f_1$
                and
                   {\boldmath$\Delta_{g_2}$}$f_2=-k_2f_2$.
      Then
      the construction of spiral minimal products
      is equivalent to building
      $G_\gamma$ 
      for
       {\boldmath$\Delta_{g_\gamma}$}$G_\gamma=-(k_1+k_2+1)G_\gamma$
       where $g_\gamma$ is of format  \eqref{indmetric}.
       However, this formulation bothers solving the following system of four ordinary differential equations 
       
        {\boldmath$\Delta_{g_\gamma}$}
        $G_\gamma=
       \big(  \frac{\gamma_1}
         {a^2}
         $
         {\boldmath$\Delta_{g_1}$}
         $f_1
        \, ,\,
         \frac{\gamma_2}
         {b^2}
         $
         {\boldmath$\Delta_{g_2}$}
         $f_2
         \big)
         $
         +
         $\big(\gamma''_1 f_1\, ,\, \gamma''_2 f_2\big)$
         +
         $
         \big(k_1\frac{a'}{a}+k_2\frac{b'}{b}\big)
         \big( \gamma'_1 f_1\, ,\, \gamma'_2 f_2\big)
         $
         \\
 \text{\ \ \ \ \ \ \, \,   \ \ \ \ \ \ }       =
       $\big( 
       \big[- \frac{k_1}
         {a^2}
         +
         \frac{\gamma''_1}{\gamma_1}
         +
         k_1\frac{a'\gamma_1'}{a\gamma_1}
         +
      {k_2\frac{b'\gamma_1'}{b\gamma_1}}
        \big]
        \gamma_1f_1
        \, ,\,
        \big[- \frac{k_2}
         {b^2}
         +
         \frac{\gamma''_2}{\gamma_2}
         +
        {k_1\frac{a'\gamma'_2}{a\gamma_2}}
         +
         k_2\frac{b'\gamma_2'}{b\gamma_2}
        \big]
        \gamma_2f_2
         \big)
         $
\begin{equation}\label{LapG}
         =-(k_1+k_2+1)G_\gamma \ \ \ \ \ \ \ \ \ \ \ \ \ \ \, \, \ \ \ \ \ \ \ \ \ \ \, \, \ \ \ \ \ \ \ \ \ \ \, \, \ \ \ \ \ \ \ \ \ \ \, \, \ \ \ \ \ \ \ \ \ \ \, \, \ \ \ \ \ \ \ \ \ \ \, \, \ \ \ \ \ \ \ \ \ \ \, \,
\end{equation}
where  derivative $'$ is taken with respect to an arc parameter $t_{arc}$ of $\gamma$.
{With \eqref{tarc}$-$\eqref{tarc3}, \eqref{dotG} and $\frac{ds}{d t_{arc}}=\frac{1}{\sqrt{1+\Theta}}$, 
one can check through some basic but delicate computations that all  spiral minimal products of varying magnitudes satisfy  \eqref{LapG}.
As for spiral minimal products of steady magnitudes which are limits of those with varying magnitudes, 
one can use either  \eqref{b/a} or  more easily Remark \ref{valueofc}  and Relation \eqref{sCd} for a direct verification that they also obey \eqref{LapG}.
Clearly,  the constructions and results in  our paper confirm that all solutions $G_\gamma(t, x, y)=\big(\gamma_1f_1(x), \gamma_2f_2(y)\big)$
by  immersed curves $\gamma=(\gamma_1,\gamma_2)\subset\mathbb S^3$ with $\gamma_1\gamma_2\neq 0$
to equation \eqref{LapG} are exactly the spiral minimal products of $f_1$ and $f_2$ with varying or steady magnitudes.}

                                    
{\ }                                    
                                          
                     We would like to end up the paper with the following conjecture,
                     which
                     we have verified for $C=-1$ with the help from calibrated geometry and 
                     disproved for $C=0$.
                     
                {\ }\\   
         {\bf Conjecture.}
      Assume that $M_1^{k_1}, M_2^{k_2}$ with $k_1,k_2\geq 1$ are
embedded  spherical minimal 
submanifolds.
Let $\mathcal Q_C=\left\{(C, \tilde C)\, \vert\, J_1^C(\tilde C)\in \pi\mathbb Q, CJ^C_2(\tilde C)\in \pi\mathbb Q, CJ^C_2(\tilde C)/J^C_1(\tilde C)\in \mathbb Q\right\}$.
Then, for a generic $C_0$ with $\#\mathcal Q_{C_0}=\infty$,
there are infinitely many pairs $\{(C_0, \tilde C_j)\}_{j\in \N}\subset \mathcal Q_{C_0}$
such that,
by virtue of \eqref{Cpx1} and \eqref{Cpx2},
each $(C_0, \tilde C_j)$ generates a spiral minimal product of $M_1$ and $M_2$
which factors through an embedding.
    
    {\ }  


{\ }

\section{Appendix for General $C$}
     In this section, we collect some useful properties with general $C$.
 \subsection{Unique critical point \boldmath$ s_C$.}
     \begin{lem}\label{AL1}
     For $C\neq 0$ or $k_2>C=0$,
     the function 
\begin{equation}\label{Cfun}
      P_C(s)=\dfrac{ 1+(C^2-1)\cos^2 s}{
                  \Big(\cos s\Big)^{2k_1+2}
                  \Big( \sin s\Big)^{2k_2+2}
                  }
\end{equation} 
         has exactly one critical point  $s_C\in (0, \frac{\pi}{2})$, where the function attains its minimum.
     \end{lem}
     \begin{proof}
     Set $t=\cos^2 s\in (0,1)$. 
     Then \eqref{Cfun} becomes
\begin{equation}\label{Cfun1}
      \dfrac{ 1+(C^2-1)t}{
                 \, t^{k_1+1}
                  ( 1-t)^{k_2+1}
                  }
\end{equation} 
     the derivative of which in $t$ is
\begin{equation}\label{Cfun2}
      \dfrac{ \big(C^2-1\big)
                  -
                  \left[1+(C^2-1)t\right]
                  \big(\frac{k_1+1}{t}-\frac{k_2+1}{1-t}\big)
                  }{
                 \, 
            {\tt T}
                  }
\end{equation} 
     where
     ${\tt T}=t^{k_1+1}
                  ( 1-t)^{k_2+1}>0$.
                  Note that the derivative of the numerator of \eqref{Cfun2}
                  is
\begin{eqnarray}       
        &&     -     \big(C^2-1\big) \left(\frac{k_1+1}{t}-\frac{k_2+1}{1-t}\right)
 -
                 \big[1+(C^2-1)t\big]
                 \left(-\frac{k_1+1}{t^2}
 -
                 \frac{k_2+1}{(1-t)^2}\right)
\nonumber\\ &= & \frac{k_1+1}{t^2}+C^2 \frac{k_2+1}{(1-t)^2}>0
\nonumber
\end{eqnarray} 
and that the numerator of \eqref{Cfun2} goes to negative infinity when $t\downarrow 0$
                 whereas it tends to positive infinity when $C\neq 0$ and $t\uparrow 1$ or tends to $k_2>0$ when $C=0$ and $t\uparrow 1$.
                 Hence, 
                 \eqref{Cfun2} 
                 has exactly one zero point $t_C\in (0,1)$
                 and
                 \eqref{Cfun1}
                  attains its minimum at $t_C$.
                 As $\frac{dt}{ds}=-\sin 2s\neq 0$ 
                 for $s\in(0,\frac{\pi}{2})$,
                 it follows that
                 $s_C=\arccos \sqrt {t_C}$ is the unique critical point of $P_C(s)$ in $(0,\frac{\pi}{2})$.
     \end{proof}
     \begin{rem}
     When $k_2=C=0$, in $\overline \Omega_{0, \tilde C}=\big[0,  \arccos{(\tilde C^{-\frac{1}{2k_1+2}})}\big]$,
     zero is the only minimal point of $P_0(s)$ with $P_0(0)=1$.
     \end{rem}
     
          \begin{rem}\label{valueofc}
     For the case of steady magnitudes with $a\equiv \cos s_C$ and $b\equiv \sin s_C$,
     the ratio $c=\frac{\dot s_2}{\dot s_1}$ in \S \ref{SMP} is exactly $C\cot^2 s_C$.
     \end{rem}

\subsection{Range of \boldmath$s_C$}     
     Based on the above lemma, the critical point is unique. 
    So
     we consider the derivative of $P_C(s)$ in $s$:
            $$
            \dfrac{\, -2(C^2-1)\cos s \sin s
            -
            \left[1+(C^2-1)\cos^2 s\right]
            \big(-(2k_1+2)\tan s+(2k_2+2)\cot s\big)}{T}
            $$
     where $T$ is the denominator of \eqref{Cfun}.
     As a result, 
           for a vanishing derivative,
        one needs
\begin{equation}   \label{sCd}
                 C^2-1
                 =\dfrac
                     {-(2k_1+2)\tan s_C+(2k_2+2)\cot s_C}
                     {\, 2k_1\tan s_C-(2k_2+2)\cot s_C  \, }
                     \dfrac{1}{\cos^2 s_C}
                     \, .
\end{equation}   
                 Therefore, we get the following property.
        \begin{lem}\label{AL2}
                 Assume that $k_1, k_2\geq 1$.
                 For the critical point $s_C$ of $P_C(s)$ in \eqref{Cfun},
                 we have 
                 $s_{C_1}\neq s_{C_2}$ unless $C_1=\pm C_2$,
                 and moreover, 
\begin{equation}\label{sC1}   
               \tan^2  s_C= \frac{k_2+1}{k_1+1}
               \text{\ \ when \ } C= \pm 1,
\end{equation}   
\begin{equation}   \label{sC2}   
            \   \tan^2  s_C\, \, \big\uparrow\, \, \frac{k_2+1}{k_1}
               \text{\ \ as \ } C\rightarrow \pm \infty,
\end{equation}   
                 and 
\begin{equation}   \label{sC3}   
               \tan^2  s_C\, \, \big \downarrow\, \, \frac{k_2}{k_1+1}
               \text{\ \ as \ } C\rightarrow 0. \
\end{equation}   
     \end{lem}
                              \begin{proof}
                              Since $s_C=s_{-C}$ for \eqref{Cfun},
                              we assume $C\geq 0$
                              and
                              by \eqref{sCd} we know that $C \mapsto s_C$ is an injective map by $\frac{d s_C}{d C}\neq 0$ when $C>0$.
                              It is clear that \eqref{sC1} follows from \eqref{sCd}
                              and that $C\rightarrow \pm \infty$ corresponds to vanishing denominator of the first factor on the right hand side of \eqref{sCd}.
                              The uparrow in \eqref{sC2} is due to the sign on the left hand side of \eqref{sCd}.
                              As for \eqref{sC3} with $k_2\geq 1$, one can see this from the fact that,
                              for any compact $K\Subset (0,\frac{\pi}{2})$,
                              the function \eqref{Cfun}
                              uniformly limits to
                              $$
                                   \dfrac{ 1}{
                  \Big(\cos s\Big)^{2k_1+2}
                  \Big( \sin s\Big)^{2k_2}
                  }
                              $$
                              on $K$
                              when $C\rightarrow 0$.
                              The downward arrow can be seen by the monotonicity of $C \mapsto s_C$ for $C\geq 0$ induced from its injectivity.
                              \end{proof}

 \subsection{Range of function \boldmath$R(C, \tilde C)$.}    
 
 {\ }
 
     Recall \ \ $J_1^C(\tilde C) 
     ={
                              \displaystyle         \mathlarger{\int}_{\Omega_{C, \tilde C}} 
                                     \dfrac{\tan s}{  \sqrt{  {\tilde C  
                 \Big(\cos s\Big)^{2k_1+2}
                  \Big( \sin s\Big)^{2k_2+2}
                  \,
                   -
                   {
                      1-(C^2-1)\cos^2 s
                   }
            }
            }
 }\, ds
}
$
$$\ \ \ J_2^C(\tilde C) 
     ={
                                          \displaystyle
                                          \mathlarger{\int}_{\Omega_{C, \tilde C}} 
                                    \dfrac{\cot s}{    \sqrt{ {\tilde C  
                 \Big(\cos s\Big)^{2k_1+2}
                  \Big( \sin s\Big)^{2k_2+2}
                  \,
                   -
                   {
                      1-(C^2-1)\cos^2 s
                   }
            }
            }
 }\, ds
} 
$$
and
set $R=R(C, \tilde C)={CJ_2^C(\tilde C)}\big/{J_1^C(\tilde C)}$
and
 $
           \Omega_+=
          \big\{
          (C, \tilde C)
        \,\vert\,
        C> 0,
           \tilde C>P_C(s_C)
               \big\}
          $.
     We obtain the following.
\begin{lem}\label{AL3}
Assume that $k_1,k_2\geq 1$ and $0<C<\infty$
     (i.e., to generate a doubly spiral minimal product).
     Then the value range of $R$ on $\Omega_+$ is $(0, \infty)$.
  \end{lem}
                                  \begin{proof}
                                         When $\tilde C\downarrow P_C(s_C)$,
                                          value of $R$
                                           accumulates 
                                           to
                                           $C\cot^2 s_C$.
                                     According to the analysis in 
                                           Lemma \ref{AL2},
                                         we know that 
                                           $\tan^2s\in \Big(\frac{k_2}{k_1+1}, \frac{k_2+1}{k_1}\Big)$
                                           with
                                                  $\left[\frac{k_2}{k_1+1}, \frac{k_2+1}{k_1}\right]\Subset \left(0,\infty\right)$.
                                        Hence, 
                                        the value range of  $R$ on $\Omega_+$%
                                                is $(0,+\infty)$.    
\end{proof}

    
\subsection{Basic domains without repetition in \boldmath$\Omega_+$.} 
   Our next lemma states that the basic domain $\Omega_{C, \tilde C}$ never appears repeatedly in $\overline{\,\Omega_+}$.
   \begin{lem}\label{Endpts}
   For any fixed $k_1,k_2\geq 0$,
   the basic domain $\Omega_{C, \tilde C}=\big(z^{C, \tilde C}_L, \,    z^{C, \tilde C}_R \big)\subset \big(0,\frac{\pi}{2}\big)$ 
   is different from those corresponding to others in $\overline{\,\Omega_+}\sim \p\Omega$.
   \end{lem}
                     
     \begin{proof}
     For simpler notations,
     let $(\alpha_1, \beta_1)=\Omega_{C_1, \tilde C_1}$ and $(\alpha_2, \beta_2)=\Omega_{C_2, \tilde C_2}$ 
     where $(C_1,\tilde C_1)$, $(C_2, \tilde C_2)\in \overline{\,\Omega_+}\sim \p\Omega$.
     If $a_2=a_1=0$, 
               then by Figure \ref{fig:2c} we have $k_2=0$, $C_2=C_1=0$, 
               $\cos^{2k_1+2}\beta_1=\tilde C_1^{-1}$
                                  and $\cos^{2k_1+2}\beta_2=\tilde C_2^{-1}$.
                                  Thus the statement simply follows.
                                  
                                  {\ }
                                  
     From now on, we suppose that $(\alpha_1, \beta_1)=(\alpha_2, \beta_2)$ and $\alpha_2=\alpha_1>0$.
     Set $a_1=\cos \alpha_1$, $b_1=\cos \beta_1$, $a_2=\cos \alpha_2$ and $b_2=\cos \beta_2$.
     On the one hand, by \eqref{4.1}, 
     for $\alpha_2=\alpha_1$,
     we must have 
     $$ \tilde C_2
       = \frac{1+(C_2^2-1)a^2_1}{a_1^{2k_1+2} (1-a_1^2)^{k_2+1}}\, ,
      $$
     and further, by $\tilde C_2 b_2^{2k_1+2} (1-b_2^2)^{k_2+1}=1+(C_2^2-1)b_2^2$ and $b_2=b_1$, 
           have
\begin{equation}\label{C1tilde} 
                   \tilde C_2
                         =\frac{1+(C_2^2-1)a^2_1}{a_1^{2k_1+2} (1-a_1^2)^{k_2+1}}
                           = \frac{1+(C_2^2-1)b^2_2}{b_2^{2k_1+2} (1-b_2^2)^{k_2+1}}
                            = \frac{1+(C_2^2-1)b^2_1}{b_1^{2k_1+2} (1-b_1^2)^{k_2+1}}
                         . 
\end{equation}
           
           On the other hand,  we consider 
\begin{equation}\label{Phi}
        \Phi(C)
           =
          \frac{1+(C^2-1)a^2_1}{a_1^{2k_1+2} (1-a_1^2)^{k_2+1}}
          \cdot  b_1^{2k_1+2} (1-b_1^2)^{k_2+1}-1-(C^2-1)b_1^2 ,
\end{equation}
          which vanishes at $C=C_2$ and leads to
\begin{equation}\label{Phider}
          \frac{1}{ b_1^{2k_1+2} (1-b_1^2)^{k_2+1}}
          \cdot
          \frac{d\Phi}{dC}
          =
                     \frac{2C a^2_1}{a_1^{2k_1+2} (1-a_1^2)^{k_2+1}}
                     -
                       \frac{2C b^2_1}{ b_1^{2k_1+2} (1-b_1^2)^{k_2+1}}.
\end{equation}
                     Note that \eqref{Phider} indicates that $  \frac{d\Phi}{dC}\neq 0$ unless $C=0$.
                     The reason is that, otherwise,
                     $$
               \frac{a^2_1}{a_1^{2k_1+2} (1-a_1^2)^{k_2+1}}
               =
               \frac{b^2_1}{ b_1^{2k_1+2} (1-b_1^2)^{k_2+1}} ,
                     $$
                     which  by \eqref{C1tilde} 
                     arrives at a contradiction that
                     $$
                     \frac{1+(C_2^2-1)a^2_1}{a_1^{2} }
                     =
                      \frac{1+(C_2^2-1)b^2_1}{b_1^{2} } 
                      \ \ \ \ \ \text{(impossible as } a_1<b_1).
                     $$
                     Therefore, we have shown that
                     $  \frac{d\Phi}{dC}\neq 0$ has a definite sign whenever $C>0$.
                   Namely, 
                     $\Phi$ is strictly increasing or strictly decreasing when $C>0$.
                     Since
                         $$ \tilde C_1
       = \frac{1+(C_1^2-1)a^2_1}{a_1^{2k_1+2} (1-a_1^2)^{k_2+1}}$$
       and 
       $$\tilde C_1 b_1^{2k_1+2} (1-b_1^2)^{k_2+1}-1-(C_1^2-1)b_1^2=0 ,$$
       we have $\Phi(C_1)=0$.
                     So, due to the relation $\Phi(C_2)=0=\Phi(C_1)$, 
                      either strict monotonicity of $\Phi$ (when $C>0$) would assert that
                     $b_2=b_1$ if and only if $C_2=C_1(\geq 0)$.
                     
                     With $C_2=C_1$, it follows that 
                     $(a_2, b_2)=(a_1, b_1)$ if and only if
                     $\tilde C_2=\tilde C_1$.
                     Hence, we finish the proof.
     \end{proof}

\subsection{A sign lemma for the case of \boldmath$k_1=k_2=0$.}
When $k_1=k_2=0$,
 signs of numerators of the fractions inside arctan in \eqref{s1anti} and \eqref{s2anti} at the ending points of $\Omega_{C,\tilde C}$ are 
             important to figure out changes of $s_1$ and $s_2$ over the basic domain $\Omega_{C,\tilde C}$.
The following sign lemma is simple and useful. 

     \begin{lem}\label{s12sign}
     When $C>0$,  each numerator of the fractions inside arctan of \eqref{s1anti} and \eqref{s2anti}
       has opposite signs  at the ending points of $\Omega_{C,\tilde C}$.
     \end{lem}   
     
     \begin{proof}
     Recall that $\Omega_{C,\tilde C}=\big\{s\in (0,\frac{\pi}{2})\, \big\vert\, \tilde C \cos^2s\, (1-\cos^2s)-1-(C^2-1)\cos^2s>0\big\}$.
     By considering
     $
     -\tilde C \cos^4s+(1+\tilde C- C^2)\cos^2 s-1=0
     $,
      at ending points of $\Omega_{C,\tilde C}$
     we have
     $2\tilde C \cos^2s=
               F\pm \sqrt{D}
                $
                where
               $F= 1+\tilde C- C^2>0$
               and
                 $D=(1+\tilde C- C^2)^2-4\tilde C>0$
                 according to the nonempty of $\Omega_{C,\tilde C}$.
             By
             $
               F\cos^2 s=1+\tilde C \cos^4s,
            $
            the numerator of \eqref{s1anti} at ending points of $\Omega_{C,\tilde C}$
            becomes
\begin{equation}\label{ends}
            \tilde C \cos^4s-1
            =\frac{1}{2\tilde C}
            \Big[
            D \pm 
            F\sqrt{D}
            \Big]
            =
            \frac{\sqrt{D}\, }{\,2\tilde C}
            \Big[
            \sqrt D \pm 
            F
            \Big].
\end{equation}
           Since $\tilde C>0$, it follows that $\sqrt D < F$.
           So, the numerator of \eqref{s1anti}, 
                which is \eqref{ends} at ending points of $\Omega_{C,\tilde C}$, 
                has opposite signs in $\p\Omega_{C,\tilde C}$.

Since the fraction inside arctan of \eqref{s2anti} 
can be rewritten as
           $$
    \frac{-2C^2+(-1+\tilde C+ C^2)\sin^2 s}{2C\sqrt{-C^2+(-1+\tilde C+ C^2)\sin^2 s- \tilde C \sin^4 s}},
     $$
the same argument results in the same conclusion on opposite signs of the numerator in $\p\Omega_{C,\tilde C}$.
In particular, $-1+\tilde C+C^2>0$.
     \end{proof}

             {\ }

\section{Appendix for $C=0$}
We shall take conventions and notations about singly spiral minimal products from \S \ref{P} 
               and 
                      discuss some quantitive properties of them.
   \begin{lem}\label{B1}
  With $s_0=s_{k_1,k_2}$ and $m_0=P_0(s_0)$ in \S \ref{P}, 
  the integral
     \begin{equation}\label{B1eq}
      I({\tilde C})= 
       \mathlarger{\int}_{\Omega_{0, \tilde C}}
       \,\,
       \dfrac{1}
       {\cos s
       \sqrt {\tilde C\, \Big(\cos s\Big)^{2k_1+2}
         \Big( \sin s\Big)^{2k_2}-1}
         }\, ds \,
\end{equation}
has the following limits, when $k_1\geq 0$ and $k_2>0$,
\begin{eqnarray}
                         \lim_{\tilde C \uparrow \infty}
                         I({\tilde C})
                         &=&
                         \ \ \frac{\pi}{2(k_1+1)},
                         \label{B101}
                         \\
                            \lim_{\tilde C \downarrow m_{0}}I({\tilde C})
                            &=&
                            \frac{\pi}{\sqrt{2(k_1+1)}}\, ;
                          \label{B102}
 \end{eqnarray}
            and, when $k_1\geq0$ and $k_2= 0$, with $\Omega_{0, \tilde C}=\big(0,  \arccos{(\tilde C^{-\frac{1}{2k_1+2}})}\big)$,
\begin{eqnarray}
                                  \lim_{\tilde C \uparrow \infty}\,\,\, I({\tilde C})
                                  &=&
                                  \ \frac{\pi}{2(k_1+1)},
                                  \label{B103}
                                  \\
                                 \lim_{\tilde C \downarrow 1}\,\,\, I({\tilde C})
                                 &=&
                                 \frac{\pi}{2\sqrt{(k_1+1)}}.
                                 \label{B104}
\end{eqnarray}
 \end{lem}
                 \begin{proof}
                 Note that, with the integrand denoted by $\digamma$ and the characteristic function by $\chi$,  
\begin{equation}\label{B01}
                 \lim_{\tilde C \uparrow \infty}
                 I({\tilde C})\,
                 =
                 \lim_{\tilde C \uparrow \infty} \mathlarger{\int}_{\text{around 0}} \chi_{{}_{\Omega_{-1,\tilde C}}}\, \digamma\, ds
                  +
                 \lim_{\tilde C \uparrow \infty} \mathlarger{\int}_{\text{around }\frac{\pi}{2}} \chi_{{}_{\Omega_{-1,\tilde C}}}\, \digamma\, ds
                  +
                  \lim_{\tilde C \uparrow \infty} \mathlarger{\int}_{\text{ middle part}} 
                  ,
\end{equation}
                  and the last term disappears in the limit.
                  Here $\chi_{\Omega_{-1,\tilde C}} F$ means the trivial extension of $F$ 
                       to domain $(0,\frac{\pi}{2})$ by assigning value zero outside $\Omega_{-1,\tilde C}$.
                  Set 
                  $T=\left(\cos s\right)^{2k_1+2}
         \left( \sin s\right)^{2k_2}$.
                  Around left ending point $z^{0, \tilde C}_L$,
                  the slope of $\tilde C T-1$
                  is given by
                  $K=-(2k_1+2)\tan s+2k_2\cot s$ which blows up when $k_2>0$ and $\tilde C\uparrow \infty$.
                  Note that, around $z^{0, \tilde C}_L$, 
                  the integral 
                      over $\big(z^{0, \tilde C}_L, z^{0, \tilde C}_L+\delta s\big)$
                       is approximately $\frac{2}{\cos z^{0, \tilde C}_L}\sqrt{\delta s/K}$ (plus higher order terms)
                       for positive infinitesimal amount $\delta s$ of variation of $s$.
                  Hence the first term of \eqref{B01} also vanishes in the limit.
                  Whereas, the second term around $\frac{\pi}{2}$
                  limits to the following
\begin{eqnarray}
  & &    \lim_{\tilde C \uparrow \infty} \mathlarger{\int}_{\text{around }\frac{\pi}{2}}\,\,
                   \dfrac{\chi_{{}_{\tilde C\, (\cos s)^{2k_1+2}
         >1}}}
       {\cos s
       \sqrt {\tilde C\, \Big(\cos s\Big)^{2k_1+2}
         -1}
         }\, ds
     \nonumber      \\      & =&
           \lim_{\tilde C \uparrow \infty}   \mathlarger{\int}_{\text{around }0}\,\,
                   \dfrac{\chi_{{}_{\tilde C\, (\sin s)^{2k_1+2}
         >1}}}
       {\sin s
       \sqrt {\tilde C\, \Big(\sin s\Big)^{2k_1+2}
         -1}
         }\, ds
   \nonumber     
\end{eqnarray}
which shares the same limit value 
\begin{eqnarray}
                   & &
                   \lim_{\tilde C \uparrow \infty}
                    \mathlarger{\int}_{\text{around $s=0$}}\,\,
                   \dfrac{\chi_{{}_{\Delta>0}}\,d\sqrt{\Delta}}
       {(k_1+1)(\Delta+1)
         }
            \nonumber \\
         &=&
         \frac{1}{k_1+1}\arctan \sqrt{\Delta}\big|_{0}^{\infty}
            \nonumber \\
          &=& \frac{\pi}{2(k_1+1)} 
             \nonumber 
\end{eqnarray}
where $\Delta=\tilde C (\sin s)^{2k_1+2}-1$.

{\ }

      When $k_2\geq 1$, 
                       about the unique interior critical point $s_0=s_{k_1,k_2}$ of $P_0(s)$ mentioned in \S \ref{P} and proven in Lemma \ref{AL1},
                       we have
                       $\dot{T}=T\cdot \big(-(2k_1+2)\tan s+2k_2\cot s\big)=0$ at $s_{0}$,
                       namely
                       $$
                       \tan^2 s_{0}=\frac{k_2}{k_1+1}>0,
                       $$
                       and moreover,
                         $$
                       \ddot T =
                                      T\cdot \Big(-(2k_1+2)\tan s+2k_2\cot s\Big)^2 
                                         +T\cdot \left(-\frac{2k_1+2}{\cos^2 s}-\frac{2k_2}{\sin^2 s}\right)
                       $$
                       which, at $s_{0}$, has value
\begin{equation}\label{B105}
                       \ddot T\left(s_{0}\right) =
                                        -4\, T\left(s_{0}\right)\cdot \big({k_1+k_2+1}\big).
\end{equation}
                             
                       As $\tilde C\downarrow m_{0}$,
                       it follows 
                       $\tilde CT\left(s_{0}\right)\downarrow 1$
                       and
                       the part under square symbol in \eqref{B1eq}
                       is 
\begin{equation}\label{B02}
                       \text{some small positive number }
                       - 2\big({k_1+k_2+1}\big) (\delta s)^2 + o(\delta s)^2
\end{equation}                       
                  where $\delta s$ is quite small.
                  Hence, by \eqref{B02},
                  the limit of \eqref{B102} reads
                  $$
                  \frac{\pi}{\cos s_{0}\sqrt{2(k_1+k_2+1)}}
                  =
                   \frac{\pi}{\sqrt{2(k_1+1)}} .
                  $$
                  
                  When $k_2=0$,
                  the domain $\Omega_{0,\tilde C}=\big(0, \arccos{(\tilde C^{-\frac{1}{2k_1+2}})}\big)$
                  and one misses half of the integral around $s=0$
                    and $\ddot T$ loses half contribution compared with \eqref{B105}.
                    So, 
                            $$\lim_{\tilde C \downarrow 1}I({\tilde C})=\frac{\pi}{2\sqrt{(k_1+1)}} .$$
                      On the other hand, when $k_2=0$, 
                                 the part ${ \displaystyle\int}_{\text{around 0}} $ in \eqref{B01} 
                                 for $\tilde C \uparrow \infty$ also goes away
                                 as the integrand now is approximately $\frac{1}{\sqrt{\tilde C-1}}$ around $s=0$.
                                 So, \eqref{B103} follows similarly by the reasoning of  \eqref{B101}.
                 \end{proof}
       %
       %
       %
       %
       %
       %
        
     We get the following corollary as we wished for the case with $C=0$.
                          \begin{cor}\label{BC1}
                                    When $k_1+k_2\geq 1$,
                                    there are infinitely many  $\tilde C$
                                    for $I({\tilde C})\in \pi\mathbb Q$.
                          \end{cor}

{\ }

\section{Appendix for $C=-1$}
    Let us study asymptotic behaviors of $J_1(\tilde C)$ for doubly spiral minimal products with $C=-1$.
     As in \eqref{sC1}, now at the unique critical point $s_{-1}$ we have $\tan ^2 s_{-1}=\sqrt{\frac{k_2+1}{k_1+1}}$.
     Denote $\frac{1}
                   {\big(\cos s_{-1}\big)^{2k_1+2}
                            \big( \sin s_{-1}\big)^{2k_2+2}}$ by ${\tt m}_0$. 
    The counterpart of Lemma \ref{B1} is the following.

     \begin{lem}\label{C1}
  The integral
     \begin{equation}\label{C1eq}
       J_1({\tilde C})= 
       \mathlarger{\int}_{\Omega_{-1, \tilde C}}
       \,\,
       \dfrac{\tan s}
       {
       \sqrt {\tilde C\, \Big(\cos s\Big)^{2k_1+2}
         \Big( \sin s\Big)^{2k_2+2}-1}
         }\, ds \,
\end{equation}
has the following limits
\begin{eqnarray}
                         \lim_{\tilde C \uparrow \infty} J_1({\tilde C})\,
                         &=&
                      \ \ \ \ \ \ \ \ \ \ \ \   \frac{\pi}{2(k_1+1)},
                         \label{C101}
                         \\
                            \lim_{\tilde C \downarrow {\tt m}_{0}}
                             J_1({\tilde C})
                             &=&
                             \frac{\sqrt{k_2+1}\,\pi}{\sqrt{2(k_1+1)(k_1+k_2+2)}}.
                             \label{C102}
 \end{eqnarray}
Consequently, combined with \eqref{2Delta} it follows that
\begin{eqnarray}
                         \lim_{\tilde C \uparrow \infty} J_2({\tilde C})\,
                         &=&
                       \ \ \ \ \ \ \ \ \ \ \ \     \frac{\pi}{2(k_2+1)},
\label{C103}\\
                            \lim_{\tilde C \downarrow {\tt m}_{0}}
    J_2({\tilde C})
    &=&
    \frac{\sqrt{k_1+1}\,\pi}{\sqrt{2(k_2+1)(k_1+k_2+2)}}.
    \label{C104}
\end{eqnarray}
\end{lem}
    \begin{proof}
    Firstly, \eqref{C101} follows 
    similarly as that in Lemma \ref{B1}.
    Since $s_{-1}$ lies in the interior of $(0,\frac{\pi}{2})$,
    the limit  expressions are uniform for all nonnegative  $k_1$ and nonnegative $k_2$.
    Comparing with \eqref{B102}, the limit \eqref{C102} comes from $\frac{\pi}{\sqrt{2(k_1+k_2+2)}}$ multiplied by $\tan s_{-1}$ not $\sec s_{-1}$.
    \end{proof}
    
\begin{rem}\label{RkC}
    Note that $k_1=k_2=0$ implies that both 
    $\lim_{\tilde C \uparrow \infty} J_1({\tilde C})$
    and
    $\lim_{\tilde C \downarrow {\tt m}_{0}}
                             J_1({\tilde C})$
                            are $\frac{\pi}{2}$.
                            Moreover, whenever $k_1+k_2\geq 1$, 
                            it simply follows that quantities in \eqref{C101} and \eqref{C102} do not equal
                            and hence $J_1$ is not constant.
\end{rem}

A final remark is that from the proof of Lemma \ref{B1} one can actually get more than   \eqref{B101}, \eqref{C101}, \eqref{B103} and \eqref{C103}.

            \begin{rem}\label{FRkC}
  For every $C$, zero or nonzero, one always has
    \begin{equation}\label{D101}
     \lim_{\tilde C \uparrow \infty} J_1^C({\tilde C})\, \, \, \,
                        =\, \, \, \, \, \,  \,  \, \, \, \frac{\pi}{2(k_1+1)}\, \, \, \, \, \, \,
      \end{equation}
                      and
   \begin{equation}\label{D102}
                 \lim_{\tilde C \uparrow \infty} J_2^C({\tilde C})\, \, \, \,
                        =\, \, \, \, \frac{1}{|C|}\cdot \frac{\pi}{2(k_2+1)}\, .
         \end{equation}

  \end{rem}
  
  {\ }
  
  {\ }
  
   \section*{Acknowledgement}
   H. Li was partially supported by NSFC (Grant No. 11831005). 
Y. Zhang was sponsored in part by NSFC (Grant Nos. 12022109 and 11971352).
The authors would like to thank Prof. Hui Ma for Example 2 and useful references, and Prof. Claudio Arezzo for helpful discussion.
   The second author also wishes to thank the Department of Mathematical Sciences of Tsinghua University 
   and the Associate Scheme of ICTP for providing wonderful research atmospheres and warm hospitalities.
\appendix

{\ }


\begin{bibdiv}
\begin{biblist}
\bib{Allard}{article}{
    author={Allard, {Williams K.}},
    title={On the first variation of a varifold},
    journal={Ann. of Math.},
    volume={95},
    date={1972},
    pages={417--491},
}


\bib{Blair}{book}
{
    author={Blair, {David E.}},
    title={Riemannian Geometry of Contact and Symplectic
Manifolds},
 place={Progress  in  Mathematics {\bf 203}, Birkh\"auser Boston},
   date={2002},
   }



\bib{Brendle}{article}{
    author={Brendle, Simon},
    title={Minimal surfaces in $\mathbb S^3$: a survey of recent results},
    journal={Bull. Math. Sci.},
    volume={3},
    date={2013},
    pages={133--171},
}



\bib{CM}{article}{
    author={Carberry, Emma},
    author={McIntosh, Ian},
    title={Minimal Lagrangian 2-tori in $\mathbb CP^2$ come in real families of every dimension},
    journal={J. Lond. Math. Soc.},
    volume={69},
    date={2004},
    pages={531--544},
}


\bib{CLU}{article}{
    author={Castro, Ildefonso}
    author={Li, Haizhong}
    author={Urbano, Francisco},
    title={Hamiltonian minimal Lagrangian submanifolds in complex space form},
    journal={Pacific Journal of Mathematics},
    volume={227},
    date={2006},
    pages={43--65},
}



\bib{CDVV}{article}{
    author={Chen, Bang-Yen},
    author={Dillen, Frank},
        author={Verstraelen, Leopold},
            author={Vrancken, Luc},
    title={Totally real submanifolds of $\mathbb CP^n$ satisfying a basic equality},
    journal={Arch. Math.},
    volume={63},
    date={1994},
    pages={553--564},
}



\bib{CH}{article}{
    author={Choe, Jaigyoung},
    author={Hoppe, Jens},
    title={Some minimal submanifolds generalizing the Clifford torus},
    journal={Math. Nach.},
    volume={291},
    date={2018},
    pages={2536--2542},
}


\bib{doCarmo}{book}
{
    author={do Carmo, {Manfredo P.}},
    title={Differential Geometry of Curves and Surfaces, Second Edition},
 place={Dover Publications, NY},
   date={2016},
   }

\bib{DKM}{article}{
    author={Dorfmeister, {Josef F.}},
    author={Kobayashi, Shimpei},
    author={Ma, Hui},
    title={Ruh–Vilms theorems for minimal surfaces without complex points and minimal Lagrangian surfaces in $\mathbb CP^2$},
    journal={Math. Z.},
    volume={296},
    date={2020},
    pages={1751--1775},
}


\bib{DM}{article}{
    author={Dorfmeister, {Josef F.}},
    author={Ma, Hui},
    title={Minimal Lagrangian surfaces in $\mathbb CP^2$
 via the loop group method Part I: The contractible case},
    journal={J. Geom. Phys.},
    volume={161},
    date={2021},
    pages={Paper No. 104016},
}



\bib{HL0}{article}{
    author={Harvey, F. Reese},
    author={{Lawson, Jr.}, H. Blaine},
    title={Extending Minimal Varieties},
    journal={Invent. Math.},
    volume={28},
    date={1975},
    pages={209--226},
}




\bib{HL2}{article}{
    author={Harvey, F. Reese},
    author={{Lawson, Jr.}, H. Blaine},
    title={Calibrated geometries},
    journal={Acta Math.},
    volume={148},
    date={1982},
    pages={47--157},
}


\bib{HL1}{article}{
    author={Harvey, F. Reese},
    author={{Lawson, Jr.}, H. Blaine},
    title={Calibrated Foliations},
    journal={Amer. J. Math.},
    volume={104},
    date={1982},
    pages={607--633},
}



\bib{Haskins}{article}{
    author={Haskins, Mark},
    title={Special Lagrangian cones},
    journal={Amer. J. Math.},
    volume={126},
    date={2004},
    pages={845--871},
}


\bib{HK0}{article}{
    author={Haskins, Mark},
    author={Kapouleas, Nicolaos},
    title={Special Lagrangian cones with higher genus links},
    journal={Invent. Math.},
    volume={167},
    date={2007},
    pages={223-294},
}




\bib{HK1}{article}{
    author={Haskins, Mark},
    author={Kapouleas, Nicolaos},
    title={Closed twisted products and $SO(p)\times SO(q)$-invariant special Lagrangian cones},
    journal={Comm. Anal. Geom.},
    volume={20},
    date={2012},
    pages={95-162},
}



\bib{HMU}{article}{
    author={Hijazi, Oussama},
     author={Montiel, 	Sebasti\'an},
    author={Urbano, Francisco},
    title={Spin$^c$ geometry of K\"ahler manifolds and the Hodge Laplacian on minimal Lagrangian submanifolds},
    journal={Math. Z.},
    volume={253},
    date={2006},
    pages={821-853},
}


\bib{JCX}{article}{
    author={Jiao, Xiaoxiang},
    author={Cui, Hongbin},
    author={Xin, Jialin},
    title={Area-minimizing cones over products of Grassmannianmanifolds},
    journal={Calc. Var. PDE},
    volume={61},
    date={2022},
    pages={205},
}

\bib{JXC}{article}{
    author={Jiao, Xiaoxiang},
        author={Xin, Jialin},
    author={Cui, Hongbin},
%
    title={Area-Minimizing Cones over Stiefel Manifolds},
    journal={to appear},
}






\bib{Joyce}{book}{
    author={Joyce, Dominic D.},
    title={Special Lagrangian 3-folds and integrable systems},
    place={pp. 189--233 
    in 
    ``Surveys on Geometry and Integrable Systems", 
    Advanced Studies in Pure Math. {\bf 51},
    Mathematical Society of Japan},
   date={2008},

}

\bib{KP}{book}{
 author={Krantz, Steven G.},
 author={Parks, Harold R.},
 title={
A Primer of Real Analytic Functions, Second Edition}
 place={ Birkh\"auser Advanced Textbooks, Birkh\"auser Boston},
   date={2002},
}


\bib{L}{article}{
    author={{Lawson, Jr.}, H. Blaine},
    title={Complete minimal surfaces in $\mathbb S^3$},
    journal={Ann. of Math.},
    volume={92},
    date={1970},
    pages={335--374},
}


\bib{Mc}{article}{
    author={McIntosh, Ian},
    title={Special Lagrangian cones in $\mathbb C^3$ and primitive harmonic maps},
    journal={J. Lond. Math. Soc.},
    volume={67},
    date={2003},
    pages={769-789},
}


\bib{M1}{book}{
    author={{Morrey, Jr.}, Charles B.},
    title={Second-order elliptic systems of differential equations},
    place={pp. 101--159 
    in 
    ``Contributions to the theory of partial differential equations", 
    Annals of Mathematics Studies {\bf 33},
    Princeton University Press},
   date={1954},

}

\bib{M2}{article}{
    author={{Morrey, Jr.}, Charles B.},
    title={On the analyticity of the solutions of analytic non-linear
              elliptic systems of partial differential equations. {I}.
              {A}nalyticity in the interior},
    journal={Amer. J. Math.},
    volume={80},
    date={1958},
    pages={198--218},
}


\bib{Rec}{book}{
    author={Reckziegel, Helmut},
    title={Horizontal lifts of isometric immersions into the bundle space of a pseudo-Riemannian submersion},
    place={pp. 264--279
    in 
    ``Global Differential Geometry and Global Analysis 1984", Lecture Notes in Mathematics {\bf 1156},
Springer-Verlag},
   date={1985},

}


\bib{T}{article}{
    author={Takahashi, Tsunero},
    title={Minimal immersions of Riemannian manifolds},
    journal={J. Math. Soc. Japan},
    volume={18},
    date={1966},
    pages={380--385},
}



\bib{TZ}{article}{
    author={Tang, Zizhou},
    author={Zhang, Yongsheng},
    title={Minimizing cones associated with isoparametric foliations},
    journal={J. Diff. Geom.},
    volume={115},
    date={2020},
    pages={367--393},
}


\bib{Top}{book}
{
    author={Toponogov, {Victor A.}},
    title={Differential Geometry of Curves and Surfaces, A Concise Guide},
 place={Birkh\"auser Boston, MA},
   date={2006},
   }


\bib{WW}{article}{
    author={Wang, Changping},
    author={Wang, Peng},
    title={The Morse index of minimal products of minimal submanifolds in spheres},
    journal={Sci. China Math.},
    volume={66},
    date={2023},
    pages={799--818},
}



\bib{X}{book}{
    author={Xin, Yuanlong},
    title={Minimal submanifolds and related topics},
    place={Nankai Tracts in Mathematics, World Scientific Publishing},
   date={2003 (and Second Edition in 2018)},

}


\bib{XYZ2}{article}{
author={Xu, Xiaowei}
author={Yang, Ling}
   author={Zhang, Yongsheng},
   title={New area-minimizing Lawson-Osserman cones},
    journal={Adv. Math.},
   Volume={330},
   date={2018},
    pages={739--762},
   }


\bib{x-y-z0}{article}{
    author={Xu, Xiaowei},
    author={Yang, Ling},
        author={Zhang, Yongsheng},
    title={Dirichlet boundary values on Euclidean balls with infinitely many solutions for the minimal surface system},
    journal={J. Math. Pur. Appl.},
    volume={129},
    date={2019},
    pages={266--300},
}





\bib{z}{article}{
        author={Zhang, Yongsheng},
    title={On realization of tangent cones of homologically area-minimizing compact singular submanifolds},
    journal={J. Diff. Geom.},
    volume={109},
    date={2018},
    pages={177--188},
}



{\ }


{\ }

\end{biblist}
\end{bibdiv}
\end{document}